\documentclass[11pt]{article}
\usepackage[colorlinks=true, linkcolor=blue, citecolor=red]{hyperref}
\usepackage{mathrsfs}
\usepackage{amsfonts}
\usepackage{amssymb}
\usepackage{amsmath}
\usepackage{amsthm}
\usepackage{bm}
\usepackage{cite}

\numberwithin{equation}{section}
\theoremstyle{plain}
\newtheorem{theorem}{Theorem}[section]

\newtheorem{lemma}{Lemma}[section]
\newtheorem{prop}{Proposition}[section]
\newtheorem{remark}{Remark}[section]
\newtheorem{definition}{Definition}[section]

\def\geq{\geqslant}
\def\leq{\leqslant}
\def\no{\nonumber}

\def\bt{\begin{theorem}}
\def\et{\end{theorem}}
\def\bl{\begin{lemma}}
\def\el{\end{lemma}}
\def\br{\begin{remark}}
\def\er{\end{remark}}

\allowdisplaybreaks
\everymath{\displaystyle}

\begin{document}

\title{
       {\Large
          \bf{Well-posedness and large deviations of fractional McKean-Vlasov stochastic reaction-diffusion equations on unbounded domains
              }
       }
      }

\author{Zhang Chen$^{\ast}$, \  
        Bixiang Wang$^{\dagger}$ 
\\
\normalsize{$^\ast$School of Mathematics, Shandong University, Jinan 250100,  China} \\
\normalsize{$^\dagger$Department of Mathematics, New Mexico Institute of Mining and Technology,} \\
\normalsize{Socorro,  NM~87801, USA}\\
\normalsize{E-mail: zchen@sdu.edu.cn,\  bwang@nmt.edu} }

\date{}
 \maketitle
\begin{minipage}{14cm}
{\bf Abstract.}  This paper is mainly concerned with the large
deviation principle of the fractional McKean-Vlasov stochastic
reaction-diffusion equation defined on $\mathbb{R}^n$ with
polynomial drift of any degree. We first prove the well-posedness of
the underlying equation under a dissipative condition, and then show
the strong convergence of solutions of the corresponding controlled
equation with respect to the weak topology of controls, by employing
the idea of uniform tail-ends estimates of solutions in order to
circumvent the non-compactness of Sobolev embeddings on unbounded
domains. We finally establish the large deviation principle of the
fractional McKean-Vlasov equation by the weak convergence method
without assuming the time H\"{o}lder continuity of the
non-autonomous diffusion coefficients.

\vspace{2mm}


 \vspace{2mm}


\end{minipage}


\maketitle

\section{Introduction}
 This paper
 deals with
 the well-posedness and
 the  large deviation
 principle (LDP)   of
 the fractional McKean-Vlasov stochastic reaction-diffusion equation
 driven by nonlinear noise defined on $\mathbb{R}^n$:
\begin{align} \label{mvsde-1}
    d u^\varepsilon(t) + (-\triangle)^\alpha u^\varepsilon(t) dt
           & + f(t, x, u^\varepsilon (t), \mathcal{L}_{u^\varepsilon (t)} ) dt
             = g(t,x,u^\varepsilon (t), \mathcal {L}_{u^\varepsilon (t)} )dt
               \nonumber \\
           & + \sqrt{\varepsilon}
                  \sum_{k=1}^{\infty}
                          \left( \sigma_{1,k}(t,x) + \kappa(x) \sigma_{2,k}(t, u^\varepsilon (t), \mathcal {L}_{u^\varepsilon (t)})
                          \right) dW_k(t),
\end{align}
 with initial data
\begin{align}\label{mvsde-1 IV}
   u^\varepsilon(0) = u_0,
\end{align}
 where
 $(-\triangle)^\alpha$ with $\alpha \in (0,1)$ is the fractional Laplace operator,
 $\varepsilon \in (0,1)$,
 $\mathcal{L}_{u^\varepsilon (t)}$
 is the distribution  law of   $u^\varepsilon (t)$,
 $f$ is a nonlinear  function with arbitrary growth rate,
 $g$ is a Lipschitz function,
 $\sigma_{1,k} : \mathbb{R} \to L^2(\mathbb{R}^n)$
 is given,
 $\kappa \in L^2(\mathbb{R}^n) \bigcap L^\infty(\mathbb{R}^n)$,
and
 $\sigma_{2,k}$ is  a nonlinear diffusion term,
 and
 $\{W_k\}_{k\in \mathbb{N}}$ is a sequence of independent
 standard  real-valued Wiener processes
 on a complete filtered probability space $\left( \Omega, \mathcal{F}, \{\mathcal{F}_t\}_{t \geq 0}, \mathbb{P} \right)$.

The fractional
 differential  equations
 have numerous  applications in finance, physics, biology and probability
   \cite{at2005PA, gm2006PTRF, jara2009CPAM}.
If  the nonlinear drift and diffusion terms
do not depend on
 $\mathcal{L}_{u^\varepsilon (t)}$, then
  \eqref{mvsde-1}
  reduces to the  standard stochastic reaction-diffusion equation with fractional Laplacian.
  The asymptotic behaviors of solutions
  of the standard
  fractional   equations
  have been extensively studied in the
  literature, see, e.g.,
   \cite{pg2013AA, lblz2015JDE, gw2016DCDS}
   for deterministic equations, and
  \cite{gz2014NA, wangb2017NA, glwy2018JDE, bm2019SIMA, wangb2019JDE, cw2021Nonl, cw2022JDE, hzsg2022JDE}
    for stochastic equations.
  We also refer the reader to \cite{dkz2019AP} and \cite{salins2022SD}
  for global solutions of stochastic reaction-diffusion equations with super-linear drift on bounded domain and unbounded domain, respectively.

 Recently,
 the  McKean-Vlasov stochastic differential equations
 have received considerable attention
 after the seminal  works \cite{kac1958, mckean1966PNAC}.
 Such stochastic equations are also called mean field equations \cite{ll2007JM, blpr2017AP}
 and  describe the limits of interacting particle systems
 \cite{S1991LNM, dv1995NoDEA, fg2015PTRF, bh2022JTP, hl2023arXiv},
  which   depend  not only on the states of solutions,
 but also on the distribution law of solutions.
 The existence of  solutions
 of the McKean-Vlasov equations
  have been intensively studied
  by many authors,
 see, e.g.
      \cite{ad1995spa, wang2018SPA, rz2021Bernoul, hds2021AP, fhsy2022SPA, ghl2022SD} and
      the references therein.
      The reader is also referred to
 \cite{b2018AIHP, rw2019JDE} for
 the Bismut formula and
   \cite{rsx2021AIHP, lwx2022arXiv} for
   the averaging principles of  such  systems.
In this paper, we investigate the
LDP of the
   fractional McKean-Vlasov    equation
\eqref{mvsde-1} on ${\mathbb{R} }^n$.

 The LDP is concerned with the exponential rate of convergence of solutions of stochastic equations with small noise,
 which has applications in many fields such as statistics, information theory and statistical physics \cite{dz1998, fw2012}.
 The classical method for proving the LDP is based on the time discretization and exponential tightness,
 see, e.g.,
           \cite{dz1998, fw2012, s1992ap, chowCPAM1992, kxAP1996, crAP2004}.
 To avoid the complicated time discretization,
 the weak convergence method for the LDP was developed in
  \cite{de1997, bd2000PMS, bdm2008AP}
 based on variational representations of positive functionals of Brownian motion,
 which is effective for proving the LDP of a large class of stochastic equations,
 see, e.g.
    \cite{ss2006SPA, dm2009SPA, cm2010AMO, liu2010AMO, bpm2012AP, bgj2017ARMA,
          dzz2017JDE, zz2017AMO, cd2019AIHP, msz2021AMO, hll2021SIMA, csw2024PRSEA}
 for equations driven by Brownian motion
 and \cite{by2015Stoch, lszz2022PA}
 for equations driven by jump processes.
 In particular, in the case of unbounded domains,
 the LDP of the fractional reaction-diffusion equation
 and the tamed 3D Navier-Stokes equations
 was examined in \cite{wangb2023JDE} and \cite{rzz2010AMO}, respectively.

 Currently, there are very few  results on the LDP of  the McKean-Vlasov stochastic equations,
 see \cite{lszz2022PA, hip2008AAP, dst2019AAP,hll2021AMO, hhl2022arXiv}.
 It is worth mentioning that in \cite{hll2021AMO, hhl2022arXiv},
 the authors established the LDP for a class of McKean-Vlasov stochastic equations
 with locally monotone nonlinearity by the weak convergence method.
 The results of both papers apply to many stochastic equations defined on bounded as well as unbounded domains,
 including the reaction-diffusion equation and the Navier-Stokes equations.

 However, the arguments of \cite{hll2021AMO, hhl2022arXiv}
 have restrictions on the growth order of the nonlinear terms
 and require the  non-autonomous diffusion terms to be regular in time.
 Specifically, for the reaction-diffusion equation \eqref{mvsde-1},
 the assumptions in \cite{hll2021AMO, hhl2022arXiv}
 impose an upper bound on the growth order of $f$ in its third argument and require $\sigma_{1,k}$
 and $\sigma_{2,k}$ to be H\"{o}lder continuous in time.
 Consequently, the results of \cite{hll2021AMO, hhl2022arXiv}
 do not apply to \eqref{mvsde-1}
 when $f$ has a polynomial  growth of
 arbitrary order in its third argument
 or when $\sigma_{1,k}$
 or $\sigma_{2,k}$ is not H\"{o}lder continuous in time.
 In the present paper,
 we deal with this case and prove the LDP of \eqref{mvsde-1}
 when $f$ has an arbitrary  growth order,
 $\sigma_{1,k}$ and $\sigma_{2,k}$  are continuous in time,
 but not H\"{o}lder continuous.

 To achieve our goal,
 we need some kind of compactness related to the tightness of probability distributions of a family of solutions.
 In the case of bounded domains,
 such compactness can be obtained from
 the regularity of solutions and the compactness of Sobolev embeddings.
 However,
 the McKean-Vlasov stochastic equation \eqref{mvsde-1}
 is defined on the unbounded domain $\mathbb{R}^n$
 and the Sobolev embeddings on $\mathbb{R}^n$ are no longer compact,
 which causes an essential difficulty for proving the LDP of \eqref{mvsde-1} with a polynomial drift $f$ of arbitrary degree.
 To solve the problem,
 we will employ the idea of uniform tail-ends estimates of solutions to establish the strong convergence of a family of solutions
 even though $f$ has an arbitrary growth order and the diffusion terms are not H\"{o}lder continuous in time.
 This approach was used in \cite{wangb2023JDE} for the standard fractional reaction-diffusion equation on $\mathbb{R}^n$
 where the drift and diffusion terms do not depend
 on the probability distributions of solutions.

 More precisely,
 to prove the LDP of \eqref{mvsde-1} with polynomial drift of any degree,
 we first establish the existence and uniqueness of adapted solutions to \eqref{mvsde-1}-\eqref{mvsde-1 IV}
 by the decoupled method
    (i.e., the method of freezing distributions), see Theorem \ref{mvsde-wellposed}.
 We then consider the controlled equation associated with \eqref{mvsde-1}-\eqref{mvsde-1 IV} for every control $v \in L^2(0,T; l^2)$,
 and prove the well-posedness by the approximation  arguments as in \cite{hll2021AMO, hhl2022arXiv}
 combined with more complicated analysis, see Theorem \ref{cewellposed}.
 To overcome the non-compactness of Sobolev embeddings on $\mathbb{R}^n$ and
 remove the assumptions of time H\"{o}lder continuity on diffusion coefficients in \cite{hll2021AMO, hhl2022arXiv},
 we derive the uniform tail-ends estimates of solutions in Lemma \ref{tailestim}
 and further prove the LDP of solutions in the Polish space
    $C([0,T], L^2({\mathbb{R}}^n ) ) \bigcap L^2(0,T; H^\alpha ({\mathbb{R}}^n ) )$
 by the weak convergence method, see Theorem \ref{main}.
 Under a stronger dissipativeness condition on the drift term $f$,
 we finally establish the LDP of solutions to \eqref{mvsde-1}-\eqref{mvsde-1 IV} in the space
 $ C([0,T], L^2({\mathbb{R}}^n )) \cap L^2(0,T; H^\alpha ({\mathbb{R}}^n ) ) \cap L^p(0,T; L^p ({\mathbb{R}}^n ) ) $
 in Theorem \ref{main-sd}.

 The paper is organized as follows.
  We first prove the well-posedness of \eqref{mvsde-1}-\eqref{mvsde-1 IV}
  and the corresponding controlled equation in Section 2 and Section 3, respectively.
  We then establish the LDP of \eqref{mvsde-1}-\eqref{mvsde-1 IV} in Section 4.

 Hereafter, we denote by
            $H=L^2(\mathbb{R}^n)$
 with inner product $(\cdot, \cdot)$ and norm $\| \cdot \|$.
 We also write
    $V=H^\alpha(\mathbb{R}^n)$
 with inner product $(\cdot, \cdot)_V$ and norm $\| \cdot \|_V$,
 where
 $$
   \| u \|_V^2
   = \| u \|^2
     + \int_{\mathbb{R}^n} \int_{\mathbb{R}^n} {\frac {|u(x) -u(y) |^2}
       {|x-y|^{n + 2\alpha}} } dx dy
   = \| u \|^2
     + {\frac 2{ C(n,\alpha) } }
       \| (- \triangle)^{\frac {\alpha}2}  u \|^2,
 $$
 where $C(n,\alpha)>0$ is a specific   number depending
 on $n$   and  $\alpha$.

 Denote by $\mathcal{P}(H)$ the space of all probability measures on $H$ with weak topology,
 and let
    $$ \mathcal{P}_m(H)
        =
          \left\{
                  \mu \in \mathcal {P}(H) : \mu(\| \cdot \|^m) := \int_H \| \xi \|^m \mu(d \xi)
                                                                < + \infty
          \right\},
    $$
 Then $\mathcal{P}_m(H)$ is a Polish space with $L^m$-Wasserstein distance given by:
    $$ \mathbb{W}_m(\mu, \nu)
       = \inf_{ \pi \in \mathcal{C}(\mu, \nu) }
                    \left(
                           \int_{H \times H} \| \xi -\eta \|^m \pi (d\xi, d\eta)
                    \right)^{\frac{1}{m}}  ,
                    \ \ \  \forall \  \mu, \nu \in \mathcal {P}_m(H),
    $$
 where $\mathcal{C}(\mu, \nu)$ is the set of all couplings of $\mu$ and $\nu$.

\section{Fractional McKean-Vlasov stochastic  equations on $\mathbb{R}^n$}

 In this section,
 we discuss the assumptions on the nonlinear drift and diffusion terms in \eqref{mvsde-1},
 and prove the existence and uniqueness of solutions.

\subsection{Assumptions on nonlinear drift and diffusion terms}

 In the sequel,
 we assume that
    $f: {\mathbb{R}} \times {\mathbb{R}}^n \times {\mathbb{R}} \times \mathcal{P}_2(H) \to {\mathbb{R}}$
 is a continuous function which satisfies the following conditions:

  \noindent
 $({\bf \Sigma 1})$\
  For any $t, u \in \mathbb{R}$,
          $x \in \mathbb{R}^n$,
          $\mu$ and $\mu_i \in \mathcal{P}_2(H)$ $(i=1,2)$,
\begin{align}\label{f-1}
   f(t,x,u,\mu) u \geq \lambda_1 |u|^p - \psi_1(t,x)(1 + |u|^2
   + \mu( \|\cdot\|^2) ),
\end{align}
\begin{align}\label{f-2}
   | f(t,x,u_1, \mu_1) - f(t,x,u_2, \mu_2) |
   \leq & \lambda_2 \left( \psi_2(t,x) + |u_1|^{p-2} + |u_2|^{p-2} \right) |u_1 - u_2|  \nonumber \\
        & + \psi_3(t,x) \mathbb{W}_2 (\mu_1, \mu_2) ),
          \ \ \
                \forall \  u_1, u_2 \in \mathbb{R},
\end{align}
\begin{align}\label{f-3}
   | f(t,x,u, \mu) |
   \leq \lambda_3 |u|^{p-1} + \psi_3(t,x) \left( 1 + \sqrt{ \mu( \|\cdot\|^2) } \right),
\end{align}
\begin{align}\label{f-4}
   \frac{\partial f}{\partial u}
    (t,x,u,\mu)
   \geq - \psi_4(t,x),
\end{align}
 where
 $p\ge 2$,  $\lambda_i >0$
 for  $i=1,2,3$,
       $\psi_i \in
        L_{loc}^\infty( \mathbb{R},
        L^\infty(\mathbb{R}^n) )
                   \bigcap
                   L_{loc}^1( \mathbb{R}, L^1(\mathbb{R}^n) )$
                   for
       $i = 1, 2,3,4$.

\begin{remark}
  Given  $t,u \in \mathbb{R}$,
          $x \in \mathbb{R}^n$
          and $\mu \in \mathcal{P}_2(H)$,
  let
     $$ f(t,x,u,\mu) = |u|^{p-2} u + \varphi(t,x) \int_{H}  h(\xi)\mu(d\xi), $$
  where
  $p\ge 2$,
  $\varphi:  \mathbb{R} \times \mathbb{R}^n \to \mathbb{R} $
  is continuous such that
  $\varphi
  \in
  L_{loc}^\infty( \mathbb{R},
  L^\infty(\mathbb{R}^n) )
  \bigcap
  L_{loc}^1( \mathbb{R}, L^1(\mathbb{R}^n) )$,
  and $h: H \rightarrow \mathbb{R}$ is     Lipschitz continuous.
  Then one may verify    that $f $ satisfies
  conditions \eqref{f-1}-\eqref{f-4}
  based on the fact that
  $\mathcal{P}_2(H) \subseteq \mathcal{P}_1(H)$
  and $\mathbb{W}_1(\cdot, \cdot)$ is the same as the Kantorovich-Rubinstein distance (see Remark 6.5  in \cite{villani2009}).
\end{remark}

We assume   the nonlinear
drift
$ g: \mathbb{R} \times \mathbb{R}^n \times \mathbb{R} \times \mathcal {P}_2(H)
\rightarrow \mathbb{R}$
is continuous which
further satisfies:
 for all $t, u_1, u_2 \in \mathbb{R},
  x \in \mathbb{R}^n$
and $\mu_1, \mu_2 \in \mathcal{P}_2(H)$,

  \noindent
 $({\bf \Sigma 2})$\
 \begin{align}
   & |g(t,x,0, \delta_0)| \leq \psi_g(t,x),   \label{g-1}  \\
   & |g(t,x, u_1, \mu_1) - g(t,x, u_1, \mu_2)|
     \leq \psi_g(t,x)
          \left( | u_1 - u_2 | + \mathbb{W}_2(\mu_1, \mu_2)
          \right),   \label{g-2}
\end{align}
 where
   $\psi_g(t,x) \in L_{loc}^\infty (\mathbb{R}, L^\infty (\mathbb{R}^n) )
                    \bigcap
                    L_{loc}^1 (\mathbb{R}, L^1 (\mathbb{R}^n) )$
   and $\delta_0$ is
   the  Dirac measure on $H$.

   It follows from \eqref{g-1}-\eqref{g-2} that
   for all
    $t, u  \in \mathbb{R},
   x \in \mathbb{R}^n$
   and $\mu  \in \mathcal{P}_2(H)$,
\begin{align} \label{g-3}
     |g(t,x, u, \mu)|
     \leq & |g(t,x, u, \mu) - g(t,x,0, \delta_0)|
             + |g(t,x,0, \delta_0)|
                \nonumber \\
     \leq & \psi_g(t,x)
                      \left (
                             1 + |u| + \sqrt{ \mu( \|\cdot\|^2) }
                      \right ).
\end{align}

Finally, for
   the diffusion terms, we assume:

  \noindent
 $({\bf \Sigma 3})$\
   $\sigma_1 = \{ \sigma_{1,k} \}_{k=1}^\infty : \mathbb{R} \rightarrow L^2(\mathbb{R}^n, l^2)$ is continuous,
   and for every $k\in \mathbb{N}$,
   $\sigma_{2,k}:
    \mathbb{R}   \times \mathbb{R} \times \mathcal {P}_2(H)
   \rightarrow \mathbb{R}$
   is continuous and
     there exist nonnegative sequences
   $\beta=\{\beta_{k}\}_{k=1}^\infty$ and $\gamma=\{\gamma_{k}\}_{k=1}^\infty$
   such that for any $t, s \in \mathbb{R}$ and $\mu \in \mathcal {P}_2(H)$,
   \begin{align}\label{sigma1}
       | \sigma_{2,k}(t,s,\mu ) | \leq \beta_{k} \left( 1 + \sqrt{ \mu( \|\cdot\|^2) } \right) + \gamma_{k} |s|.
   \end{align}
   In addition, we assume that  $\sigma_{2,k}(t,s,\mu)$ is Lipschitz continuous in
   $s$ and $\mu$
   uniformly for $t \in \mathbb{R}$;
    that is,
   there exists a positive constant $L_{\sigma,k}$ such that
   for any $t, s_1, s_2 \in \mathbb{R}$ and $\mu_1, \mu_2 \in \mathcal {P}_2(H)$,
  \begin{align} \label{sigma2}
      | \sigma_{2,k}(t, s_1, \mu_1) - \sigma_{2,k}(t, s_2, \mu_2) |
      \leq L_{\sigma,k} \left( | s_1 - s_2 | + \mathbb{W}_2 (\mu_1, \mu_2) \right),
  \end{align}
  where
       $ \sum_{k=1}^\infty
                  \left( \beta_{k}^2 + \gamma_{k}^2 + L_{\sigma,k}^2
                  \right)
        < \infty$.

 Let
     $l^2$
     be the Hilbert space of square summable sequences of real numbers,
     and   $W$ be a cylindrical
Wiener process in $l^2$ with
identity covariance operator
which means that there exists
a separable Hilbert space
$U$ such that
$W$ takes values in $U$ and   the
embedding $l^2 \hookrightarrow U$ is  Hilbert-Schmidt.

 Given $t \in \mathbb{R} $, $u \in H$ and $\mu \in \mathcal {P}_2(H)$,
 define $\sigma(t,u,\mu): l^2 \rightarrow H$ by
\begin{align}\label{sigmadef}
  \sigma (t,u,\mu)(\theta)(x)
   = \sum_{k=1}^{\infty}
                       \left( \sigma_{1,k} (t,x)
                              + \kappa(x) \sigma_{2,k} (t, u(x), \mu)
                       \right) \theta_k,\ \
                                 \forall\ \theta =(\theta_k)_{k=1}^{\infty}\in l^2,
                                           x \in \mathbb{R}^n.
\end{align}

 It follows from \eqref{sigma1} that the series in \eqref{sigmadef} is convergent in $H$.
 Furthermore, the operator $\sigma(t,u,\mu)$ is Hilbert-Schmidt from $l^2$ to $H$ and
\begin{align} \label{sigma3}
        & \|\sigma(t,u,\mu)\|_{L_2(l^2, H)}^2  \nonumber \\
      = & \sum_{k=1}^\infty \| \sigma_{1,k}(t,x) + \kappa(x) \sigma_{2,k} (t, u(x), \mu ) \|^2
           \nonumber \\
   \leq & 2 \sum_{k=1}^\infty \| \sigma_{1,k}(t) \|^2
          + 8 \| \kappa \|^2 \sum_{k=1}^\infty \beta_{k}^2 \left( 1 + \mu( \|\cdot\|^2) \right)
          + 4 \| \kappa \|_{L^\infty(\mathbb{R}^n)}^2 \sum_{k=1}^\infty \gamma_{k}^2 \| u \|^2  \nonumber\\
      = & 2 \| \sigma_{1}(t) \|_{L^2(\mathbb{R}^n, l^2)}^2
          + 8 \| \kappa \|^2 \| \beta \|_{l^2}^2 \left( 1 + \mu( \|\cdot\|^2) \right)
          + 4 \| \kappa \|_{L^\infty(\mathbb{R}^n)}^2 \| \gamma\|_{l^2}^2 \| u \|^2
             \nonumber \\
   \leq & M_{\sigma,T} \left( 1 + \| u \|^2 + \mu( \|\cdot\|^2) \right),
          \ \ \ \forall \ t \in [0,T],
\end{align}
 where $L_2(l^2, H)$ is the space of Hilbert-Schmidt operators from $l^2$ to $H$ with norm $\| \cdot \|_{L_2(l^2, H)}$,
 $M_{\sigma,T} = 2 \| \sigma_{1} \|_{C([0,T]; L^2(\mathbb{R}^n, l^2) )}^2
                 + 8 \| \kappa \|^2 \| \beta \|_{l^2}^2
                 + 4 \| \kappa \|_{L^\infty(\mathbb{R}^n)}^2 \| \gamma\|_{l^2}^2$.

 On the other hand, by \eqref{sigma2},
 we have for all $t \in \mathbb{R}$, $u_1, u_2 \in H$ and $\mu_1, \mu_2 \in \mathcal {P}_2(H)$,
\begin{align} \label{sigmalip}
     & \| \sigma (t, u_1, \mu_1) -\sigma (t, u_2, \mu_2)\|^2 _{L_2(l^2, H)}
          \nonumber \\
   = & \sum_{k=1}^\infty
            \| \kappa(x)
                      \left( \sigma_{2,k} (t, u_1, \mu_1)
                              - \sigma_{2,k} (t, u_2, \mu_2)
                      \right)
            \|^2  \nonumber \\
  \leq &
       2  \sum_{k=1}^\infty
        L_{\sigma,k}^2
                    \left[ \| \kappa \|_{L^\infty(\mathbb{R}^n)}^2
                    \| u_1 - u_2 \|^2
                           +
                           \| \kappa\|^2 \left( \mathbb{W}_2 (\mu_1, \mu_2)
                             \right)^2
                    \right]
                           \nonumber \\
  \leq & L_{\sigma} \left[ \| u_1 - u_2 \|^2
                           + \left( \mathbb{W}_2 (\mu_1, \mu_2)
                             \right)^2
                    \right],
\end{align}
 where
      $L_{\sigma} =
      2 \sum_{k=1}^\infty L_{\sigma,k}^2
      \left (
      \| \kappa \|_{L^\infty(\mathbb{R}^n)}^2
      +\| \kappa \|^2
      \right ) $.

With above notation,
the stochastic  equation \eqref{mvsde-1}
takes the  form:
\begin{align}\label{mvsde-2}
      & d u^\varepsilon(t) + (-\triangle)^\alpha u^\varepsilon(t) dt
         + f(t, \cdot,  u^\varepsilon (t), \mathcal{L}_{u^\varepsilon (t)} )
         dt \nonumber \\
    = & g(t, \cdot, u^\varepsilon (t), \mathcal{L}_{u^\varepsilon (t)} ) dt
        + \sqrt{\varepsilon}
                  \sigma(t, u^\varepsilon (t), \mathcal{L}_{u^\varepsilon (t)} ) dW(t),
\end{align}
with initial data
\begin{align} \label{mvsde-2 IV}
    u^\varepsilon(0) = u_0.
\end{align}

\subsection{Well-posedness of the stochastic equation}

A solution of
problem
\eqref{mvsde-2}-\eqref{mvsde-2 IV}
is understood in the following sense.

\begin{definition} \label{defsol}
  For $T>0$, $\varepsilon \in (0,1)$,
  $u_0 \in L^2(\Omega, \mathscr{F}_0; H)$,
 a continuous $H$-valued $\mathscr{F}_t$-adapted stochastic process $u^\varepsilon$ is called a
 solution of \eqref{mvsde-2}-\eqref{mvsde-2 IV} on $[0,T]$ if
 $$ u^\varepsilon
    \in L^2(\Omega, C([0,T], H) )
        \bigcap
        L^2(\Omega, L^2(0,T; V) )
        \bigcap
        L^p(\Omega, L^p(0,T; L^p(\mathbb{R}^n ) ) )
 $$
 such that for all $t \in [0,T]$ and $\xi \in V \bigcap L^p(\mathbb{R}^n)$,
 \begin{align*}
    & (u^\varepsilon (t), \xi)
      + \int_0^t
                 \left( (-\triangle)^{\frac{\alpha}{2}} u^\varepsilon(s),
                        (-\triangle)^{\frac{\alpha}{2}} \xi
                 \right) ds
      + \int_0^t \int_{\mathbb{R} ^n}  f( s, x, u^\varepsilon (s), \mathcal{L}_{u^\varepsilon (s)} ) \xi(x) dx ds
           \nonumber \\
  = & (u_0, \xi)
      + \int_0^t \left( g(s, \cdot, u^\varepsilon (s), \mathcal{L}_{u^\varepsilon (s)} ), \xi \right) ds
      + \sqrt{\varepsilon} \int_0^t
                                   \left( \xi, \sigma(s, u^\varepsilon(s), \mathcal{L}_{u^\varepsilon (s)} ) dW(s)
                                   \right),
                                   \ \ \   \mathbb{P} \text{-almost surely}.
 \end{align*}
\end{definition}

Note that if
$u^\varepsilon$
is a solution of
 \eqref{mvsde-2}-\eqref{mvsde-2 IV}
 in the sense of
 Definition \ref{defsol}, then
   $ (-\triangle)^{\alpha} u^\varepsilon \in L^2(\Omega, L^2(0,T; V^\ast) ) $
 and
   $ f(\cdot , \cdot,  u^\varepsilon , \mathcal{L}_{u^\varepsilon } ) \in L^q(\Omega, L^q(0,T; L^q(\mathbb{R}^n ) ) ) $
with $\frac{1}{p} + \frac{1}{q} =1$,
and hence   for all $t \in [0,T]$,
 \begin{align} \label{sfnsol-u}
    & u^\varepsilon (t)
      + \int_0^t (-\triangle)^{\alpha} u^\varepsilon(s) ds
      + \int_0^t f( s, \cdot, u^\varepsilon (s), \mathcal{L}_{u^\varepsilon (s)} )ds
        \nonumber \\
  = & u_0
      + \int_0^t g(s, \cdot, u^\varepsilon (s), \mathcal{L}_{u^\varepsilon (s)}) ds
      + \sqrt{\varepsilon} \int_0^t \sigma(s, u^\varepsilon (s), \mathcal{L}_{u^\varepsilon (s)} ) dW(s)
                                    \quad \text{in} \  \left( V \bigcap L^p(\mathbb{R}^n) \right)^\ast,
 \end{align}
 $\mathbb{P}$-almost surely.

The following theorem is concerned
with the   well-posedness of
problem  \eqref{mvsde-2}-\eqref{mvsde-2 IV}.

\begin{theorem} \label{mvsde-wellposed}
  Suppose $({\bf \Sigma 1})$-$({\bf \Sigma 3})$ hold.
  If $T>0$, $\varepsilon \in (0,1)$, $u_0 \in L^2(\Omega, \mathscr{F}_0; H)$,
  then problem \eqref{mvsde-2}-\eqref{mvsde-2 IV} has a unique solution
         $u^\varepsilon$
  in the sense of Definition \ref{defsol},
  which satisfies the energy equation: for all $t\in [0,T]$,
 \begin{align}\label{wang0321}
    & \| u^\varepsilon (t) \|^2
      + 2 \int_0^t
                 \| (-\triangle)^{\frac{\alpha}{2}} u^\varepsilon(s)
                 \|^2 ds
      + 2 \int_0^t \int_{\mathbb{R} ^n}  f( s, x, u^\varepsilon (s), \mathcal{L}_{u^\varepsilon (s)} ) u^\varepsilon (s) dx ds
        \nonumber \\
  = & \| u_0 \|^2
      + 2 \int_0^t (g(s, u^\varepsilon (s), \mathcal{L}_{u^\varepsilon (s)} ),  u^\varepsilon (s) ) ds
      + \varepsilon \int_0^t \| \sigma(s, u^\varepsilon (s), \mathcal{L}_{u^\varepsilon (s)} ) \|_{L_2(l^2,H)}^2 ds  \nonumber\\
    & + 2 \sqrt{\varepsilon} \int_0^t
          \left( u^\varepsilon (s),
         \sigma(s, u^\varepsilon (s), \mathcal{L}_{u^\varepsilon (s)}) dW(s)
                                      \right), \ \ \    \mathbb{P} \text{-almost surely}.
 \end{align}
 Moreover, the following uniform estimates are valid:
 \begin{align*}
         & \| u^\varepsilon \|_{L^2(\Omega, C([0,T], H) )}^2
           + \| u^\varepsilon \|_{L^2(\Omega, L^2(0,T; V) )}^2
           + \| u^\varepsilon \|_{L^p(\Omega, L^p(0,T; L^p(\mathbb{R}^n ) ) )}^p
                 \nonumber\\
    \leq & M_1
              \left( 1 + \| u_0 \|_{L^2(\Omega, \mathscr{F}_0; H)}^2
              \right),
 \end{align*}
 where $M_1=M_1(T) > 0$ is independent of $u_0$ and $\varepsilon$.
\end{theorem}

\begin{proof}
    We first prove the existence
    and uniqueness of solutions,
    and then derive the uniform
    estimates of solutions.

 {\bf Step 1:}  existence and uniqueness of solutions.

Given
$u_0 \in L^2(\Omega, \mathscr{F}_0; H)$
and   $\mu  \in C( [0,T], \mathcal{P}_2(H) )$,
    consider the    stochastic equation:
\begin{align}\label{mvsde-3}
     & d u_\mu^\varepsilon(t) + (-\triangle)^\alpha u_\mu^\varepsilon(t) dt
         + f_\mu(t, \cdot, u_\mu^\varepsilon (t) ) dt
           = g_\mu(t, \cdot, u_\mu^\varepsilon (t) ) dt
          + \sqrt{\varepsilon}
                  \sigma_\mu(t, u_\mu^\varepsilon (t) ) dW(t),
\end{align}
 with initial condition: \begin{equation}\label{mvsde-3a}
    u_\mu^\varepsilon(0) = u_0,
    \end{equation}
 where
 $$ f_\mu (t, \cdot, u_\mu^\varepsilon (t)) = f(t, \cdot,  u_\mu^\varepsilon (t), \mu(t) ),\
    g_\mu (t, \cdot, u_\mu^\varepsilon (t)) = g(t, \cdot, u_\mu^\varepsilon (t), \mu(t) ),\
    \sigma_\mu (t, u_\mu^\varepsilon (t)) = \sigma(t, u_\mu^\varepsilon (t), \mu(t) ).
 $$
 By  $({\bf \Sigma 1})$-$({\bf \Sigma 3})$, we find that
   $f_\mu, g_\mu $ and $\sigma_\mu$
 satisfy all the conditions in \cite[Theorem 6.3]{wangb2019JDE},
 and hence problem  \eqref{mvsde-3}-\eqref{mvsde-3a} has a unique solution
  $$ u_\mu^\varepsilon
     \in L^2(\Omega, C([0,T], H) )
         \bigcap
         L^2(\Omega, L^2(0,T; V) )
         \bigcap
         L^p(\Omega, L^p(0,T; L^p(\mathbb{R}^n ) ) ),
  $$
 which satisfies
 the energy equation
 \eqref{wang0321}
 with
 $\mathcal{L}_{u^\varepsilon (s)}
 $ replaced by $\mu (s)$, and
 the
  uniform estimate:
 \begin{align*}
         & \| u_\mu^\varepsilon \|_{L^2(\Omega, C([0,T], H) )}^2
           + \| u_\mu^\varepsilon \|_{L^2(\Omega, L^2(0,T; V) )}^2
           + \| u_\mu^\varepsilon \|_{L^p(\Omega, L^p(0,T; L^p(\mathbb{R}^n ) ) )}^p
                 \nonumber\\
    \leq & M_{T,\mu}
              \left( 1 + \| u_0 \|_{L^2(\Omega, \mathscr{F}_0; H)}^2
              \right),
 \end{align*}
 where $M_{T,\mu} = M (T, \mu) > 0$ is independent of $u_0$ and $\varepsilon$.

 By $u_\mu^\varepsilon \in L^2(\Omega, C([0,T], H) )$ and the Lebesgue dominated convergence theorem,
 we see that $u_\mu^\varepsilon \in  C([0,T],L^2(\Omega, H) )$.
Note that
 \begin{align*}
    \mathbb{W}_2
                 \left( \mathcal{L}_{u_\mu^\varepsilon(t)},
                               \mathcal{L}_{u_\mu^\varepsilon(s)}
                 \right)
    \leq
         \left( \mathbb{E}
                      \left[ \| u_\mu^\varepsilon(t) - u_\mu^\varepsilon(s) \|^2
                      \right]
         \right)^{ \frac{1}{2} },
             \ \ \   \forall \  s,t \in [0,T],
 \end{align*}
and hence
    $\mathcal{L}_{u_\mu^\varepsilon(\cdot)} \in C( [0,T], \mathcal{P}_2(H) )$.

  Define a  map
    $\Phi^{u_0} : C( [0,T], \mathcal{P}_2(H) ) \rightarrow C( [0,T], \mathcal{P}_2(H) )$
  by
 \begin{align} \label{map}
     \Phi^{u_0} (\mu)(t) = \mathcal{L}_{u_\mu^\varepsilon(t)},
                          \ \ \
                                 t\in [0, T], \  \mu \in C( [0,T], \mathcal{P}_2(H) ),
 \end{align}
 where $u_\mu^\varepsilon$ is the solution of \eqref{mvsde-3}-\eqref{mvsde-3a}.

Next, we will prove the existence of
solutions of \eqref{mvsde-2}-\eqref{mvsde-2 IV}
by finding a fixed point
of   $\Phi^{u_0}$.
To that end, we need  to show
  $\Phi^{u_0}$ is a contractive map
 in the complete  metric  space
 $ (C( [0,T], \mathcal{P}_2(H) ),
 \ d_{C( [0,T], \mathcal{P}_2(H) )})$ where the metric
 $ d_{C( [0,T], \mathcal{P}_2(H) )}$ is defined  by
$$ d_{C( [0,T], \mathcal{P}_2(H) )} (\mu, \nu) = \sup_{t \in [0,T]} e^{-\lambda t} \mathbb{W}_2 \left( \mu(t), \nu(t) \right),
\quad \forall \ \mu, \nu \in C( [0,T], \mathcal{P}_2(H) ).
$$

 Let $\mu, \nu \in C( [0,T], \mathcal{P}_2(H) )$,
 and $u_\mu^\varepsilon$ and $u_\nu^\varepsilon$ be the solution of \eqref{mvsde-3}-\eqref{mvsde-3a}.
By It\^{o}'s formula, we have
\begin{align}\label{sol-1}
       & \| u_\mu^\varepsilon(t) - u_\nu^\varepsilon(t) \|^2
         + 2 \int_0^t
                  \| (-\triangle)^{\frac{\alpha}{2}}
                           \left( u_\mu^\varepsilon(s)
                                  - u_\nu^\varepsilon(s)
                           \right)
                  \|^2 ds  \nonumber \\
       & + 2 \int_0^t \int_{\mathbb{R}^n}
                   \left( f(s, x, u_\mu^\varepsilon (s) , \mu(s) )
                          - f(s, x, u_\nu^\varepsilon (s), \nu(s) )
                   \right)
                   \left( u_\mu^\varepsilon (s)
                          - u_\nu^\varepsilon (s)
                   \right) dx ds
                         \nonumber \\
     = & 2 \int_0^t
                   \left( g(s, \cdot, u_\mu^\varepsilon (s), \mu(s) )
                           - g(s, \cdot, u_\nu^\varepsilon (s), \nu(s) ),
                          u_\mu^\varepsilon (s)
                           - u_\nu^\varepsilon (s)
                   \right) ds
                             \nonumber \\
       & + 2 \sqrt{\varepsilon}
                \int_0^t
                  \left( u_\mu^\varepsilon (s)
                          - u_\nu^\varepsilon (s),
                         \left( \sigma(s, u_\mu^\varepsilon (s), \mu(s) )
                                - \sigma(s, u_\nu^\varepsilon (s), \nu(s) )
                         \right) dW(s)
                  \right)
                  \nonumber \\
       & + \varepsilon
                \int_0^t
                  \| \sigma(s, u_\mu^\varepsilon (s), \mu(s) )
                      - \sigma(s, u_\nu^\varepsilon (s), \nu(s) )
                  \|_{L_{2}(l^2,H)}^2 ds.
\end{align}
 Taking
 the  expectation on both sides
 of  \eqref{sol-1}, we obtain
\begin{align}\label{sol-2}
      & \mathbb{E}
          \left[ \| u_\mu^\varepsilon(t) - u_\nu^\varepsilon(t) \|^2
          \right]
        + 2 \int_0^t
                 \mathbb{E}
                         \left[
                                \| (-\triangle)^{\frac{\alpha}{2}}
                                    \left( u_\mu^\varepsilon(s)
                                           - u_\nu^\varepsilon(s)
                                    \right)
                                \|^2
                         \right] ds  \nonumber \\
      & + 2 \mathbb{E}
               \left[\int_0^t \int_{\mathbb{R}^n}
                      \left( f(s, x, u_\mu^\varepsilon (s) , \mu(s) )
                             - f(s, x, u_\nu^\varepsilon (s), \nu(s) )
                      \right)
                      \left( u_\mu^\varepsilon (s)
                             - u_\nu^\varepsilon (s)
                      \right) dx ds
               \right]
                         \nonumber \\
     = & 2 \mathbb{E}
              \left[\int_0^t
                      \left( g(s, \cdot, u_\mu^\varepsilon (s), \mu(s) )
                             - g(s, \cdot, u_\nu^\varepsilon (s), \nu(s) ),
                             u_\mu^\varepsilon (s)
                              - u_\nu^\varepsilon (s)
                      \right) ds
              \right]
                             \nonumber \\
       & + \varepsilon
                \mathbb{E}
                     \left[ \int_0^t
                                     \| \sigma(s, u_\mu^\varepsilon (s), \mu(s) )
                                         - \sigma(s, u_\nu^\varepsilon (s), \nu(s) )
                                     \|_{L_{2}(l^2,H)}^2 ds
                     \right].
\end{align}

 For the third term on left-hand side of \eqref{sol-2},
 by \eqref{f-2}, \eqref{f-4} and the H\"{o}lder inequality, we obtain
\begin{align}\label{sol-3}
       & 2 \mathbb{E}
              \left[\int_0^t \int_{\mathbb{R}^n}
                      \left( f(s, x, u_\mu^\varepsilon (s) , \mu(s) )
                             - f(s, x, u_\nu^\varepsilon (s), \nu(s) )
                      \right)
                      \left( u_\mu^\varepsilon (s)
                             - u_\nu^\varepsilon (s)
                      \right) dx ds
               \right]
               \nonumber \\
     = & 2 \mathbb{E}
              \left[\int_0^t \int_{\mathbb{R}^n}
                      \left( f(s, x, u_\mu^\varepsilon (s) , \mu(s) )
                             - f(s, x, u_\nu^\varepsilon (s), \mu(s) )
                      \right)
                      \left( u_\mu^\varepsilon (s)
                             - u_\nu^\varepsilon (s)
                      \right) dx ds
               \right]
               \nonumber \\
       & + 2 \mathbb{E}
              \left[\int_0^t \int_{\mathbb{R}^n}
                      \left( f(s, x, u_\nu^\varepsilon (s) , \mu(s) )
                             - f(s, x, u_\nu^\varepsilon (s), \nu(s) )
                      \right)
                      \left( u_\mu^\varepsilon (s)
                             - u_\nu^\varepsilon (s)
                      \right) dx ds
               \right]
               \nonumber \\
  \geq & - 2 \| \psi_4 \|_{L^\infty(0,T; L^\infty(\mathbb{R}^n) ) }
           \int_0^t \mathbb{E}
                      \left[ \| u_\mu^\varepsilon (s)
                             - u_\nu^\varepsilon (s) \|^2
                      \right] ds  \nonumber \\
       & - 2 \mathbb{E}
              \left[ \int_0^t \int_{\mathbb{R}^n}
                        |\psi_3(s,x)| \
                       \mathbb{W}_2 \left( \mu(s), \nu(s) \right) \
                       \left| u_\mu^\varepsilon (s)
                              - u_\nu^\varepsilon (s)
                       \right| dx ds
               \right]
               \nonumber \\
  \geq & - 2 \| \psi_4 \|_{L^\infty(0,T; L^\infty(\mathbb{R}^n) ) }
           \int_0^t \mathbb{E}
                      \left[ \| u_\mu^\varepsilon (s)
                             - u_\nu^\varepsilon (s) \|^2
                      \right] ds  \nonumber \\
       & - 2 \mathbb{E}
              \left[ \int_0^t
                      \mathbb{W}_2 \left( \mu(s), \nu(s) \right) \
                        \|\psi_3(s)\| \
                        \| u_\mu^\varepsilon (s)
                              - u_\nu^\varepsilon (s)
                       \| ds
               \right]
               \nonumber \\
  \geq & - 2 \| \psi_4 \|_{L^\infty(0,T; L^\infty(\mathbb{R}^n) ) }
           \int_0^t \mathbb{E}
                      \left[ \| u_\mu^\varepsilon (s)
                             - u_\nu^\varepsilon (s) \|^2
                      \right] ds  \nonumber \\
       & - \int_0^t \| \psi_3(s) \|^2
                    \mathbb{E}
                      \left[ \| u_\mu^\varepsilon (s)
                             - u_\nu^\varepsilon (s) \|^2
                      \right] ds
         - \int_0^t
                    \left( \mathbb{W}_2 \left( \mu(s), \nu(s) \right)
                    \right)^2
            ds.
\end{align}

 For the first term on right-hand side of \eqref{sol-2},
 by \eqref{g-2} and the H\"{o}lder inequality, we obtain
\begin{align} \label{sol-4}
       & 2 \mathbb{E}
              \left[ \int_0^t
                      \left( g(s, \cdot, u_\mu^\varepsilon (s), \mu(s) )
                             - g(s, \cdot, u_\nu^\varepsilon (s), \nu(s) ),
                             u_\mu^\varepsilon (s)
                              - u_\nu^\varepsilon (s)
                      \right) ds
              \right]  \nonumber \\
  \leq & 2 \mathbb{E}
              \left[ \int_0^t \int_{\mathbb{R}^n}
                      \left| g(s, x, u_\mu^\varepsilon (s), \mu(s) )
                             - g(s, x, u_\nu^\varepsilon (s), \nu(s) )
                      \right|
                      \left| u_\mu^\varepsilon (s)
                              - u_\nu^\varepsilon (s)
                      \right| dx ds
              \right]  \nonumber \\
  \leq & 2 \mathbb{E}
              \left[ \int_0^t \int_{\mathbb{R}^n}
                                     \psi_g(s,x)
                                     \left( | u_\mu^\varepsilon (s) - u_\nu^\varepsilon (s) |
                                             + \mathbb{W}_2 \left( \mu(s), \nu(s) \right)
                                     \right)
                                     | u_\mu^\varepsilon (s) - u_\nu^\varepsilon (s) |
                       dx ds
              \right]  \nonumber \\
  \leq & 2 \| \psi_g \|_{L^\infty(0,T; L^\infty(\mathbb{R}^n) ) }
         \mathbb{E}
              \left[ \int_0^t
                              \| u_\mu^\varepsilon (s) - u_\nu^\varepsilon (s) \|^2
                      ds
              \right]   \nonumber \\
       & + 2 \mathbb{E}
              \left[
                     \int_0^t \mathbb{W}_2 \left( \mu(s), \nu(s) \right)
                                     \| \psi_g(s) \|
                                     \| u_\mu^\varepsilon (s) - u_\nu^\varepsilon (s) \|
                      ds
              \right]  \nonumber \\
  \leq & 2 \| \psi_g \|_{L^\infty(0,T; L^\infty(\mathbb{R}^n) ) }
         \mathbb{E}
              \left[ \int_0^t
                              \| u_\mu^\varepsilon (s) - u_\nu^\varepsilon (s) \|^2
                      ds
              \right]   \nonumber \\
       & + \int_0^t \| \psi_g(s) \|^2
                   \mathbb{E}
                             \left[
                                    \| u_\mu^\varepsilon (s) - u_\nu^\varepsilon (s) \|^2
                             \right]
            ds
         + \int_0^t
                    \left( \mathbb{W}_2 \left( \mu(s), \nu(s) \right)
                    \right)^2
            ds.
\end{align}

 For the second term on right-hand side of \eqref{sol-2},
 by \eqref{sigmalip}, we obtain that for any $\varepsilon \in (0,1)$,
\begin{align} \label{sol-5}
       & \varepsilon
                \mathbb{E}
                     \left[ \int_0^t
                                     \| \sigma(s, u_\mu^\varepsilon (s), \mu(s) )
                                         - \sigma(s, u_\nu^\varepsilon (s), \nu(s) )
                                     \|_{L_{2}(l^2,H)}^2 ds
                     \right]  \nonumber \\
  \leq & L_\sigma
                \int_0^t \mathbb{E}
                               \left[
                                     \| u_\mu^\varepsilon (s)
                                         - u_\nu^\varepsilon (s)
                                     \|^2
                               \right] ds
         + L_\sigma
                  \int_0^t
                           \left( \mathbb{W}_2 \left( \mu(s), \nu(s) \right)
                           \right)^2
                    ds.
\end{align}

 It follows from \eqref{sol-2}-\eqref{sol-5} that
 \begin{align}\label{sol-6}
      & \mathbb{E}
          \left[ \| u_\mu^\varepsilon(t) - u_\nu^\varepsilon(t) \|^2
          \right]
        + 2 \int_0^t
                 \mathbb{E}
                         \left[
                                \| (-\triangle)^{\frac{\alpha}{2}}
                                    \left( u_\mu^\varepsilon(s)
                                           - u_\nu^\varepsilon(s)
                                    \right)
                                \|^2
                         \right] ds  \nonumber \\
  \leq & \left( 2 \| \psi_4 \|_{L^\infty(0,T; L^\infty(\mathbb{R}^n) ) }
                + 2 \| \psi_g \|_{L^\infty(0,T; L^\infty(\mathbb{R}^n) ) }
                + L_\sigma
         \right)
           \int_0^t \mathbb{E}
                      \left[ \| u_\mu^\varepsilon (s)
                             - u_\nu^\varepsilon (s) \|^2
                      \right] ds  \nonumber \\
       & + \int_0^t
                   \left( \| \psi_3(s) \|^2
                          + \| \psi_g(s) \|^2
                   \right)
                    \mathbb{E}
                      \left[ \| u_\mu^\varepsilon (s)
                             - u_\nu^\varepsilon (s) \|^2
                      \right] ds
         + (2 +
          L_\sigma)
           \int_0^t
                    \left( \mathbb{W}_2 \left( \mu(s), \nu(s) \right)
                    \right)^2
            ds  \nonumber \\
     = & \int_0^t
                   \left( c_1
                          + \| \psi_3(s) \|^2
                          + \| \psi_g(s) \|^2
                   \right)
                    \mathbb{E}
                      \left[ \| u_\mu^\varepsilon (s)
                             - u_\nu^\varepsilon (s) \|^2
                      \right] ds  \nonumber \\
      & + (2 + L_\sigma)
           \int_0^t
                    \left( \mathbb{W}_2 \left( \mu(s), \nu(s) \right)
                    \right)^2
            ds,
\end{align}
 where
       $
          c_1 = 2 \| \psi_4 \|_{L^\infty(0,T; L^\infty(\mathbb{R}^n) ) }
                + 2 \| \psi_g \|_{L^\infty(0,T; L^\infty(\mathbb{R}^n) ) }
                + L_\sigma
              > 0.
       $

 Applying the Gronwall inequality to \eqref{sol-6}, we
 obtain  that for any $0 \leq   t \leq T$,
\begin{align} \label{sol-9}
      & \mathbb{E}
          \left[ \| u_\mu^\varepsilon(t) - u_\nu^\varepsilon(t) \|^2
          \right]
           \nonumber \\
  \leq & (2 + L_\sigma) \int_0^t
                   \left( \mathbb{W}_2 \left( \mu(s), \nu(s) \right)
                   \right)^2
            ds \cdot
            e^{\int_0^T
                   \left( c_1
                          + \| \psi_3(s) \|^2
                          + \| \psi_g(s) \|^2
                   \right) ds
               }   \nonumber \\
     = & C_T \int_0^t
                   \left( \mathbb{W}_2 \left( \mu(s), \nu(s) \right)
                   \right)^2
            ds,
\end{align}
 where $C_T = (2 + L_\sigma)
             e^{\int_0^T
                   \left( c_1
                          + \| \psi_3(s) \|^2
                          + \| \psi_g(s) \|^2
                   \right) ds
               }
          > 0$.
  By \eqref{sol-9}, we obtain,
  for  $t \in [0,T]$,
\begin{align}\label{sol-10}
      e^{- 2\lambda t}
            \mathbb{E}
            \left[ \| u_\mu^\varepsilon(t) - u_\nu^\varepsilon(t) \|^2
            \right]
    \leq & C_T \int_0^t e^{- 2\lambda (t-s)} e^{- 2\lambda s}
                    \left( \mathbb{W}_2 \left( \mu(s), \nu(s) \right)
                    \right)^2
             ds  \nonumber \\
    \leq & \frac{C_T }{2\lambda} \sup_{s \in [0,T]} e^{-2\lambda s}
                    \left( \mathbb{W}_2 \left( \mu(s), \nu(s) \right)
                    \right)^2.
\end{align}
 By \eqref{sol-10} and the definition of $\mathbb{W}_2(\cdot,\cdot)$,
 we have
\begin{align}\label{sol-11}
            e^{- 2\lambda t}
              \mathbb{W}_2
                         \left( \mathcal{L}_{u_\mu^\varepsilon(t)}, \mathcal{L}_{u_\nu^\varepsilon(t)}
                         \right)^2
    \leq \frac{C_T }{2\lambda} \sup_{s \in [0,T]} e^{- 2\lambda s}
                    \left( \mathbb{W}_2 \left( \mu(s), \nu(s) \right)
                    \right)^2
                    \ \  \text{for all} \  t \in [0,T].
\end{align}
   By \eqref{sol-11} and the definition of $d_{C( [0,T], \mathcal{P}_2(H) )}$, we
   obtain
\begin{align}\label{sol-12}
          d_{C( [0,T], \mathcal{P}_2(H) )}
              \left( \mathcal{L}_{u_\mu^\varepsilon},
                     \mathcal{L}_{u_\nu^\varepsilon}
              \right)
    \leq  \left( \frac{C_T}{2\lambda} \right)^{\frac{1}{2}}
          d_{C( [0,T], \mathcal{P}_2(H) )} \left( \mu, \nu \right).
\end{align}

 Choosing $\lambda > 0$ such that $\frac{C_T}{2\lambda} < \frac{1}{4}$,
  by \eqref{sol-12} and the definition of $\Phi^{u_0}$,
 we obtain
\begin{align*}
         {d_{C( [0,T], \mathcal{P}_2(H) )}}
           \left( \Phi^{u_0} (\mu),
                  \Phi^{u_0} (\nu)
           \right)
   \leq \frac{1}{2}
    {d_{C( [0,T], \mathcal{P}_2(H) )}}
   \left( \mu, \nu \right),
   \quad \forall \ \mu, \nu \in
    {C( [0,T], \mathcal{P}_2(H) )},
\end{align*}
 which shows $\Phi^{u_0}$ is a  contractive map in
 ${C( [0,T], \mathcal{P}_2(H) )}
 $ with metric
$
 d_{C( [0,T], \mathcal{P}_2(H) )}
 )$.
 Therefore,
  $\Phi^{u_0}$ has a unique fixed point $\bar{\mu} \in
  {C( [0,T], \mathcal{P}_2(H) )}
 $.
 Then,
    $u_{\bar{\mu}}^\varepsilon
    $
 is a solution of
 \eqref{mvsde-2}-\eqref{mvsde-2 IV}   on $[0,T]$.

 On the other hand, for every solution $u^\varepsilon$ of
 \eqref{mvsde-2}-\eqref{mvsde-2 IV},
 if we set
   $\mu(t) = \mathcal{L}_{u^\varepsilon(t)}$
   for $t\in [0,T]$,
   then  $\mu$
   is a  fixed point of $\Phi^{u_0}$ in
   $
{C( [0,T], \mathcal{P}_2(H) )}
$, which along with the uniqueness
of fixed points of
$\Phi^{u_0}$ and
\eqref{sol-9}
implies the uniqueness of solutions
of
 \eqref{mvsde-2}- \eqref{mvsde-2 IV}.

 {\bf Step 2:}  uniform estimates of solutions  of    \eqref{mvsde-2}- \eqref{mvsde-2 IV}.

 By It\^{o}'s formula,
 we have
\begin{align}\label{sue-1}
 & \| u^\varepsilon(t) \|^2
  + 2 \int_0^t \|(-\triangle)^{\frac{\alpha}{2}} u^\varepsilon(s) \|^2 ds
  + 2 \int_0^t \int_{ \mathbb{R}^n}
   f(s, x, u^\varepsilon(s), \mathcal{L}_{u^\varepsilon(s)} )  u^\varepsilon(s)  dx  ds \no\\
  = & \| u_0 \|^2
      + 2 \int_0^t \left( g(s, \cdot, u^\varepsilon(s), \mathcal{L}_{u^\varepsilon(s)} ),\ u^\varepsilon(s) \right) ds
      + \int_0^t \| \sigma (s, u^\varepsilon(s), \mathcal{L}_{u^\varepsilon(s)} ) \|_{L_2(l^2,H)}^2 ds \no\\
    & + 2 \int_0^t \left( u^\varepsilon(s),\  \sigma(s, u^\varepsilon(s), \mathcal{L}_{u^\varepsilon(s)} ) dW(s) \right),
     \ \ \  \text{for all} \  t \in [0,T].
\end{align}

 For the third
  term on the left-hand side of \eqref{sue-1},
 by \eqref{f-1}, we have
\begin{align} \label{sue-2}
      & 2 \int_0^t \int_{ \mathbb{R}^n}
        f(s, x, u^\varepsilon(s), \mathcal{L}_{u^\varepsilon(s)} )  u^\varepsilon(s) dx ds
           \nonumber \\
 \geq & 2 \lambda_1 \int_0^t \| u^\varepsilon (s) \|_{L^p(\mathbb{R}^n)}^p ds
        - 2 \int_0^t \| \psi_1(s) \|_{L^1(\mathbb{R}^n)} \left( 1 + \mathcal{L}_{u^\varepsilon(s)}
        (\|\cdot\|^2) \right) ds
         \nonumber \\
          & -2
        \| \psi_1 \|_{
            L^\infty (0,T;
         L^\infty (\mathbb{R}^n) )}
     \int_0^t
      \| u^\varepsilon (s) \|^2 ds.
 \end{align}

 For the second term on the right-hand side of \eqref{sue-1}, by \eqref{g-3}
  we obtain
\begin{align} \label{sue-3}
        & 2 \int_0^t \left( g(s, \cdot, u^\varepsilon(s), \mathcal{L}_{u^\varepsilon(s)} ),\ u^\varepsilon(s) \right) ds  \nonumber \\
         \leq &
         2\int_0^t
         \int_{\mathbb{R}^n}
         \psi_g (t,x)
         \left (
         |u^\varepsilon(s)|
         + |u^\varepsilon(s)^2
         +
         |u^\varepsilon(s)|
         \sqrt{\mathcal{L}_{u^\varepsilon(s)}
         (\|\cdot\|^2)
        }
         \right ) dx ds
         \nonumber \\
         \leq &
         4 \|\psi_g\|_{L^\infty(0,T; L^\infty(\mathbb{R}^n) )}
         \int_0^t
         \|  u^\varepsilon(s) \|^2 ds
         +
         \int_0^t
         \| \psi_g (s) \|_{L^1(\mathbb{R}^n)}
         (1+  \mathcal{L}_{u^\varepsilon(s)}
         (\| \cdot\|^2)
         )ds.
 \end{align}
 Then by \eqref{sue-1},
 \eqref{sue-2} and
 \eqref{sue-3}, we get
 \begin{align}\label{sue-4}
     & \| u^\varepsilon(t) \|^2
       + 2 \int_0^t \| (-\triangle)^{\frac{\alpha}{2}} u^\varepsilon(s) \|^2 ds
       +2 \lambda_1 \int_0^t \| u^\varepsilon (s) \|^p_{L^p} ds  \nonumber \\
\leq & \| u_0 \|^2
       +  \int_0^t
       ( 2\| \psi_1(s) \|_{L^1(\mathbb{R}^n)}
       +\| \psi_g(s) \|_{L^1(\mathbb{R}^n)}
       )
        \left( 1 + \mathcal{L}_{u^\varepsilon(s)}(\|\cdot\|^2) \right) ds  \nonumber \\
     &
       +  2 \left(
       \|\psi_1\|_{L^\infty(0,T; L^\infty(\mathbb{R}^n))}
        + 2 \|\psi_g\|_{L^\infty(0,T; L^\infty(\mathbb{R}^n))} \right) \int_0^t \| u^\varepsilon(s) \|^2 ds  \nonumber \\
     & + \int_0^t \| \sigma(s, u^\varepsilon(s), \mathcal{L}_{u^\varepsilon(s)} ) \|_{L_2(l^2,H)}^2 ds
       + 2 \int_0^t \left( u^\varepsilon(s),\ \sigma(s, u^\varepsilon(s), \mathcal{L}_{u^\varepsilon(s)} )dW(s) \right).
\end{align}
By \eqref{sue-4}, we obtain
for all $t\in [0,T]$,
\begin{align}\label{sue-5}
    & \mathbb{E} \left [ \sup_{0 \leq r\leq t}
    \left (\| u^\varepsilon(r) \|^2
    + 2 \int_0^r \| (-\triangle)^{\frac{\alpha}{2}} u^\varepsilon(s) \|^2 ds
    +2 \lambda_1 \int_0^r \| u^\varepsilon (s) \|^p_{L^p} ds
    \right )
    \right ] \nonumber \\
    \leq &
    \mathbb{E} \left (\| u_0 \|^2
    \right )
    +  \int_0^t
    ( 2\| \psi_1(s) \|_{L^1(\mathbb{R}^n)}
    +\| \psi_g(s) \|_{L^1(\mathbb{R}^n)}
    )
    \left( 1 + \mathcal{L}_{u^\varepsilon(s)}
    (\|\cdot\|^2) \right) ds  \nonumber \\
    &
    +  2 \left(
    \|\psi_1\|_{L^\infty(0,T; L^\infty(\mathbb{R}^n))}
    + 2 \|\psi_g\|_{L^\infty(0,T; L^\infty(\mathbb{R}^n))} \right) \int_0^t
    \mathbb{E} \left (
     \| u^\varepsilon(s) \|^2
     \right )  ds  \nonumber \\
    & +
    \mathbb{E} \left (
    \int_0^t \| \sigma(s, u^\varepsilon(s), \mathcal{L}_{u^\varepsilon(s)} ) \|_{L_2(l^2,H)}^2 ds
    \right ) \nonumber \\
    &
    + 2
    \mathbb{E} \left (
    \sup_{0\le r \le t}
     \int_0^r \left( u^\varepsilon(s),\ \sigma(s, u^\varepsilon(s), \mathcal{L}_{u^\varepsilon(s)} )dW(s) \right)
     \right ).
\end{align}

 For the last two terms on the right-hand side of \eqref{sue-5},
 by the Burkholder-Davis-Gundy inequality and \eqref{sigma3},
 we have for all $ t \in [0,T]$,
\begin{align} \label{sue-6}
       & \mathbb{E}
                       \left[ \int_0^t \| \sigma(s, u^\varepsilon(s), \mathcal{L}_{u^\varepsilon(s)} ) \|_{L_2(l^2,H)}^2 ds
                       \right]  \nonumber \\
       & + 2 \mathbb{E}
                        \left[ \sup_{
                            0\le r \le t  }
                                \left| \int_0^r
                                                \left( u^\varepsilon(s),\
                                                       \sigma(s, u^\varepsilon(s), \mathcal{L}_{u^\varepsilon(s)} ) dW(s)
                                                \right)
                                \right|
                        \right]  \nonumber \\
  \leq & \frac{1}{2} \mathbb{E} \left[ \sup_{0 \leq s \leq t} \|u^\varepsilon (s)\|^2 \right]
            + (1+2c^2) \mathbb{E}
                                 \left[ \int_0^t \| \sigma(s, u^\varepsilon (s), \mathcal{L}_{u^\varepsilon(s)} ) \|_{L_2(l^2,H)}^2 ds
                                 \right]  \nonumber \\
  \leq & \frac{1}{2} \mathbb{E}
                                \left[ \sup_{0 \leq s \leq t} \|u^\varepsilon(s)\|^2
                                \right]
         + (1+2c^2) M_{\sigma,T}
                                \mathbb{E}
                                 \left[ \int_0^t
                                              \left( 1 + \|u^\varepsilon(s)\|^2
                                                       + \mathcal{L}_{u^\varepsilon(s)}(\|\cdot\|^2)
                                              \right)
                                        ds
                                 \right],
  \end{align}
  where $c$ is a positive constant from the Burkholder-Davis-Gundy inequality.

 Then by \eqref{sue-5}-\eqref{sue-6}, we obtain
 for all $t\in [0,T]$,
 \begin{align}\label{sue-7}
    & \mathbb{E} \left [ \sup_{0 \leq r\leq t}
    \left (\| u^\varepsilon(r) \|^2
    + 2 \int_0^r \| (-\triangle)^{\frac{\alpha}{2}} u^\varepsilon(s) \|^2 ds
    +2 \lambda_1 \int_0^r \| u^\varepsilon (s) \|^p_{L^p} ds
    \right )
    \right ] \nonumber \\
    \leq &
 2  \mathbb{E} \left (\| u_0 \|^2
    \right )
    + 2 \int_0^t
    ( 2\| \psi_1(s) \|_{L^1(\mathbb{R}^n)}
    +\| \psi_g(s) \|_{L^1(\mathbb{R}^n)}
    )
    \left( 1 + \mathcal{L}_{u^\varepsilon(s)}
    (\|\cdot\|^2) \right) ds  \nonumber \\
    &
    +  4 \left(
    \|\psi_1\|_{L^\infty(0,T; L^\infty(\mathbb{R}^n))}
    + 2 \|\psi_g\|_{L^\infty(0,T; L^\infty(\mathbb{R}^n))} \right) \int_0^t
    \mathbb{E} \left (
    \| u^\varepsilon(s) \|^2
    \right )  ds  \nonumber \\
    &
    + 2(1+2c^2) M_{\sigma,T}
    \mathbb{E}
    \left[ \int_0^t
    \left( 1 + \|u^\varepsilon(s)\|^2
    + \mathcal{L}_{u^\varepsilon(s)}(\|\cdot\|^2)
    \right)
    ds
    \right]
     \nonumber \\
    = &
    2   \mathbb{E} \left (\| u_0 \|^2
    \right )
    + 2 \int_0^t
    ( 2\| \psi_1(s) \|_{L^1(\mathbb{R}^n)}
    +\| \psi_g(s) \|_{L^1(\mathbb{R}^n)}
    )
    \left( 1 +
        \mathbb{E} \left (
    \| u^\varepsilon(s) \|^2
    \right )
     \right) ds  \nonumber \\
    &
    +  4 \left(
    \|\psi_1\|_{L^\infty(0,T; L^\infty(\mathbb{R}^n))}
    + 2 \|\psi_g\|_{L^\infty(0,T; L^\infty(\mathbb{R}^n))} \right) \int_0^t
    \mathbb{E} \left (
    \| u^\varepsilon(s) \|^2
    \right )  ds  \nonumber \\
    &
    + 2(1+2c^2) M_{\sigma,T}
      \int_0^t
    \left( 1 + 2
    \mathbb{E} (
    \|u^\varepsilon(s)\|^2
    )
    \right)
    ds
    \nonumber \\
    \le  &
    2   \mathbb{E} \left (\| u_0 \|^2
    \right )
    +c_1
    +  \int_0^t
    (c_2+  4\| \psi_1(s) \|_{L^1(\mathbb{R}^n)}
    +2\| \psi_g(s) \|_{L^1(\mathbb{R}^n)}
    )
    \mathbb{E} \left (
    \sup_{0\le r\le s}
    \| u^\varepsilon(r) \|^2
    \right) ds  ,
 \end{align}
 where
  \begin{align*}
 c_1 =   4 \| \psi_1 \|_{L^1(0,T; L^1(\mathbb{R}^n) ) }
         + 2 \| \psi_g \|_{L^1(0,T; L^1(\mathbb{R}^n) ) }
          + 2(1+2c^2) M_{\sigma,T}T,
 \end{align*}
and
 \begin{align*}
    c_2 =   4 \| \psi_1 \|_{L^\infty (0,T; L^\infty (\mathbb{R}^n) ) }
    + 8 \| \psi_g \|_{L^\infty (0,T; L^\infty (\mathbb{R}^n) ) }
    + 4(1+2c^2) M_{\sigma,T}.
\end{align*}
 From \eqref{sue-7} and Gronwall's inequality, it follows that
 for all $t\in [0, T]$,
\begin{align*}
        &  \mathbb{E} \left [ \sup_{0 \leq r\leq t}
    \left (\| u^\varepsilon(r) \|^2
    + 2 \int_0^r \| (-\triangle)^{\frac{\alpha}{2}} u^\varepsilon(s) \|^2 ds
    +2 \lambda_1 \int_0^r \| u^\varepsilon (s) \|^p_{L^p} ds
    \right )
    \right ] \nonumber \\
  & \leq
        \left( 2 \mathbb{E} \left[ \| u_0 \|^2 \right]
               + c_1
        \right)
         e^{   \int_0^t
          \left( c_2  + 4 \| \psi_1(s)
                         \|_{L^1
                        (\mathbb{R}^n)}
                +2 \| \psi_g(s) \|_{L^1(\mathbb{R}^n)
                    }
                         \right) ds
           },
\end{align*}
   which together with
   \eqref{sue-1}
   completes the proof.
\end{proof}

\section{Measure-dependent controlled   equations on $\mathbb{R}^n$  }

In this section, we consider a measure-dependent  controlled equation corresponding to \eqref{mvsde-2},
which is useful
for  investigating
the LDP of the
fractional McKean-Vlasov stochastic reaction-diffusion equations
by the  weak convergence method.
We first discuss the well-posedness
and then derive
the uniform estimates of solutions
for the controlled equation.

\subsection{Well-posedness of controlled equations}

 If  $\varepsilon = 0$,
 then \eqref{mvsde-2} reduces to  the deterministic reaction-diffusion equation on $\mathbb{R}^n$:
\begin{align}\label{dde}
     \frac{\partial u^0(t)}{\partial t}
      + (-\triangle)^\alpha u^0(t)
      + f(t, \cdot,  u^0(t), \mathcal{L}_{u^0(t)} )
    = g(t, \cdot, u^0(t), \mathcal{L}_{u^0(t)} ),
\end{align}
 with initial data
\begin{align} \label{dde IV}
    u^0(0) = u_0,
\end{align}
 where the distribution
 $\mathcal{L}_{u^0(t)}$
 of ${u^0(t)}$  is the Dirac measure
 $\delta_{u^0(t)}$.

Similar to
 Theorem \ref{mvsde-wellposed},
 problem   \eqref{dde}-\eqref{dde IV}
 is well-posed in $H$ as stated below.

\begin{theorem} \label{dde-wellposed}
  Suppose $({\bf \Sigma 1})$-$({\bf \Sigma 2})$ hold.
  If $T>0$, $u_0 \in H$,
  then problem \eqref{dde}-\eqref{dde IV} has a unique solution
    $u^0 \in C([0,T], H)
             \bigcap
             L^2(0,T; V)
             \bigcap
             L^p(0,T; L^p(\mathbb{R}^n ) )$.
 Moreover, the following uniform estimates are valid:
 \begin{align*}
           \| u^0 \|_{C([0,T], H)}^2
           + \| u^0 \|_{L^2(0,T; V)}^2
           + \| u^0 \|_{L^p(0,T; L^p(\mathbb{R}^n ) )}^p
      \leq M_0 \left( 1 + \| u_0 \|^2 \right),
 \end{align*}
 where $M_0 = M_0(T) > 0$ is independent of   $u_0$.
\end{theorem}

We now
 consider the controlled equation on $\mathbb{R}^n$
 for   a control $v \in L^2(0,T; l^2)$:
\begin{align}\label{contequ}
      & \frac{\partial u_v (t)}{\partial t}
        + (-\triangle)^\alpha u_v (t)
        + f(t, \cdot,  u_v(t), \mathcal{L}_{u^0(t)} )
    =  g(t, \cdot, u_v(t), \mathcal{L}_{u^0(t)} )
        + \sigma(t, u_v (t), \mathcal{L}_{u^0(t)}) v(t),
\end{align}
 with initial data
 \begin{align}\label{contequ IV}
       u_v(0) = u_0,
 \end{align}
 where $u^0(t)$ is the solution to \eqref{dde}-\eqref{dde IV}.

\begin{remark}
  Note that the controlled equation
  \eqref{contequ}  associated with
  the  McKean-Vlasov stochastic
  equation \eqref{mvsde-2}
   depends on distributions of
   the solutions of the deterministic equation \eqref{dde}.
  This  is quite  different from the usual controlled equations associated with  the   stochastic equations
  which are independent of    distributions,
  see  \cite[Remark 2.2]{hll2021AMO}
  for more details in this regard.
\end{remark}

The next theorem is concerned
with the   well-posedness of
 \eqref{contequ}-\eqref{contequ IV}.

\begin{theorem}\label{cewellposed}
  Suppose that $({\bf \Sigma 1})$-$({\bf \Sigma 3})$ hold.
  Then for every $v \in L^2(0,T; l^2)$,
  problem \eqref{contequ}-\eqref{contequ IV} has a unique solution
    $ u_v \in C([0,T], H) \bigcap L^2(0,T; V) \bigcap L^p(0,T; L^p(\mathbb{R}^n)). $

  Moreover, for any
      $\| u_{0} \| \vee \| u_{0,i} \| \leq R_1$,
  and $\| v \|_{L^2(0, T; l^2)} \vee \| v_i \|_{L^2(0, T; l^2)} \leq R_2$,
      $i = 1,2$,
  the solutions $u_v(t)$ and $u_{v_i}(t)$ of \eqref{contequ}
  with initial data $u_0$ and $u_{0,i}$ respectively,
  satisfy for any $t \in [0,T]$,
  \begin{align}\label{dceue}
      \| u_{v}(t) \|^2
      + \int_0^T \| u_v(s) \|_V^2 ds
      + \int_0^T \| u_v(s) \|_{L^p(\mathbb{R}^n)}^p ds
      \leq M_2
  \end{align}
  and
  \begin{align} \label{dcelip}
         & \| u_{v_1}(t) - u_{v_2}(t)
         \|^2
           + \int_0^T \| u_{v_1}(s) - u_{v_2}(s) \|_V^2 ds
             \nonumber\\
    \leq & M_2
               \left( \| u_{0,1} - u_{0,2} \|^2
                      + \| v_1 - v_2 \|^2_{L^2(0, T; l^2)}
               \right),
  \end{align}
  where the constant $M_2 > 0$ depends on $R_1, R_2$ and $T$.
\end{theorem}

\begin{proof}
   Given $v\in L^2(0,T; l^2)$,
   by
      $({\bf \Sigma 1})$-$({\bf \Sigma 3})$
      and the argument of   \cite{wangb2019JDE},
      one can verify  that
      for every
      $u_0\in H$,
       problem \eqref{contequ}-\eqref{contequ IV} has a unique solution
 $u_v$ in the space
    $$C([0,T], H) \bigcap L^2(0,T; V) \bigcap L^p(0,T; L^p(\mathbb{R}^n)).$$

    Next, we derive
     the uniform  estimates on
     the
     solutions
      of \eqref{contequ}-\eqref{contequ IV}. Note that
   \begin{align}\label{ceue-1}
     & \frac{d}{dt} \| u_v(t) \|^2
        + 2 \| (-\triangle)^{\frac{\alpha}{2}} u_v(t) \|^2
        + 2 \int_{\mathbb{R}^n} f(t, x, u_v(t), \mathcal{L}_{u^0(t)} ) u_v(t) dx
          \nonumber \\
   = & 2
         \left( g(t, \cdot, u_v(t), \mathcal{L}_{u^0(t)} ),
                u_v(t)
         \right)
       + 2 \left( \sigma(t, u_v(t), \mathcal{L}_{u^0(t)} ) v(t),
                  u_v(t)
           \right).
\end{align}

From \eqref{f-1} and Theorem \ref{dde-wellposed}, it follows that
\begin{align}\label{ceue-2}
        & 2 \int_{\mathbb{R}^n} f(t, x, u_v(t), \mathcal{L}_{u^0(t)} ) u_v(t) dx
          \nonumber \\
   \geq &  2 \lambda_1 \| u_v(t) \|_{L^p(\mathbb{R}^n)}^{p}
           - 2 \| \psi_{1}(t) \|_{L^1(\mathbb{R}^n)}
               \left( 1 + \mathcal{L}_{u^0(t)} (\|\cdot\|^2) \right)
               -2 \| \psi_1 (t)\|
               _{L^\infty (\mathbb{R}^n)}
               \| u_v(t) \|^2
                 \nonumber \\
   \geq &  2 \lambda_1 \| u_v(t) \|_{L^p(\mathbb{R}^n)}^{p}
           - 2 \| \psi_{1}(t) \|_{L^1(\mathbb{R}^n)}
               \left( 1 + \| u^0(t) \|^2) \right)
                -2 \| \psi_1 (t)\|
               _{L^\infty (\mathbb{R}^n)}
               \| u_v(t) \|^2
                 \nonumber \\
   \geq &  2 \lambda_1 \| u_v(t) \|_{L^p(\mathbb{R}^n)}^{p}
           - 2 \| \psi_{1}(t) \|_{L^1(\mathbb{R}^n)}
               \left[ 1 + M_0( 1 + R_1^2 ) \right]
                -2 \| \psi_1 (t)\|
               _{L^\infty (\mathbb{R}^n)}
               \| u_v(t) \|^2.
\end{align}

 By \eqref{g-3} and Theorem \ref{dde-wellposed},
 we have
\begin{align}\label{ceue-3}
        & 2 \left( g(t, \cdot, u_v(t), \mathcal{L}_{u^0(t)} ), u_v(t) \right)
            \nonumber \\
   \leq & \| u_v(t) \|^2
           + \| g(t, \cdot, u_v(t), \mathcal{L}_{u^0(t)} ) \|^2
             \nonumber \\
   \leq & \| u_v(t) \|^2
           + 4 \| \psi_g(t) \|_{L^\infty(\mathbb{R}^n) }^2
                 \|  u_v(t) \|^2
                 + 2 \| \psi_g(t) \|^2
                 \left (
                 1+ 2M_0( 1 + R_1^2 )
                 \right )  .
\end{align}
By \eqref{sigma3} and Theorem \ref{dde-wellposed},  we have
\begin{align}\label{ceue-4}
        & 2
            \left( \sigma(t, u_v(t), \mathcal{L}_{u^0(t)} ) v(t),
                   u_v(t)
            \right)  \nonumber\\
   \leq & \| \sigma(t, u_v(t), \mathcal{L}_{u^0(t)} ) \|^2_{L_2(l^2, H)} + \| v(t) \|_{l^2}^2 \| u_v(t) \|^2 \nonumber\\
   \leq & M_{\sigma,T} \left( 1 + \|u^0(t)\|^2 \right)
          + \left( M_{\sigma,T} + \| v(t) \|_{l^2}^2 \right) \| u_v (t) \|^2
            \nonumber \\
   \leq & M_{\sigma,T} \left[ 1 + M_0\left( 1 + R_1^2 \right) \right]
          + \left( M_{\sigma,T} + \| v(t) \|_{l^2}^2 \right) \| u_v (t) \|^2.
\end{align}
Then by \eqref{ceue-1}-\eqref{ceue-4}, we obtain
\begin{align}\label{ceue-5}
      & \frac{d}{dt} \| u_v(t) \|^2
        + 2 \| (-\triangle)^{\frac{\alpha}{2}} u_v(t) \|^2
        + 2 \lambda_1 \| u_v(t) \|_{L^p(\mathbb{R}^n)}^{p}
             \nonumber \\
 \leq &
          \left( 1 +
          2 \| \psi_1 (t) \|
          _{L^\infty(\mathbb{R}^n) }
          + 4 \| \psi_g(t) \|_{L^\infty(\mathbb{R}^n) }^2
                   + M_{\sigma,T}
                   + \| v(t) \|_{l^2}^2
          \right) \| u_v(t) \|^2  \nonumber\\
      & + 2 \| \psi_{1}(t) \|_{L^1(\mathbb{R}^n)}
               \left[ 1 + M_0( 1 + R_1^2 ) \right]
        + 2 \| \psi_g(t) \| ^2
        (1+ 2 M_0( 1 + R_1^2 ) )
          \nonumber \\
      &
        + M_{\sigma,T} \left[ 1 + M_0(1 + R_1^2) \right]
        \nonumber \\
 \leq &
         \left( c + \| v(t) \|_{l^2}^2 \right) \| u_v(t) \|^2
         + c
             \left( 1 + \| \psi_{1}(t) \|_{L^1(\mathbb{R}^n)}
                      + \| \psi_g(t) \|^2
             \right),
\end{align}
 where $c=c(T, R_1)>0$

 By \eqref{ceue-5} and Gronwall's inequality,
 we obtain  for all  $t \in [0,T]$,
\begin{align}\label{ceue-6}
        \| u_v(t) \|^2
 \leq & \| u_0 \|^2 e^{\int^{t}_{0} \left( c + \|v(s)\|_{l^2}^2 \right) ds}
       + c \int_0^t e^{\int_s^t \left( c + \|v(r)\|_{l^2}^2 \right) dr}
                   \left( 1 + \| \psi_{1}(s) \|_{L^1(\mathbb{R}^n)}
                            + \| \psi_g(s) \|^2
                   \right)
            ds  \nonumber \\
 \leq & R_1^2 e^{cT + R_2^2}
       + c e^{ cT + R_2^2 }
                   \left( T + \| \psi_{1} \|_{L^1(0,T; L^1(\mathbb{R}^n) ) }
                            + \| \psi_g \|_{L^2(0,T; H) }^2
                   \right).
\end{align}
By \eqref{ceue-5}-\eqref{ceue-6} we
see that   there exists   $c_1 = c_1(T, R_1, R_2)>0$ such that
\begin{align}\label{ceue-7}
      & \| u_v(t) \|^2
        + \int_0^t 2 \| (-\triangle)^{\frac{\alpha}{2}} u_v(s) \|^2 ds
        + 2 \lambda_1 \int_0^t \| u_v(s) \|_{L^p(\mathbb{R}^n)}^{p} ds
      \leq c_1 , \ \ \  \forall \  t \in [0,T].
\end{align}
Then \eqref{dceue} follows from
\eqref{ceue-7}.

We now  prove \eqref{dcelip},
 which implies   the uniqueness of solutions to \eqref{contequ}-\eqref{contequ IV}.
  Let $u_1(t) = u_{v_1}(t, u_{0,1})$
 and $u_2(t) = u_{v_2}(t, u_{0,2})$
  be the solutions of \eqref{contequ}-\eqref{contequ IV} with
  initial data
  $u_{0,1}$ and
  $u_{0,2}$ and
    controls $v_1$
  and $v_2$,
  respectively.

 By \eqref{dceue}, we find that
 if  $\| u_{0,1} \|
 \vee \| u_{0,2} \|
  \leq R_1$
 and $\|v_1\|_{L^2(0,T; l^2)}
 \vee \|v_2\|_{L^2(0,T; l^2)} \leq R_2$,
 then   for all  $ t \in [0,T]$,
 \begin{align}\label{dcel-0}
    \| u_1(t) \|^2 \leq  M_2  \
    \text{ and } \
    \| u_2(t) \|^2 \leq  M_2.
 \end{align}
 By \eqref{contequ}, we obtain
 \begin{align}\label{dcel-1}
     & \frac{d}{dt} \| u_1(t) - u_2(t) \|^2
       + 2 \| (-\triangle)^{\frac{\alpha}{2}} \left( u_1(t) - u_2(t) \right) \|^2
             \nonumber \\
     & + 2 \int_{\mathbb{R}^n}
                   \left( f(t, x, u_1(t), \mathcal{L}_{u_0(t)} )
                           - f(t, x, u_2(t), \mathcal{L}_{u_0(t)})
                   \right)
                   \left( u_1(t) - u_2(t) \right) dx
          \nonumber \\
   = & 2
         \left(
                g(t, \cdot, u_1(t), \mathcal{L}_{u_0(t)} )
                 - g(t, \cdot, u_2(t), \mathcal{L}_{u_0(t)} ), \
                u_1(t) - u_2(t)
         \right)  \nonumber \\
     & + 2
          \left(
                 \sigma(t, u_1(t), \mathcal{L}_{u_0(t)} ) v_1(t)
                  - \sigma(t, u_2(t), \mathcal{L}_{u_0(t)} ) v_2(t), \
                 u_1(t) - u_2(t)
          \right).
 \end{align}
 For the third term on the left-and side of \eqref{dcel-1},
 by \eqref{f-4}, we have
\begin{align}\label{dce1-2}
     &  - 2  \int_{\mathbb{R}^n}
                    \left( f(t, x, u_1(t), \mathcal{L}_{u_0(t)} )
                            - f(t, x, u_2(t), \mathcal{L}_{u_0(t)})
                    \right)
                    \left( u_1(t) - u_2(t)
                    \right) dx
                   \nonumber \\
     \leq & 2 \| \psi_4(t) \|_{L^\infty(\mathbb{R}^n)} \| u_1(t) - u_2(t) \|^2.
\end{align}

 For the first term on the right-hand side of \eqref{dcel-1},
 by \eqref{g-2}, we obtain
\begin{align}\label{dce1-3}
       & 2
          \left(
                 g(t, \cdot, u_1(t), \mathcal{L}_{u_0(t)} )
                  - g(t, \cdot, u_2(t), \mathcal{L}_{u_0(t)} ), \
                 u_1(t) - u_2(t)
          \right)  \nonumber \\
  \leq & 2 \| g(t, \cdot, u_1(t), \mathcal{L}_{u_0(t)} ) - g(t, \cdot, u_2(t), \mathcal{L}_{u_0(t)} ) \|
           \| u_1(t) - u_2(t) \|
             \nonumber \\
  \leq & 2 \|\psi_g\|_{L^\infty(0,T;L^\infty(\mathbb{R}^n) ) }  \| u_1(t) - u_2(t) \|^2.
\end{align}

 For the second term on the right-hand side of \eqref{dcel-1},
 by \eqref{sigma3}, \eqref{dcel-0}, and Theorem \ref{dde-wellposed}, we get
\begin{align}\label{dcel-4}
       & 2
           \left(
                  \sigma(t, u_1(t), \mathcal{L}_{u_0(t)} ) v_1(t)
                   - \sigma(t, u_2(t), \mathcal{L}_{u_0(t)} ) v_2(t), \
                  u_1(t) - u_2(t)
           \right)  \nonumber \\
  \leq & 2 \| \left( \sigma(t, u_1(t), \mathcal{L}_{u_0(t)} ) - \sigma(t, u_2(t), \mathcal{L}_{u_0(t)} ) \right) v_1(t) \|
               \| u_1(t) - u_2(t) \|  \nonumber \\
       & + 2 \| \sigma(t, u_2(t), \mathcal{L}_{u_0(t)} ) \left( v_1(t) - v_2(t) \right) \|
                 \| u_1(t) - u_2(t) \|  \nonumber \\
  \leq & 2 \| \sigma(t, u_1(t), \mathcal{L}_{u_0(t)} ) - \sigma(t, u_2(t), \mathcal{L}_{u_0(t)} ) \|_{L_2(l^2, H)}
               \| v_1(t) \|_{l^2}
               \| u_1(t) - u_2(t) \|  \nonumber \\
       & + \| \sigma(t, u_2(t), \mathcal{L}_{u_0(t)} ) \|_{L_2(l^2, H)}^2
           \| u_1(t) - u_2(t) \|^2
         + \| v_1(t) - v_2(t) \|_{l^2}^2
                  \nonumber \\
  \leq & 2 L_{\sigma}^{\frac{1}{2}}
               \| v_1(t) \|_{l^2}
               \| u_1(t) - u_2(t) \|^2  \nonumber \\
       & + M_{\sigma,T}
                       \left( 1 + \| u_2(t) \|^2 + \|u^0(t)\|^2
                       \right)  \| u_1(t) - u_2(t) \|^2
         + \| v_1(t) - v_2(t) \|_{l^2}^2
              \nonumber \\
  \leq & 2 L_{\sigma}^{\frac{1}{2}}
               \| v_1(t) \|_{l^2}
               \| u_1(t) - u_2(t) \|^2  \nonumber \\
       & + M_{\sigma,T}
                       \left[ 1 + M_2 + M_0 \left( 1 + R_1^2 \right)
                       \right]  \| u_1(t) - u_2(t) \|^2
         + \| v_1(t) - v_2(t) \|_{l^2}^2
              \nonumber \\
  \leq & c_1 \left( 1 + \| v_1(t) \|_{l^2} \right) \| u_1(t) - u_2(t) \|^2
         + \| v_1(t) - v_2(t) \|_{l^2}^2,
\end{align}
where
     $ c_1 = 2 L_{\sigma}^{\frac{1}{2}}
             + M_{\sigma,T}
                       \left[ 1 + M_2 + M_0 \left( 1 + R_1^2 \right)
                       \right] $.

 It  follows from \eqref{dcel-1}-\eqref{dcel-4} that for $t \in [0,T]$,
\begin{align} \label{dcel-5}
        & \frac{d}{dt} \| u_1(t) - u_2(t) \|^2
          + 2 \| (-\triangle)^{\frac{\alpha}{2}} \left( u_1(t) - u_2(t) \right) \|^2
             \nonumber \\
   \leq & \left( 2 \| \psi_4(t) \|_{L^\infty(\mathbb{R}^n)}
                    + 2 \|\psi_g\|_{L^\infty(0,T;L^\infty(\mathbb{R}^n) ) }
                    + c_1
                    + c_1 \| v_1(t) \|_{l^2}
            \right)
            \| u_1(t) - u_2(t) \|^2
                \nonumber \\
        & + \| v_1(t) - v_2(t) \|_{l^2}^2
                \nonumber \\
   \leq & c_2 \left( 1 + \| v_1(t) \|_{l^2}
              \right)  \| u_1(t) - u_2(t) \|^2
          + \| v_1(t) - v_2(t) \|_{l^2}^2,
\end{align}
 where
      $ c_2 = 2 \| \psi_4 \|_{L^\infty(0,T; L^\infty(\mathbb{R}^n) ) }
              + 2 \|\psi_g\|_{L^\infty(0,T;L^\infty(\mathbb{R}^n) ) }
              + c_1.
      $

 By  applying Gronwall's inequality to \eqref{dcel-5}, we
 get for all  $t \in [0,T]$,
\begin{align*}
       & \| u_1(t) - u_2(t) \|^2
             \nonumber \\
  \leq & \| u_{0,1} - u_{0,2} \|^2
         e^{ c_2 \int_{0}^{t}
                          \left( 1 + \| v_1(r) \|_{l^2}
                           \right) dr }
        + \int_{0}^{t}
                       e^{ c_2 \int_{s}^{t}
                                  \left( 1 + \| v_1(r) \|_{l^2}
                                 \right) dr }
                       \| v_1(s) - v_2(s) \|_{l^2}^2 ds,
\end{align*}
 from which we obtain for any $t \in [0,T]$,
\begin{align} \label{dcel-6}
           \| u_1(t) - u_2(t) \|^2
    \leq & e^{ c_2 T + c_2 T^{\frac{1}{2}} R_2 }
           \left( \| u_{0,1} - u_{0,2} \|^2
                   + \| v_1 - v_2 \|_{L^2(0,T; l^2)}^2
           \right).
\end{align}
 Integrating \eqref{dcel-5}, we obtain by \eqref{dcel-6}  that
\begin{align} \label{dcel-7}
        &  2\int_0^T
                   \| (-\triangle)^{\frac{\alpha}{2}} \left( u_1(s) - u_2(s) \right)
                   \|^2 ds  \nonumber \\
   \leq &
    \| u_{0,1} - u_{0,2} \|^2
    +
   c_2
             \left( T + T^{\frac{1}{2}} R_2
             \right)
          \sup_{t\in [0,T]} \| u_1(t) - u_2(t) \|^2
          + \| v_1 - v_2 \|_{L^2(0,T; l^2)}^2.
\end{align}
Then, \eqref{dcelip} follows from \eqref{dcel-6} and \eqref{dcel-7}.
\end{proof}

\subsection{Uniform tail-estimates of solutions}

 In this subsection,
 we   derive the uniform estimates
 on  the tails of solutions to \eqref{contequ}-\eqref{contequ IV} outside a bounded domain,
 which are crucial for establishing the precompactness of solutions.

\begin{lemma} \label{tailestim}
  Suppose   $({\bf \Sigma 1})$-$({\bf \Sigma 3})$ hold,
  $T>0$  and $ u_0 \in H$.
  Then for any $R > 0$ and $\delta > 0$, there exists $m_0 = m_0(u_0, T, R, \delta) > 0$ such that
  for all $v \in L^2(0, T; l^2)$ with $\| v \|_{L^2(0, T; l^2)} \leq R$,
  the solution $u_v$ to problem \eqref{contequ}-\eqref{contequ IV}
  satisfies
  \begin{align*}
      \int_{|x| \geq m} |u_v(t,x)|^2 dx < \delta,
        \ \ \forall\  m \geq m_0, \   t \in [0,T].
  \end{align*}
\end{lemma}

\begin{proof}
   Let $\zeta: \mathbb{R}^n \rightarrow [0,1]$ be a smooth function
   such that
   $\zeta(x)=0$ for $|x| \leq \frac{1}{2}$
   and $\zeta(x)=1$ for $|x| \geq 1$.

   Given $m \in \mathbb{N}$,
   let $\zeta_m(x) = \zeta( \frac{x}{m} )$,
   then by \eqref{contequ}, we have
\begin{align*}
        & \frac{\partial}{\partial t} \zeta_m u_v (t)
          + \zeta_m (-\triangle)^\alpha u_v (t)
          + \zeta_m f(t, \cdot, u_v(t), \mathcal{L}_{u^0(t)} )
            \nonumber \\
      = & \zeta_m g(t, \cdot, u_v(t), \mathcal{L}_{u^0(t)} )
          + \zeta_m \sigma(t, u_v (t), \mathcal{L}_{u^0(t)}) v(t).
\end{align*}
 Then  we get
\begin{align} \label{te1}
     & \frac{d}{dt} \| \zeta_m u_v(t) \|^2
        + 2
            \left( (-\triangle)^{\frac{\alpha}{2}} u_v(t),
                      (-\triangle)^{\frac{\alpha}{2}}
                                \left( \zeta_m^2 u_v(t) \right)
            \right)
             \nonumber \\
     & + 2 \int_{\mathbb{R}^n} f(t, x, u_v(t), \mathcal{L}_{u^0(t)} ) \zeta_m^2(x) u_v(t) dx
           \nonumber \\
   = & 2 \left( \zeta_m g(t, \cdot, u_v(t), \mathcal{L}_{u^0(t)} ), \zeta_m u_v(t) \right)
       + 2 \left( \zeta_m \sigma(t, u_v (t), \mathcal{L}_{u^0(t)}) v(t), \zeta_m u_v(t) \right).
\end{align}
 For the second term on the left-hand side of \eqref{te1},
 by (8.18) in \cite{wangbJDE2019},
 we find that there exists a positive constant $c_1$ independent of $m$ such that for all $t \geq 0$,
\begin{align} \label{te2}
       & - 2
            \left( (-\triangle)^{\frac{\alpha}{2}} u_v(t),
                      (-\triangle)^{\frac{\alpha}{2}}
                                \left( \zeta_m^2 u_v(t) \right)
            \right)
             \nonumber \\
  \leq & c_1 C(n,\alpha) m^{-\alpha} \| u_v(t) \|^2
         + 2 c_1 m^{-\alpha} \| (-\triangle)^{\frac{\alpha}{2}} u_v(t) \|^2.
\end{align}
 For the third term on the left-hand side of \eqref{te1},
 by \eqref{f-1} and Theorem \ref{dde-wellposed}, we have
\begin{align} \label{te3}
       & 2 \int_{\mathbb{R}^n} f(t, x, u_v(t), \mathcal{L}_{u^0(t)} ) \zeta_m^2(x) u_v(t) dx
           \nonumber \\
  \geq & - 2 \int_{\mathbb{R}^n} \zeta_m^2(x) \psi_1(t,x) (1 +
  |u_v(t)|^2
  +  \mathcal{L}_{u^0(t)}( \|\cdot\|^2) ) dx
           \nonumber \\
  \geq & - 2 (1 + \| u^0(t) \|^2) ) \int_{|x| \geq \frac{m}{2}} |\psi_1(t,x)| dx
   - 2 \int_{\mathbb{R}^n} \zeta_m^2(x) \psi_1(t,x)
  |u_v(t)|^2  dx
           \nonumber \\
  \geq & - 2 \left[ 1 + M_0( 1 + \| u_0 \|^2) \right] \int_{|x| \geq \frac{m}{2}} |\psi_1(t,x)| dx
  - 2 \|
  \psi_1(t)
  \|_{L^\infty(\mathbb{R}^n)}
    \| \zeta_m
   u_v(t)
   \|^.
\end{align}
  For the first term on the right-hand side of \eqref{te1},
 by \eqref{g-3},
 Young's inequality and Theorem \ref{dde-wellposed}, we obtain
\begin{align} \label{te4}
       & 2 \left( \zeta_m g(t, \cdot, u_v(t), \mathcal{L}_{u^0(t)} ), \zeta_m u_v(t) \right)
  \leq  \| \zeta_m g(t, \cdot, u_v(t), \mathcal{L}_{u^0(t)} ) \|^2
         + \| \zeta_m u_v(t) \|^2
            \nonumber \\
  \leq & 4 \int_{\mathbb{R}^n}
                 \zeta^2_m(x) |\psi_g(t,x)|^2
                            \left| u_v(t) \right|^2 dx
         + 4 \int_{\mathbb{R}^n}
                 \zeta^2_m(x) |\psi_g(t,x)|^2 \| u^0(t) \|^2 dx
            \nonumber \\
       & + 2 \int_{\mathbb{R}^n}
                   \zeta^2_m(x) |\psi_g(t,x)|^2 dx
         + \| \zeta_m u_v(t) \|^2
            \nonumber \\
  \leq & \left(1 + 4 \|\psi_g\|_{L^\infty(0,T; L^\infty(\mathbb{R}^n) ) }^2
         \right)
        \| \zeta_m u_v(t) \|^2
            \nonumber \\
       & + 2 \left[ 1+ 2 M_0 \left(1 + \| u_0 \|^2\right)
             \right]
           \int_{|x| \geq \frac{m}{2}} |\psi_g(t,x)|^2 dx.
\end{align}
  For the second term on the right-hand side of \eqref{te1},
 by \eqref{sigmadef} and Theorem \ref{dde-wellposed}, we obtain
\begin{align} \label{te5}
        & 2 \left( \zeta_m \sigma(t, u_v (t), \mathcal{L}_{u^0(t)}) v(t), \zeta_m u_v(t) \right) \nonumber \\
   \leq & \| \zeta_m \sigma(t, u_v (t), \mathcal{L}_{u^0(t)}) \|_{L_2(l^2, H)}^2  \| v(t) \|_{l^2}^2
          + \| \zeta_m u_v(t) \|^2 \nonumber \\
      = & \sum_{k=1}^{\infty}
                       \| \zeta_m \sigma_{1,k} (t)
                          + \zeta_m \kappa \sigma_{2,k}(t, u_v(t), \mathcal{L}_{u^0(t)} )
                       \|^2
           \| v(t) \|_{l^2}^2
          + \| \zeta_m u_v(t) \|^2 \nonumber \\
   \leq & 2 \sum_{k=1}^{\infty}
                       \| \zeta_m \sigma_{1,k} (t) \|^2  \| v(t) \|_{l^2}^2
          + \| \zeta_m u_v(t) \|^2  \nonumber \\
        & + 4 \sum_{k=1}^{\infty}
                            \left[ 2 \beta_{k}^2 \left( 1 + \mathcal{L}_{u^0(t)}( \|\cdot\|^2) \right) \| \zeta_m \kappa \|^2
                                   + \gamma_k^2
                                     \| \kappa \|_{L^\infty(\mathbb{R}^n)}^2
                                     \| \zeta_m u_v(t) \|^2
                            \right]
             \| v(t) \|_{l^2}^2
           \nonumber \\
   \leq & 2 \sum_{k=1}^{\infty}
                       \int_{|x| \geq \frac{m}{2}} |\sigma_{1,k} (t,x)|^2 dx  \| v(t) \|_{l^2}^2
          + \| \zeta_m u_v(t) \|^2  \nonumber \\
        & + 8 \left[ 1 + M_0(1 + \|u_0\|^2) \right]
              \sum_{k=1}^{\infty} \beta_k^2
                                             \int_{|x| \geq \frac{m}{2}} |\kappa(x)|^2 dx
            \| v(t) \|_{l^2}^2
             \nonumber \\
        & + 4 \| \kappa \|_{L^\infty(\mathbb{R}^n)}^2
                   \sum_{k=1}^{\infty} \gamma_k^2 \| \zeta_m u_v(t) \|^2
          \| v(t) \|_{l^2}^2.
\end{align}

 Then from \eqref{te1}-\eqref{te5}, it follows that for  $t \in [0, T]$,
\begin{align} \label{te6}
        \frac{d}{dt} \| \zeta_m u_v(t) \|^2
 \leq & c_1 C(n,\alpha) m^{-\alpha} \| u_v(t) \|^2
        + 2 c_1 m^{-\alpha} \| (-\triangle)^{\frac{\alpha}{2}} u_v(t) \|^2
            \nonumber \\
      & + \left(2 + 4 \|\psi_g\|_{L^\infty(0,T; L^\infty(\mathbb{R}^n) ) }^2
                  + 4 \| \kappa \|_{L^\infty(\mathbb{R}^n)}^2
                      \sum_{k=1}^{\infty} \gamma_k^2 \| v(t) \|_{l^2}^2
          \right)
          \| \zeta_m u_v(t) \|^2
            \nonumber \\
      & + 2 \left[ 1 + M_0( 1 + \| u_0 \|^2) \right] \int_{|x| \geq \frac{m}{2}}
      \left ( |\psi_1(t,x)|
      +
       |\psi_g(t,x)|^2
      \right ) dx
            \nonumber \\
       & + 2 \sum_{k=1}^{\infty}
                       \int_{|x| \geq \frac{m}{2}} |\sigma_{1,k} (t,x)|^2 dx  \| v(t) \|_{l^2}^2
           \nonumber \\
      & + 8 \left[ 1 + M_0(1 + \|u_0\|^2) \right]
              \sum_{k=1}^{\infty} \beta_k^2
                                             \int_{|x| \geq \frac{m}{2}} |\kappa(x)|^2 dx
            \| v(t) \|_{l^2}^2.
\end{align}

 By Gronwall's inequality and \eqref{te6}, we obtain for any $t \in [0, T]$,
\begin{align} \label{te7}
      & \| \zeta_m u_v(t) \|^2 \nonumber \\
 \leq & e^{ c_2 \int_0^T \left( 1 + \| v(r) \|_{l^2}^2 \right) dr }
        \| \zeta_m u_0 \|^2
        + c_1 C(n,\alpha) m^{-\alpha} \int_0^T e^{ c_2 \int_s^T \left( 1 + \| v(r) \|_{l^2}^2 \right) dr } \| u_v(s) \|^2 ds
          \nonumber \\
      & + 2 c_1 m^{-\alpha} \int_0^T e^{ c_2 \int_s^T \left( 1 + \| v(r) \|_{l^2}^2 \right) dr } \| (-\triangle)^{\frac{\alpha}{2}} u_v(s) \|^2 ds
            \nonumber \\
      & + 2 \left[ 1 + 2M_0( 1 + \| u_0 \|^2) \right]
               \int_0^T e^{ c_2 \int_s^T \left( 1 + \| v(r) \|_{l^2}^2 \right) dr }
                       \int_{|x| \geq \frac{m}{2}}
                                \left( |\psi_1(s,x)| + |\psi_g(s,x)|^2 \right)
                       dx
             ds
          \nonumber \\
      & + 2 \int_0^T
      \left (
       e^{ c_2
         \int_s^T \left( 1 + \| v(r) \|_{l^2}^2 \right) dr }
                     \| v(s) \|_{l^2}^2
                       \sum_{k=1}^{\infty}
                         \int_{|x| \geq \frac{m}{2}} |\sigma_{1,k} (s,x)|^2 dx
           \right )  ds
           \nonumber \\
      & + 8 \left[ 1 + M_0(1 + \|u_0\|^2) \right]
          \sum_{k=1}^{\infty} \beta_k^2
             \int_{|x| \geq \frac{m}{2}} |\kappa(x)|^2 dx
                   \int_0^T e^{ c_2 \int_s^T \left( 1 + \| \theta(r) \|_{l^2}^2 \right) dr }
                            \| v(s) \|_{l^2}^2
                   ds,
\end{align}
 where
      $ c_2 = \left( 2 + 4 \|\psi_g\|_{L^\infty(0,T; L^\infty(\mathbb{R}^n) ) }^2
              \right)
              + 4 \| \kappa \|_{L^\infty(\mathbb{R}^n)}^2
                       \sum_{k=1}^{\infty} \gamma_k^2.
      $
For the first term on the right-hand side of \eqref{te7},
 we have
\begin{align} \label{te8}
      e^{ c_2 \int_0^T \left( 1 + \| v(r) \|_{l^2}^2 \right) dr }
      \| \zeta_m u_0 \|^2
 \leq e^{ c_2 (T + R^2)}
        \int_{|x| \geq \frac{m}{2}} | u_0(x) |^2
  \rightarrow  \ 0, \ \ \   \text{as} \ m \to +\infty.
 \end{align}

 For the second and third terms on the right-hand side of \eqref{te7},
 by Theorem \ref{cewellposed}, we obtain
\begin{align} \label{te9}
      & c_1 C(n,\alpha) m^{-\alpha} \int_0^T e^{ c_2 \int_s^T \left( 1 + \| v(r) \|_{l^2}^2 \right) dr } \| u_v(s) \|^2 ds
          \nonumber \\
      & + 2 c_1 m^{-\alpha} \int_0^T e^{ c_2 \int_s^T \left( 1 + \| v(r) \|_{l^2}^2 \right) dr } \| (-\triangle)^{\frac{\alpha}{2}} u_v(s) \|^2 ds
            \nonumber \\
 \leq & c_1 m^{-\alpha}  e^{c_2 (T+ R^2)} \left( C(n,\alpha) \vee 2 \right)
            \int_0^T \left( \| u_v(s) \|^2 + \| (-\triangle)^{\frac{\alpha}{2}} u_v(s) \|^2 \right) ds
            \nonumber \\
 \leq & c_1 m^{-\alpha}  e^{c_2 (T+ R^2)} \left( C(n,\alpha) \vee 2  \right) M_2
            \nonumber \\
 \rightarrow & \ 0,  \ \ \   \text{as} \ m \to +\infty.
\end{align}

 For the fourth term on the right-hand side of \eqref{te7},
 by $\psi_1 \in L_{loc}^1(\mathbb{R}, L^1(\mathbb{R}^n))$
 and $\psi_g \in L_{loc}^2(\mathbb{R}, L^2(\mathbb{R}^n))$, we obtain
\begin{align} \label{te10}
       & 2 \left[ 1 + 2M_0( 1 + \| u_0 \|^2) \right]
               \int_0^T e^{ c_2 \int_s^T \left( 1 + \| v(r) \|_{l^2}^2 \right) dr }
                       \int_{|x| \geq \frac{m}{2}}
                                \left( |\psi_1(s,x)| + |\psi_g(s,x)|^2 \right)
                       dx
               ds
             \nonumber \\
  \leq &  2 \left[ 1 + 2M_0( 1 + \| u_0 \|^2) \right]
           e^{ c_2 \left( T + R^2 \right) }
               \int_0^T
                        \int_{|x| \geq \frac{m}{2}}
                                \left( |\psi_1(s,x)| + |\psi_g(s,x)|^2 \right)
                        dx
                ds
          \nonumber \\
 \rightarrow & \ 0,  \ \ \   \text{as} \ m \to +\infty.
\end{align}

 For the fifth term on the right-hand side of \eqref{te7}, we have
\begin{align} \label{te11}
      & 2 \int_0^T
      \left (
      e^{ c_2 \int_s^T \left( 1 + \| v(r) \|_{l^2}^2 \right) dr }
                     \| v(s) \|_{l^2}^2
                       \sum_{k=1}^{\infty}
                         \int_{|x| \geq \frac{m}{2}} |\sigma_{1,k} (s,x)|^2 dx
                         \right )
          ds
          \nonumber \\
 \leq & 2 e^{ c_2 \left( T + R^2 \right) }
             \int_0^T
             \left (  \| v(s) \|_{l^2}^2
                       \sum_{k=1}^{\infty}
                          \int_{|x| \geq \frac{m}{2}} |\sigma_{1,k} (s,x)|^2 dx
                          \right )
             ds.
\end{align}

 Since $\sigma_{1} :[0,T]\rightarrow L^2(\mathbb{R}^n, l^2)$ is continuous due to   $({\bf \Sigma 3})$,
 the set $\{ \sigma_{1} (s): s\in [0,T] \}$ is a compact subset of $L^2(\mathbb{R}^n, l^2)$,
 which implies that for any $\delta  >0$,
 there exists $k \in \mathbb{N}$ and $s_1, \cdots, s_k \in [0,T]$
 such that the open balls centered at $\sigma_1(s_i) \in L^2(\mathbb{R}^n, l^2)$ $(i=1,\cdots,k)$ with radius
 $\frac{\delta }{2}$ form a finite open cover of the set
 $\{ \sigma_{1} (s): s \in [0,T] \}$ in $L^2(\mathbb{R}^n, l^2)$.
 Therefore, we infer that there exists $m_0>0$ such that for every $m \geq m_0$ and $s \in [0,T]$
\begin{align} \label{te12}
       \left( \int_{|x| \geq \frac{m}{2}}
                 \sum_{k=1}^{\infty}
                    | \sigma_{1,k}(s,x) |^2 dx
       \right)^{\frac{1}{2}}
       < \delta.
\end{align}
By \eqref{te11} and \eqref{te12}, we obtain
\begin{align} \label{te13}
      & 2 \int_0^T e^{ c_2 \int_s^T \left( 1 + \| v(r) \|_{l^2}^2 \right) dr }
                     \| v(s) \|_{l^2}^2
                       \sum_{k=1}^{\infty}
                         \int_{|x| \geq \frac{m}{2}} |\sigma_{1,k} (s,x)|^2 dx
          ds
          \nonumber \\
 \leq & 2 \delta  e^{ c_2 \left( T + R^2 \right) }
                       \int_0^T \| v(s) \|_{l^2}^2 ds
 \leq  2 \delta  R^2  e^{ c_2 \left( T + R^2 \right) }.
\end{align}
 For the sixth term on the right-hand side of \eqref{te7},
 by $\kappa \in L^2(\mathbb{R}^n)$  we get
\begin{align} \label{te14}
       & 8 \left[ 1 + M_0(1 + \|u_0\|^2) \right]
              \sum_{k=1}^{\infty} \beta_k^2
                   \int_{|x| \geq \frac{m}{2}} |\kappa(x)|^2 dx
                         \int_0^T e^{ c_2 \int_s^T \left( 1 + \| \theta(r) \|_{l^2}^2 \right) dr }
                                  \| v(s) \|_{l^2}^2
                         ds
                           \nonumber \\
  \leq & 8 \left[ 1 + M_0(1 + \|u_0\|^2) \right]
           R^2 e^{ c_2 \left(T + R^2\right) }
               \sum_{k=1}^{\infty} \beta_k^2
                    \int_{|x| \geq \frac{m}{2}} |\kappa(x)|^2 dx
                            \nonumber \\
  \rightarrow & \ 0, \ \ \   \text{as} \ m \to +\infty.
\end{align}

 It follows from \eqref{te7}-\eqref{te10} and \eqref{te13}-\eqref{te14} that there exists
     $m_1 = m_1 (u_0, T, R,  \delta) > 0$
 such that for all $m \geq m_1$ and $t \in [0, T]$,
\begin{align*}
      \int_{|x| \geq m} |u_v(t,x)|^2 dx < c_2 \delta,
\end{align*}
where $c_2=c_2(u_0, T, R)>0$
is independent of $\delta$,
as desired.
\end{proof}

\section{Large deviations of    McKean-Vlasov stochastic equations}

 In this section,
 we   prove
 the LDP  of
 the  solutions
 of the
 fractional McKean-Vlasov stochastic reaction-diffusion equation
   \eqref{mvsde-1}-\eqref{mvsde-1 IV}
 on $\mathbb{R}^n$
 by the weak convergence method.

\subsection{Weak convergence theory
    of      large deviations}

 We now review the  basic concepts and results
 of the weak converge theory
 for the LDP.

 Let $\mathbb{S}$ be a polish space.
 For any $\varepsilon>0$,
 let $\mathcal{G}^\varepsilon:
                C([0,T],U)\rightarrow\mathbb{S}$
 be a measurable map.
 Denote by
    $ X^\varepsilon = \mathcal{G}^\varepsilon(W) $
 for all $\varepsilon > 0$.

\begin{definition}
 A function $I: \mathbb{S} \rightarrow [0,\infty]$ is called a rate function on $\mathbb{S}$
 if it is lower semicontinuous in $\mathbb{S}$.
 A rate function $I$ on $\mathbb{S}$ is said to be a good rate function on $\mathbb{S}$
 if for every $0\leq C<\infty$,
 the level set $\{ x \in \mathbb{S}: I(x) \leq C \}$ is a compact subset of $\mathbb{S}$.
\end{definition}

\begin{definition}
  The family $\{X^\varepsilon\}$ is said to satisfy the LDP
   in $\mathbb{S}$
  with a rate function $I:\mathbb{S}\rightarrow[0,\infty]$,
  if for every Borel subset $A$ of $\mathbb{S}$,
  \begin{align*}
   -\inf_{x \in A^\circ} I(x)
   \leq \liminf_{\varepsilon \rightarrow 0} \varepsilon \log \mathbb{P}(X^\varepsilon \in A)
   \leq \limsup_{\varepsilon \rightarrow 0} \varepsilon \log \mathbb{P}(X^\varepsilon \in A)
   \leq -\inf_{x \in \overline{A}} I(x),
 \end{align*}
 where $A^\circ$ and $\overline{A}$ are the interior and the closure of $A$ in $\mathbb{S}$, respectively.
\end{definition}

Given $N > 0$, denote by
\begin{equation*}
    S_N = \left\{v \in L^2(0,T; l^2): \int_{0}^T \|v(t)\|^2_{l^2} dt \leq N \right\}.
\end{equation*}
 Then $S_N$ is a polish space endowed with the weak topology.
 Let $\mathcal{A}$ be the space of all $l^2$-valued stochastic processes $v$
 which are progressively measurable with respect to $\{\mathcal{F}_t\}_{t\in [0,T]}$
 and $\int_{0}^{T} \|v(t)\|_{l^2}^2 dt < \infty$ $\mathbb{P}$-almost surely.
 Denote by
   $\mathcal{A}_N = \{ v \in \mathcal{A} : v(\omega) \in S_N\
                         \text{for almost all }\ \omega\in \Omega \}$.

 In order to prove the LDP
  of $X^\varepsilon$,
 we will assume that the family
    $\left\{\mathcal{G}^\varepsilon\right\}$
 fulfills the following conditions:
 there exists a measurable map $\mathcal{G}^0: C([0,T],U) \rightarrow \mathbb{S}$
 such that:

 $({\bf H1})$
    For every $N < \infty$ and
    sequence $\{v_i : i \in \mathbb{N} \} \subseteq S_N$,
    if   $v_i$ converges to
     $v$ in $ S_N$ as $\varepsilon \rightarrow 0$, then
       $\mathcal{G}^0 \big( \int_0^\cdot v_i(s) ds \big)$
    converges to
       $\mathcal{G}^0 \big( \int_0^\cdot v(s) ds \big)$
    in  $\mathbb{S}$.

 $({\bf H2})$
    For every $N < \infty$, $\{v^\varepsilon : \varepsilon > 0\} \subseteq \mathcal{A}^N$ and $\delta > 0$,
    $$ \lim_{\varepsilon \to 0}
             \mathbb{P}
                  \left( \rho
                             \Big( \mathcal{G}^\varepsilon
                                          \big( W + \varepsilon^{-\frac{1}{2}} \int_{0}^{\cdot} v^\varepsilon(t) dt
                                          \big), \
                                    \mathcal{G}^0
                                          \big( \int_{0}^{\cdot} v^\varepsilon(t) dt
                                          \big)
                             \Big) > \delta
                  \right)
        = 0,
    $$
 where $\rho(\cdot, \cdot)$ stands for the metric in  $\mathbb{S}$.

 We recall the following theorem
 from  \cite[Theorem 3.2]{msz2021AMO} and \cite[Theorem 4.4]{lszz2022PA}
 which is essentially   established by Budhiraja and Dupuis \cite[Theorem 4.4]{bd2000PMS}.

\begin{prop}\cite{msz2021AMO, lszz2022PA}\label{ldpth}
  Under $({\bf H1})$-$({\bf H2})$,
  the family $\{X^\varepsilon\}$ satisfies the LDP  in $\mathbb{S}$
  with rate function $I$ given by,
  \begin{align*}
       I(x) = \inf_{v \in L^2(0,T; l^2): x = \mathcal{G}^0 \big( \int_{0}^{\cdot} v(t)dt \big) }
                   \left\{ \frac{1}{2} \int_{0}^{T} \|v(t)\|_{l^2}^2 dt
                   \right\}, \ \ \   x \in \mathbb{S},
  \end{align*}
  with  convention   $\inf \{ \emptyset \} = \infty$.
\end{prop}

 \subsection{Main results}

 In this subsection,
 we present the   result
 on   the LDP
 of  \eqref{mvsde-2}-\eqref{mvsde-2 IV} in
 $ C([0,T], H)  \bigcap  L^2(0,T; V) $
 as follows.

 \begin{theorem}\label{main}
    Suppose that $({\bf \Sigma 1})$-$({\bf \Sigma 3})$ hold,
    and $u^\varepsilon$ is the solution of \eqref{mvsde-2}-\eqref{mvsde-2 IV}.
    Then as $\varepsilon \rightarrow 0$, the family $\{ u^\varepsilon \}$
    satisfies the LDP in
    $ C([0,T], H)  \bigcap  L^2(0,T; V) $
    with   good rate function given by
    \begin{align}\label{rate}
        I(\phi)
        = \inf
        \left\{ \frac{1}{2} \int_{0}^{T}\| v(t) \|_{l^2}^2dt :
        v \in L^2(0,T; l^2),\  u_v = \phi
        \right\},
    \end{align}
    where
    $\phi \in C([0,T], H) \bigcap L^2(0,T; V)$,
    and $u_v$ is the solution to   \eqref{contequ}-\eqref{contequ IV}.
    As usual, the infimum of the empty set is taken to be $\infty$.
 \end{theorem}

 To prove
 Theorem \ref{main}, we will
 verify the conditions of Proposition \ref{ldpth},
 for which we need the following lemma.

 \begin{lemma} \label{measmap}
    Suppose $({\bf \Sigma 1})$-$({\bf \Sigma 3})$ hold.
    Then for any $\varepsilon \in (0,1)$ and $u_0 \in H$,
    there exists a Borel-measurable map
    {\small
        $$
        \mathcal{G}^\varepsilon{\large }:
        \    C([0,T],U)
        \rightarrow
        L^2(\Omega, C([0,T], H) )
        \bigcap
        L^2(\Omega, L^2(0,T; V) )
        \bigcap
        L^p(\Omega, L^p(0,T; L^p(\mathbb{R}^n ) ) )
        $$
    }
    such that
    $ \mathcal{G}^{\varepsilon} (W)= u^\varepsilon$
    almost surely,
    where $u^\varepsilon$ is the solution to
    \eqref{mvsde-2}-\eqref{mvsde-2 IV}.

    Furthermore,    for every control $v \in L^2(0,T; l^2) $,
    $u_v^\varepsilon
    = \mathcal{G}^\varepsilon \left( W + \varepsilon^{-\frac{1}{2}} \int_{0}^{\cdot} v(t) dt \right)$
    is the unique solution of
    \begin{equation}\label{sce}
        \left\{
        \begin{aligned}
            & d u_v^\varepsilon (t)
            + (-\triangle)^\alpha u_v^\varepsilon (t) dt
            + f(t, \cdot, u_v^\varepsilon(t), \mathcal{L}_{u^\varepsilon(t)} ) dt
            \\
            & = g(t, \cdot, u_v^\varepsilon(t), \mathcal{L}_{u^\varepsilon(t)} ) dt
            + \sigma(t, u_v^\varepsilon (t), \mathcal{L}_{u^\varepsilon(t)} ) v(t) dt
            + \sqrt{\varepsilon} \sigma(t, u_v^\varepsilon(t), \mathcal{L}_{u^\varepsilon(t)} ) dW(t),
            \\
            & u_v^\varepsilon (0) = u_0 \in H.
        \end{aligned}
        \right.
    \end{equation}
 \end{lemma}

 \begin{proof}
    For a fixed $\mu\in
    C([0,T], \mathcal{P}_2(H))$,
    $\varepsilon \in (0,1)$
    and $u_0\in H$,
    consider the
    decoupled stochastic equation:
    \begin{align} \label{dcsde-1}
        & d y_\mu ^\varepsilon(t) + (-\triangle)^\alpha y_\mu^\varepsilon(t) dt
        + f(t, \cdot,  y_\mu^\varepsilon (t), \mu (t) ) dt
        \nonumber \\
        = & g(t, \cdot, y_\mu^\varepsilon (t), \mu (t) ) dt
        + \sqrt{\varepsilon}
        \sigma(t, y_\mu^\varepsilon (t), \mu (t) ) dW(t),
    \end{align}
    with initial data
    \begin{align} \label{dcsde-1 IV}
        y_\mu ^\varepsilon(0) = u_0.
    \end{align}

    Since
    $\mu\in
    C([0,T], \mathcal{P}_2(H))$
is fixed, problem
    \eqref{dcsde-1}-\eqref{dcsde-1 IV}
    is similar to the classical
    stochastic equation independent of
    distributions as discussed in
    \cite{wangb2019JDE}, which means that
    under conditions
    $({\bf \Sigma 1})$-$({\bf \Sigma 3})$,
    for every
    $\varepsilon \in (0,1)$
    and $u_0\in H$, problem
    problem
    \eqref{dcsde-1}-\eqref{dcsde-1 IV}
    has a unique solution
    $$y_\mu^\varepsilon
    \in  L^2(\Omega, C([0,T], H) )
    \bigcap
    L^2(\Omega, L^2(0,T; V) )
    \bigcap
    L^p(\Omega, L^p(0,T; L^p(\mathbb{R}^n ) ) ).
    $$
    Consequently,
    there exists   a Borel-measurable map
    (see, e.g.,
    \cite[Theorem 3.6]{lszz2022PA}):
    {\small
        $$  \mathcal{G}
        _ \mu ^\varepsilon : \   C([0,T],U)
        \rightarrow
        L^2(\Omega, C([0,T], H) )
        \bigcap
        L^2(\Omega, L^2(0,T; V) )
        \bigcap
        L^p(\Omega, L^p(0,T; L^p(\mathbb{R}^n ) ) )
        $$
    }
    such that
    $y^\varepsilon_\mu=
    \mathcal{G}
    _ \mu ^\varepsilon (W)$.

    Let
    $u^\varepsilon$
    be the solution of
    \eqref{mvsde-2}-\eqref{mvsde-2 IV}
    and   $\mu^\varepsilon(t)  = \mathcal{L}_{u^\varepsilon(t)}$
    for $t\in [0,T]$.
    Then by Theorem \ref{mvsde-wellposed},
    we find that
    $u^\varepsilon
    =y^\varepsilon_{\mu^\varepsilon}$
    and hence
    $
    u^\varepsilon
    =  \mathcal{G}^\varepsilon
    _ {\mu ^\varepsilon} (W)
    $.
    Denote by
    $\mathcal{G}^\varepsilon
    =  \mathcal{G}^\varepsilon
    _ {\mu ^\varepsilon}$.
    Then
    $u^\varepsilon
    =  \mathcal{G}^\varepsilon
    (W)$, and  for every control $v \in L^2(0,T; l^2) $,
    $u_v^\varepsilon
    = \mathcal{G}^\varepsilon \left( W + \varepsilon^{-\frac{1}{2}} \int_{0}^{\cdot} v(t) dt \right)$
    is the unique solution of
    \begin{equation*}
        \left\{
        \begin{aligned}
            & d u_v^\varepsilon (t)
            + (-\triangle)^\alpha u_v^\varepsilon (t) dt
            + f(t, \cdot, u_v^\varepsilon(t), \mathcal{L}_{u^\varepsilon(t)} ) dt
            \\
            & = g(t, \cdot, u_v^\varepsilon(t), \mathcal{L}_{u^\varepsilon(t)} ) dt
            + \sigma(t, u_v^\varepsilon (t), \mathcal{L}_{u^\varepsilon(t)} ) v(t) dt
            + \sqrt{\varepsilon} \sigma(t, u_v^\varepsilon(t), \mathcal{L}_{u^\varepsilon(t)} ) dW(t),
            \\
            & u_v^\varepsilon (0) = u_0 \in H,
        \end{aligned}
        \right.
    \end{equation*}
    which concludes the
    proof.
 \end{proof}


 Now, define
 $\mathcal{G}^0: C([0,T],U)
 \rightarrow
 C([0,T], H) \bigcap L^2(0,T; V)
 \bigcap
 L^p(\Omega, L^p(0,T; L^p(\mathbb{R}^n ) ) )$
 by, for every $\zeta\in C([0,T],U)$,
 \begin{equation}\label{G0}
    \mathcal{G}^0(\zeta)
    = \left\{
    \begin{array} {ll}
        & u_v, \ \ \  \text{if}\  \zeta = \int_{0}^{\cdot} v(t) dt \ \text{for\ some}\  v \in L^2(0,T; l^2),\\
        & 0,\ \ \  \text{otherwise},
    \end{array}
    \right.
 \end{equation}
 where $u_v$ is the solution to \eqref{contequ}-\eqref{contequ IV}.

 In the next subsection, we will
 show  $\mathcal{G}^0$
 satisfies $({\bf H1})$ in Proposition
 \ref{ldpth}.

 \subsection{Convergence of solutions of controlled equations}

 We now prove that    $\mathcal{G}^0$
 satisfies
 the condition  $({\bf H1})$; more precisely, we will show  the
 convergence of  solutions
 of the controlled equation
 \eqref{contequ}
 in the space
 $C([0,T], H) \bigcap L^2(0,T; V)$
 with respect to controls in the weak
 topology  of $L^2(0,T;l^2)$.
 The argument is similar to
 the case of stochastic
 partial differential equations
 which are independent of distributions.
 For the reader's convenience, we   sketch the proof with
 necessary
 modifications.

 \begin{theorem} \label{h1}
    Suppose that $({\bf \Sigma 1})$-$({\bf \Sigma 3})$ hold,
    $u_0 \in H$, and
    $v, v_i \in L^2(0, T; l^2)$ for $i \in \mathbb{N}$.
    If $v_i \rightarrow v$ weakly in $L^2(0, T; l^2)$,
    then
    $\mathcal{G}^0 \big( \int_0^\cdot v_i(s) ds \big)$
    converges to
    $\mathcal{G}^0 \big( \int_0^\cdot v(s) ds \big)$
    strongly
    in  $C([0,T], H) \bigcap L^2(0,T; V)$.
 \end{theorem}

 \begin{proof}
    Let  $u_v$ and $u_{v_i}$
    be the solutions of \eqref{contequ}-\eqref{contequ IV} corresponding to
    controls $v$ and $v_i$, respectively.
    By \eqref{G0}, we   have
    $u_v = \mathcal{G}^0 \big( \int_0^\cdot v(s) ds \big)$
    and
    $u_{v_i} = \mathcal{G}^0 \big( \int_0^\cdot v_i(s) ds \big)$.
    Therefore, we only need to  prove $u_{v_i}$ converges to $u_v$ strongly in $C([0,T], H) \bigcap L^2(0,T; V)$.

    {\bf Step 1:}   equicontinuity of $\{ u_{v_i} \}_{i=1}^\infty$
    in $\left( V \bigcap L^p(\mathbb{R}^n) \right)^\ast$.
    Since $v_i \rightarrow v$ weakly in $L^2(0, T; l^2)$,
    we know  that $\{ v_i \}_{i=1}^\infty$ is bounded in $L^2(0,T; l^2)$.
    By Theorem \ref{cewellposed}, there exists $C_1 = C_1(T) > 0$ such
    that for all $i \in \mathbb{N}$,
    \begin{align}\label{sc-1}
        \| u_{v_i} \|_{C([0,T],H)}
        + \| u_{v_i} \|_{L^2(0,T; V)}
        + \| u_{v_i} \|_{L^p(0,T;L^p(\mathbb{R}^n))}
        + \| v_i \|_{L^2(0,T; l^2)}
        \leq C_1.
    \end{align}
    By \eqref{f-4},
    \eqref{sc-1} and Theorem \ref{dde-wellposed},
    one
    find that there exists $C_2 = C_2(T) > 0$ such that
    \begin{align} \label{sc-2}
        & \| f(\cdot, \cdot, u_{v_i}, \mathcal{L}_{u^0(\cdot)} ) \|_{L^q(0,T;L^q(\mathbb{R}^n) ) }^q
        \nonumber \\
        = & \int_0^T \int_{\mathbb{R}^n}
        \left| f(t,x, u_{v_i}(t), \mathcal{L}_{u^0(t)} )
        \right|^q dx dt
        \nonumber \\
        \leq & \int_0^T \int_{\mathbb{R}^n}
        \left[ \lambda_3 |u_{v_i}(t)|^{p-1} + \psi_3(t,x) \left( 1 + \|u^0(t)\| \right)
        \right]^q dx dt
        \nonumber \\
        \leq & \int_0^T \int_{\mathbb{R}^n}
        \left\{ \lambda_3 |u_{v_i}(t)|^{p-1}
        + \psi_3(t,x)
        \left[ 1 + M_0^{\frac 12} (1 + \|u_0\|)
        \right]
        \right\}^q dx dt
        \nonumber \\
        \leq & 2^{q-1} \lambda_3^q
        \int_0^T \int_{\mathbb{R}^n} |u_{v_i}(t)|^{q(p-1)} dx dt
        + 2^{q-1}
        \left[ 1 + M_0^{\frac 12}(1 + \|u_0\|)
        \right]^q
        \int_0^T \int_{\mathbb{R}^n} |\psi_3(t,x)|^q dx dt
        \nonumber \\
        \leq & 2^{q-1} \lambda_3^q \|u_{v_i}\|_{L^p(0,T;L^p(\mathbb{R}^n) ) }^p
        + 2^{q-1}
        \left[ 1 + M_0^{\frac 12}(1 + \|u_0\|)
        \right]^q
        \int_0^T \int_{\mathbb{R}^n} |\psi_3(t,x)|^q dx dt
        \le  C_2,
    \end{align}
    where $C_2=C_2(T)>0$.
    Similar to \eqref{sfnsol-u},
    by \eqref{contequ}-\eqref{contequ IV}
    we have
    \begin{align} \label{sc-3}
        & u_{v_i} (t)
        + \int_0^t (-\triangle)^{\alpha} u_{v_i}(s) ds
        + \int_0^t f( s, \cdot, u_{v_i}(s), \mathcal{L}_{u^0(s)} ) ds
        \nonumber \\
        = & u_0
        + \int_0^t g(s, \cdot, u_{v_i}(s), \mathcal{L}_{u^0(s)} ) ds
        + \int_0^t \sigma(s, u_{v_i}(s), \mathcal{L}_{u^0(s)} ) v_i(s) ds
    \end{align}
    in $  \left( V \bigcap L^p(\mathbb{R}^n) \right)^\ast$.
    Then for all $0 \leq s \leq t \leq T$,
    \begin{align} \label{sc-4}
        & \| u_{v_i} (t) - u_{v_i} (s) \|_{\left( V \bigcap L^p(\mathbb{R}^n) \right)^\ast}
        \nonumber \\
        \leq & \int_s^t \| (-\triangle)^{\alpha} u_{\theta_i}(r) \|_{V^\ast} dr
        + \int_s^t \| f( r, \cdot, u_{v_i}(r), \mathcal{L}_{u^0(r)} ) \|_{L^q(\mathbb{R}^n)} dr
        \nonumber \\
        & + \int_s^t \| g(r, \cdot, u_{v_i}(r), \mathcal{L}_{u^0(r)} ) \| dr
        + \int_s^t \| \sigma(r, u_{v_i}(r), \mathcal{L}_{u^0(r)} ) v_i(r) \| dr.
    \end{align}

    For the first term on the right-hand side of \eqref{sc-4},
    by
    \eqref{sc-1}
    and the H\"{o}lder inequality,
    we obtain
    \begin{align} \label{sc-5}
        \int_s^t \| (-\triangle)^{\alpha} u_{v_i}(r) \|_{V^\ast} dr
        \leq \int_s^t \| u_{v_i}(r) \|_{V} dr
        \leq C_1 (t-s)^{\frac{1}{2}}.
    \end{align}

    For the second term on the right-hand side of \eqref{sc-4},
    by the H\"{o}lder inequality and \eqref{sc-2}, we have
    \begin{align} \label{sc-6}
        \int_s^t \| f( r, \cdot, u_{v_i}(r), \mathcal{L}_{u^0(r)} ) \|_{L^q(\mathbb{R}^n)} dr
        \leq (t-s)^{\frac{1}{p}} \| f( \cdot, \cdot, u_{v_i}, \mathcal{L}_{u^0(\cdot)} ) \|_{ L^q( 0, T; L^q(\mathbb{R}^n) ) }
        \leq C_2^{\frac{1}{q}} (t-s)^{\frac{1}{p}}.
    \end{align}

    For the third term on the right-hand side of \eqref{sc-4},
    by \eqref{g-3}  and Theorem \ref{dde-wellposed}, we obtain
    \begin{align} \label{sc-7}
        & \int_s^t \| g(r, \cdot, u_{v_i}(r), \mathcal{L}_{u^0(r)} ) \| dr
        \nonumber \\
        \leq & (t-s)^{\frac{1}{2}} \| g(\cdot, \cdot, u_{v_i}, \mathcal{L}_{u^0(\cdot)} ) \|_{ L^2(0, T; H) }
        \nonumber \\
        \leq & (t-s)^{\frac{1}{2}}
        \left\{
        \int_0^T
        \left[ 4 \| \psi_g(r) \|_{L^\infty(\mathbb{R}^n) }^2
        \|  u_{v_i} (r) \|^2 +
        4 M_0( 1 + \|u_0\|^2 )
        \| \psi_g(r)\|^2
        + 2 \| \psi_g(r) \|^2
        \right]
        dr
        \right\}^{\frac{1}{2}}
        \nonumber \\
        \leq & (t-s)^{\frac{1}{2}}
        \left\{ 2 \sqrt{T} C_1 \| \psi_g \|_{L^\infty(0,T; L^\infty(\mathbb{R}^n) ) }
        + (2 \sqrt{M_0( 1 + \|u_0\|^2 )}
        + \sqrt{2}) \| \psi_g \|_{L^2(0,T; l^2(\mathbb{R}^n) ) }
        \right\}.
    \end{align}

    For the fourth term on the right-hand side of \eqref{sc-4},
    by \eqref{sigma3}, \eqref{sc-1} and Theorem \ref{dde-wellposed}, we obtain
    \begin{align} \label{sc-8}
        & \int_s^t \| \sigma(r, u_{v_i}(r), \mathcal{L}_{u^0(r)} ) v_i(r) \| dr  \nonumber \\
        \leq & \int_s^t \| \sigma(r, u_{v_i}(r), \mathcal{L}_{u^0(r)} ) \|_{L_2(l^2, H)}  \| v_i(r) \|_{l^2} dr  \nonumber \\
        \leq & \left( \int_s^t \| \sigma(r, u_{v_i}(r), \mathcal{L}_{u^0(r)} ) \|_{L_2(l^2, H)}^2 dr \right)^{\frac{1}{2}}
        \| v_i \|_{L^2(0,T; l^2)}  \nonumber \\
        \leq & \left[ \int_s^t M_{\sigma, T}
        \left( 1 + \| u_{v_i}(r)\|^2
        + \mathcal{L}_{u^0(r)}(\|\cdot\|^2)
        \right)
        dr
        \right]^{\frac{1}{2}}
        \| v_i \|_{L^2(0,T; l^2)}  \nonumber \\
        \leq & \left\{ M_{\sigma, T} \left[ 1 + C_1^2 + M_0 \left( 1 + \|u_0\|^2 \right) \right] \right\}^{\frac{1}{2}}
        C_1 |t-s|^{\frac{1}{2}}.
    \end{align}
    Then from \eqref{sc-4}-\eqref{sc-8},
    it follows that
    \begin{align} \label{sc-9}
        \| u_{v_i} (t) - u_{v_i} (s) \|_{\left( V \bigcap L^p(\mathbb{R}^n) \right)^\ast}
        \leq C_3 (t-s)^{\frac{1}{p}},
    \end{align}
    where $C_3 = C_3(T) > 0$.

    From \eqref{sc-1} and \eqref{sc-2},
    we   know that there exist
    $$
    z \in L^\infty(0,T; H) \bigcap L^2(0,T; V) \bigcap L^p(0,T; L^p(\mathbb{R}^n) ),\
    \varphi \in L^q(0,T; L^q(\mathbb{R}^n) ),
    $$
    and a  subsequence of $\{ u_{v_i} \}_{i=1}^\infty$ (not relabeled)
    such that
    \begin{align}
        & u_{v_i} \rightarrow z \  \text{weak-star in} \  L^\infty(0,T; H),  \label{uws} \\
        & u_{v_i} \rightarrow z \  \text{weakly in} \  L^2(0,T; V),  \label{uw} \\
        & u_{v_i} \rightarrow z \  \text{weakly in} \  L^p(0,T; L^p(\mathbb{R}^n) ),  \label{uwlp} \\
        & f(\cdot, \cdot, u_{v_i}, \mathcal{L}_{u^0(\cdot)} ) \rightarrow \varphi \  \text{weakly in} \  L^q(0,T; L^q(\mathbb{R}^n) ).  \label{fw}
    \end{align}

    Next, we prove the strong convergence
    of $\{u_{v_i}\}_{i=1}^\infty$ in
    $ L^2(0,T; H)$.

    {\bf Step 2:}  proving
    \begin{align} \label{sc-10}
        u_{v_i} \rightarrow  z  \ \ \  \text{strongly in} \ \  L^2(0,T; H).
    \end{align}
    We  firstly prove
    $u_{v_i} \rightarrow z \
    \text{strongly in} \  C \left([0,T], \left( V \bigcap L^p(\mathbb{R}^n) \right)^\ast \right).$
    By Step 1, we know
    $\{ u_{v_i} \}_{i=1}^\infty$
    is equicontinuous in
    $\left( V \bigcap L^p(\mathbb{R}^n) \right)^\ast$
    on $[0,T]$.
    So by the Arzel\`{a}-Ascoli Theorem,
    we only need to  show that
    for each $t \in [0,T]$,
    $\{ u_{v_i}(t) \}_{i=1}^\infty$
    is precompact in
    $\left( V \bigcap L^p(\mathbb{R}^n) \right)^\ast$.

    Let $\zeta_m$  be the cut-off function
    as introduced in the proof of Lemma
    \ref{tailestim}.
    Given $t \in [0,T]$, let
    $\widetilde{u}_{v_i}^m (t,x)
    = \zeta_m(x) u_{v_i}(t,x)$
    and
    $\widehat{u}_{v_i}^m(t,x)
    = (1 - \zeta_m(x) ) u_{v_i}(t,x)$
    for $x\in \mathbb{R}^n$.
    Then
    $u_{v_i}(t,x)
    = \widetilde{u}_{v_i}^m(t,x)
    + \widehat{u}_{v_i}^m(t,x).$
    for every $m\in \mathbb{N}$.
    Moreover, by \eqref{sc-1}, we have
    \begin{align} \label{sc-10a}
        \| \widehat{u}_{v_i}^m(t) \|
        \leq \| u_{v_i} \|_{C([0,T], H)}
        \leq C_1, \ \ \  \forall\ i,m \in \mathbb{N}, \ t\in [0,T].
    \end{align}
    Notice that
    $\widehat{u}_{v_i}^m(t,x) = 0$  for $|x| \geq m$ and   $t\in [0,T]$.
    By  \eqref{sc-10a} and the
    compactness of the embedding
    $H^\alpha (|x|<m)
    \hookrightarrow L^2(|x|<m)$,
    we infer that
    $\{ \widehat{u}_{v_i}^m(t) \}_{i=1}^\infty$
    is precompact in
    $\left( V \bigcap L^p(\mathbb{R}^n) \right)^\ast$.

    On the other hand, by Lemma \ref{tailestim} and \eqref{sc-1},
    we find that for every $\delta>0$,
    there exists
    $m_0 = m_0 (u_0, T,  \delta ) > 0$
    such that for all $m \geq m_0$, $i \in \mathbb{N}$ and $t \in [0,T]$,
    \begin{align*}
        \| \widetilde{u}_{v_i}^m(t)
        \|_{\left( V \bigcap L^p(\mathbb{R}^n) \right)^\ast}^2
        \leq
        \| \widetilde{u}_{v_i}^m(t)
        \|_{H}^2
        \leq \int_{|x| \geq {\frac 12}  m}
        |u_{v_i}(t,x)|^2
        dx
        < \frac{1}{4}  \delta^2.
    \end{align*}
    Then we infer that
    $\{ u_{v_i}(t) \}_{i=1}^\infty
    =
    \{ \widetilde{u}_{v_i}^m(t)
    + \widehat{u}_{v_i}^m(t)
    \}_{i=1}^\infty$
    has a finite open cover with radius $\delta$ in
    $\left( V \bigcap L^p(\mathbb{R}^n) \right)^\ast$,
    and hence
    $
    \{ u_{v_i}(t) \}_{i=1}^\infty
    $
    is precompact in
    $\left( V \bigcap L^p(\mathbb{R}^n) \right)^\ast$
    for  $t \in [0,T]$,
    which along with \eqref{uws}
    implies that
    \begin{align} \label{sc-11}
        u_{v_i} \rightarrow  z  \ \  \text{strongly in} \  C \left( [0,T], \left( V \bigcap L^p(\mathbb{R}^n) \right)^\ast \right).
    \end{align}

    Since
    \begin{align*}
        & \int_0^T \| u_{v_i}(t) - z(t) \|^2 dt  \nonumber \\
        = & \int_0^T
        \left\langle  u_{v_i}(t) - z(t),
        u_{v_i}(t) - z(t)
        \right\rangle_{\left( V \bigcap L^p(\mathbb{R}^n), \left( V \bigcap L^p(\mathbb{R}^n) \right)^\ast \right)} dt
        \nonumber \\
        \leq & \left( \int_0^T
        \| u_{v_i}(t) - z(t) \|_{V \bigcap L^p(\mathbb{R}^n)}^2 dt
        \right)^{\frac{1}{2}}
        \left( \int_0^T
        \| u_{v_i}(t) - z(t) \|_{\left( V \bigcap L^p(\mathbb{R}^n) \right)^\ast}^2 dt
        \right)^{\frac{1}{2}},
    \end{align*}
    it follows from \eqref{sc-1} and \eqref{sc-11} that
    \begin{align*}
        u_{v_i}  \rightarrow  z  \ \  \text{strongly in} \ \  L^2(0,T; H).
    \end{align*}

    {\bf Step 3:}   proving  $u_v = z$.

    By \eqref{contequ}, we obtain that for any
    $t\in [0,T]$ and
    $\xi \in  V \bigcap L^p(\mathbb{R}^n)$,
    \begin{align} \label{sc-12}
        & \left( u_{v_i} (t), \xi \right)
        + \int_0^t
        \left( (-\triangle)^{\frac{\alpha}{2}} u_{v_i}(s),
        (-\triangle)^{\frac{\alpha}{2}} \xi
        \right) ds
        \nonumber \\
        & + \int_0^t \int_{\mathbb{R}^n} f( s, x, u_{v_i}(s), \mathcal{L}_{u^0(s)} ) \xi(x) dx ds
        \nonumber\\
        = & \left(u_0, \xi\right)
        + \int_0^t \left( g(s, u_{v_i}(s), \mathcal{L}_{u^0(s)} ), \xi \right) ds
        + \int_0^t \left( \sigma(s, u_{v_i}(s), \mathcal{L}_{u^0(s)} ) v_i(s), \xi \right) ds.
    \end{align}

    It follows from \eqref{sc-10} in Step 2 that (up to a subsequence)
    $$u_{v_i}  \rightarrow  z \ \ \   \text{a.e. on }  \ (0,T) \times \mathbb{R}^n.$$
    Then by the continuity of $f$, we get
    \begin{align} \label{sc-13}
        f(t, x, u_{v_i}(t), \mathcal{L}_{u^0(t)} )  \rightarrow  f(t, x, z, \mathcal{L}_{u^0(t)} )\ \ \
        \text{a.e. on } \  (0,T) \times \mathbb{R}^n.
    \end{align}
    By \eqref{fw}, \eqref{sc-13} and Mazur's theorem, we have
    \begin{align} \label{sc-14}
        \varphi = f(t, x, z, \mathcal{L}_{u^0(t)})  \ \ \  \text{a.e. on } \  (0,T) \times \mathbb{R}^n,
    \end{align}
    and hence
    \begin{align} \label{sc-15}
        f(\cdot, \cdot, u_{v_i}, \mathcal{L}_{u^0(\cdot)} ) \rightarrow f(\cdot, \cdot, z, \mathcal{L}_{u^0(\cdot)} )
        \ \  \text{weakly in} \  L^q(0,T; L^q(\mathbb{R}^n) ).
    \end{align}
    Taking the limit
    of the third term on the left-hand side of \eqref{sc-11},
    we obtain
    \begin{align} \label{sc-16}
        \lim_{i \to   \infty} \int_0^t \int_{\mathbb{R}^n} f( s, x, u_{v_i}(s), \mathcal{L}_{u^0(s)}) \xi(x) dx ds
        = \int_0^t \int_{\mathbb{R}^n} f( s, x, z(s), \mathcal{L}_{u^0(s)} ) \xi(x) dx ds
    \end{align}
    For the second term on the right-hand side of \eqref{sc-12},
    by \eqref{g-2} and \eqref{sc-10}, we have
    \begin{align} \label{sc-17}
        &
        \left |
        \int_0^t \left( g(s, u_{v_i}(s), \mathcal{L}_{u^0(s)} ) - g(s, z(s), \mathcal{L}_{u^0(s)} ), \xi \right) ds
        \right |
        \nonumber \\
        \leq & \int_0^t
        \| g(s, u_{v_i}(s), \mathcal{L}_{u^0(s)} ) - g(s, z(s), \mathcal{L}_{u^0(s)} ) \|
        \| \xi \|
        ds  \nonumber \\
        \leq & \| \psi_g \|_{L^\infty(0,T; L^\infty(\mathbb{R}^n) ) }
        \| \xi \|
        t^{\frac{1}{2}}
        \left( \int_0^t
        \| u_{v_i}(s) - z(s) \|
        ds
        \right)^{\frac{1}{2}}
        \rightarrow   \ 0, \ \ \  \text{as} \  i \to  \infty.
    \end{align}
    For the third term on the right-hand side of \eqref{sc-12}, we have
    \begin{align} \label{sc-18}
        & \int_0^t \left( \sigma(s, u_{v_i}(s), \mathcal{L}_{u^0(s)} ) v_i(s), \xi \right) ds \nonumber \\
        = & \int_0^t
        \left(
        \left( \sigma(s, u_{v_i}(s), \mathcal{L}_{u^0(s)} ) - \sigma(s, z(s), \mathcal{L}_{u^0(s)} )
        \right) v_i(s),
        \xi
        \right) ds
        \nonumber \\
        & + \int_0^t
        \left( \sigma(s, z(s), \mathcal{L}_{u^0(s)} ) v_i(s), \xi \right) ds.
    \end{align}
    For the first term on the right-hand side of \eqref{sc-18},
    by \eqref{sigmalip}, \eqref{sc-1} and \eqref{sc-10},  we have
    \begin{align} \label{sc-19}
        & \| \int_0^t
        \left(
        \left( \sigma(s, u_{v_i}(s), \mathcal{L}_{u^0(s)} ) - \sigma(s, z(s), \mathcal{L}_{u^0(s)} )
        \right) v_i(s),
        \xi
        \right) ds
        \|
        \nonumber \\
        \leq & \int_0^t
        \| \sigma(s, u_{v_i}(s), \mathcal{L}_{u^0(s)} ) - \sigma(s, z(s), \mathcal{L}_{u^0(s)} ) \|_{L_2(l^2, H)}
        \| v_i(s) \|_{l^2}
        \|\xi\|
        ds
        \nonumber \\
        \leq & L_\sigma^{\frac{1}{2}}  \|\xi\|
        \int_0^t
        \| u_{v_i}(s) - z(s) \|
        \| v_i(s) \|_{l^2}
        ds
        \nonumber \\
        \leq & L_\sigma^{\frac{1}{2}}  \|\xi\|
        \left( \int_0^T
        \| u_{v_i}(s) - z(s) \|^2
        ds
        \right)^{\frac{1}{2}}
        \left( \int_0^T \| v_i(s) \|_{l^2}^2 ds
        \right)^{\frac{1}{2}}
        \nonumber \\
        \rightarrow & \  0, \ \ \  \text{as} \  i \to +\infty.
    \end{align}
    For the second term on the right-hand side of \eqref{sc-18},
    we define an operator $\mathcal{G}: L^2(0,T; l^2)  \rightarrow  H$ by
    $$\mathcal{G}(\tilde{v})
    = \int_0^t  \sigma(s, z(s), \mathcal{L}_{u^0(s)} ) \tilde{v}(s) ds, \ \ \
    \forall \  \tilde{v} \in L^2(0,T; l^2).
    $$
    By \eqref{sigma3}, \eqref{uws} and Theorem \ref{dde-wellposed},
    we get for any $t \in [0,T]$,
    \begin{align} \label{sc-20}
        & \| \int_0^t  \sigma(s, z(s), \mathcal{L}_{u^0(s)} ) \tilde{v}(s) ds \|  \nonumber \\
        \leq & \int_0^t \| \sigma(s, z(s), \mathcal{L}_{u^0(s)} ) \|_{L_2(l^2, H)}
        \| \tilde{v}(s) \|_{l^2}
        ds  \nonumber \\
        \leq &
        \left( \int_0^t \| \sigma(s, z(s), \mathcal{L}_{u^0(s)} ) \|_{L_2(l^2, H)}^2 ds
        \right)^{\frac{1}{2}}
        \left( \int_0^t \| \tilde{v}(s) \|_{l^2}^2 ds
        \right)^{\frac{1}{2}} \nonumber \\
        \leq &
        \left[ \int_0^t M_{\sigma,T}
        \left( 1 + \|z(s)\|^2
        + \mathcal{L}_{u^0(s)}(\|\cdot\|^2)
        \right) ds
        \right]^{\frac{1}{2}}
        \| \tilde{v} \|_{L^2(0,T; l^2)} \nonumber \\
        \leq &
        \left\{ TM_{\sigma,T}
        \left[ 1 + \|z\|_{L^\infty(0,T;H)}^2
        + M_0 \left( 1 + \|u_0\|^2 \right)
        \right]
        \right\}^{\frac{1}{2}}
        \| \tilde{v} \|_{L^2(0,T; l^2)},
    \end{align}
    which shows that
    $\mathcal{G}$ is a bounded linear operator, and thus
    is weakly continuous.
    Since $v_i \rightarrow v$ weakly in $L^2(0, T; l^2)$,
    it follows that
    $\mathcal{G}(v_i)$ converges to
    $\mathcal{G}(v)$
    weakly
    in $H$.
    So we have
    \begin{align} \label{sc-21}
        \lim_{i \to  \infty}
        \int_0^t
        \left( \sigma(s, z(s), \mathcal{L}_{u^0(s)} ) v_i(s), \xi \right) ds
        = \int_0^t
        \left( \sigma(s, z(s), \mathcal{L}_{u^0(s)} ) v(s), \xi \right) ds.
    \end{align}
    By \eqref{sc-18}, \eqref{sc-19} and \eqref{sc-21},
    we obtain
    \begin{align} \label{sc-22}
        \lim_{i \to \infty}
        \int_0^t \left(\sigma(s, u_{v_i}(s), \mathcal{L}_{u^0(s)} ) v_i(s), \xi \right) ds
        = \int_0^t
        \left( \sigma(s, z(s), \mathcal{L}_{u^0(s)} ) v(s), \xi \right) ds.
    \end{align}

    Taking the limit  of \eqref{sc-12}, by \eqref{sc-16}, \eqref{sc-17} and \eqref{sc-22},
    we  obtain
    \begin{align*}
        & \left( z (t), \xi \right)
        + \int_0^t
        \left( (-\triangle)^{\frac{\alpha}{2}} z(s),
        (-\triangle)^{\frac{\alpha}{2}} \xi
        \right) ds
        \nonumber \\
        & + \int_0^t \int_{\mathbb{R}^n} f( s, x, z(s), \mathcal{L}_{u^0(s)} ) \xi(x) dx ds
        \nonumber\\
        = & \left(u_0, \xi\right)
        + \int_0^t \left( g(s, z(s), \mathcal{L}_{u^0(s)} ), \xi \right) ds
        + \int_0^t \left( \sigma(s, z(s), \mathcal{L}_{u^0(s)} ) v(s), \xi \right) ds,
    \end{align*}
    which shows that  $z$ is a solution to \eqref{contequ}-\eqref{contequ IV}.
    Then by the uniqueness of solutions to \eqref{contequ}-\eqref{contequ IV} in Lemma \ref{cewellposed},
    we obtain $u_v = z$.

    {\bf Step 4:}   proving
    $ u_{v_i}  \rightarrow  u_v $ strongly in  $C([0,T], H) \bigcap L^2(0,T; V)$.

    By \eqref{contequ}, we obtain
    \begin{align}\label{sc-23}
        & \frac{d}{dt} \| u_{v_i}(t) - u_v(t) \|^2
        + 2 \| (-\triangle)^{\frac{\alpha}{2}} \left( u_{v_i}(t) - u_v(t) \right) \|^2
        \nonumber \\
        & + 2 \int_{\mathbb{R}^n}
        \left( f(t, x, u_{v_i}(t), \mathcal{L}_{u^0(t)} )
        - f(t, x, u_v(t), \mathcal{L}_{u^0(t)})
        \right)
        \left( u_{v_i}(t) - u_v(t)
        \right)
        dx
        \nonumber \\
        = & 2
        \left( g(t, \cdot, u_{v_i}(t), \mathcal{L}_{u^0(t)} )
        - g(t, \cdot, u_v(t), \mathcal{L}_{u^0(t)} ), \
        u_{v_i}(t) - u_v(t)
        \right)  \nonumber \\
        & + 2
        \left( \sigma(t, u_{v_i}(t), \mathcal{L}_{u^0(t)} ) v_i(t)
        - \sigma(t, u_v(t), \mathcal{L}_{u^0(t)} ) v(t), \
        u_{v_i}(t) - u_v(t)
        \right)
    \end{align}
    with $u_{v_i}(0) - u_v(0) = 0$.
    For the third term on the left-hand side of \eqref{sc-23},
    by \eqref{f-4}, we have
    \begin{align} \label{sc-24}
        & 2 \int_{\mathbb{R}^n}
        \left( f(t, x, u_{v_i}(t), \mathcal{L}_{u^0(t)} )
        - f(t, x, u_v(t), \mathcal{L}_{u^0(t)})
        \right)
        \left( u_{v_i}(t) - u_v(t)
        \right)
        dx  \nonumber \\
        \geq & - 2 \| \psi_4(t) \|_{L^\infty(\mathbb{R}^n)}  \| u_{v_i}(t) - u_v(t) \|^2.
    \end{align}
    For the first term on the right-hand side of \eqref{sc-23},
    by \eqref{g-2}, we obtain
    \begin{align} \label{sc-25}
        & 2
        \left( g(t, \cdot, u_{v_i}(t), \mathcal{L}_{u^0(t)} )
        - g(t, \cdot, u_v(t), \mathcal{L}_{u^0(t)} ), \
        u_{v_i}(t) - u_v(t)
        \right)  \nonumber \\
        \leq & 2 \| \psi_g(t) \|_{L^\infty(\mathbb{R}^n)}
        \| u_{v_i}(t) - u_v(t) \|^2.
    \end{align}
    For the second term on the right-hand side of \eqref{sc-23},
    by \eqref{sigma3} and \eqref{sc-1}, we obtain
    \begin{align} \label{sc-26}
        & 2
        \left( \sigma(t, u_{v_i}(t), \mathcal{L}_{u^0(t)} ) v_i(t)
        - \sigma(t, u_v(t), \mathcal{L}_{u^0(t)} ) v(t), \
        u_{v_i}(t) - u_v(t)
        \right) \nonumber \\
        \leq & 2
        \left( \| \sigma(t, u_{v_i}(t), \mathcal{L}_{u^0(t)} ) \|_{L_2(l^2, H)}
        \| v_i(t) \|_{l^2}
        + \| \sigma(t, u_v(t), \mathcal{L}_{u^0(t)} ) \|_{L_2(l^2, H)}
        \| v(t) \|_{l^2}
        \right)
        \| u_{v_i}(t) - u_v(t) \|  \nonumber \\
        \leq & 2 \hat{C}_1
        \left( \| v_i(t) \|_{l^2}  +  \| v(t) \|_{l^2}
        \right)
        \| u_{v_i}(t) - u_v(t) \|,
    \end{align}
    where $\hat{C}_1 = \hat{C}_1(T) >0$.

    By \eqref{sc-23}-\eqref{sc-26},
    we obtain
    \begin{align}\label{sc-27}
        & \frac{d}{dt} \| u_{v_i}(t) - u_v(t) \|^2
        + 2 \| (-\triangle)^{\frac{\alpha}{2}} \left( u_{v_i}(t) - u_v(t) \right) \|^2
        \nonumber \\
        \leq & 2
        \left( \| \psi_4(t) \|_{L^\infty(\mathbb{R}^n)}
        + \| \psi_g(t) \|_{L^\infty(\mathbb{R}^n)}
        \right)
        \| u_{v_i}(t) - u_v(t) \|^2
        \nonumber \\
        & + 2 \hat{C}_1
        \left(
        \| v_i(t) \|_{l^2}
        + \| v(t) \|_{l^2}
        \right)
        \| u_{v_i}(t) - u_v(t) \|.
    \end{align}
    Integrating \eqref{sc-27}
    on $(0,t)$,
    by H\"{o}lder's inequality, \eqref{sc-1}, Step 2 and Step 3,
    we  get for all $t\in [0, T]$,
    \begin{align} \label{sc-28}
        & \| u_{v_i}(t) - u_v(t) \|^2
        + 2 \int_0^t
        \| (-\triangle)^{\frac{\alpha}{2}} \left( u_{v_i}(s) - u_v(s) \right) \|^2 ds
        \nonumber \\
        \leq & 2
        \left(   \| \psi_4(t) \|_{L^\infty(0,T; L^\infty(\mathbb{R}^n))}
        + \| \psi_g(t) \|_{L^\infty(0,T; L^\infty(\mathbb{R}^n))}
        \right)
        \| u_{v_i} - u_v \|_{L^2(0,T; H)}^2
        \nonumber \\
        & + 2 \sqrt{2} \hat{C}_1
        \left(
        \| v_i \|_{L^2(0,T; l^2}
        + \| v \|_{L^2(0,T; l^2}
        \right)
        \| u_{v_i} - u_v \|_{L^2(0,T; H)}  \nonumber \\
        \rightarrow & \ 0, \ \ \  i \to  \infty,
    \end{align}
    which concludes the proof.
 \end{proof}

 \subsection{Convergence of solutions of stochastic equations }

 In this subsection, we prove
 $\mathcal{G}^\varepsilon$
 and  $\mathcal{G}^0$
 satisfy
 $({\bf H2})$
 in Proposition \ref{ldpth}; more precisely, we will show  the
 convergence of  solutions
 of the stochastic equation
 \eqref{sce}
 in
 $C([0,T], H) \bigcap L^2(0,T; V)$,
 for which we need the following uniform
 estimates.

 \begin{lemma}\label{pre-h2}
    Suppose  $({\bf \Sigma 1})$-$({\bf \Sigma 3})$ hold.
    Then for every $R>0$, $T>0$ and $N>0$,
    there exists $M_3 = M_3 (R,T,N) > 0$ such that for any $u_0 \in H$ with $\| u_0 \| \leq R$
    and any $v \in \mathcal{A}_N$,
    the solution $u_v^\varepsilon$ of \eqref{sce}
    satisfies for all $\varepsilon \in (0,1)$,
    \begin{align}\label{bound}
        & \mathbb{E} \left[ \| u_v^\varepsilon \|^2_{C([0,T],H)} \right]
        + \mathbb{E} \left[ \| u_v^\varepsilon \|^2_{L^2(0,T; V)} \right]
        + \mathbb{E} \left[ \| u_v^\varepsilon \|^2_{L^p(0,T; L^p(\mathbb{R}^n) )} \right]
        \leq  M_3.
    \end{align}
 \end{lemma}

 \begin{proof}
    It follows from \eqref{sce} and It\^{o}'s formula that for all $t\in[0,T]$, $P$-almost surely,
    \begin{align}\label{w-1}
        & \|u_v^\varepsilon(t)\|^2
        + 2 \int_0^t \|(-\triangle)^{\frac{\alpha}{2}} u_v^\varepsilon (s) \|^2 ds
        + 2 \int_{0}^{t} \int_{\mathbb{R}^n}
        f(s, x, u_v^\varepsilon (s), \mathcal{L}_{u^\varepsilon(s)} )  u_v^\varepsilon (s)
        dx ds
        \nonumber \\
        = & \| u_0 \|^2
        + 2 \int_{0}^{t}
        \left( g(s, \cdot, u_v^\varepsilon (s), \mathcal{L}_{u^\varepsilon(s)} ), u_v^\varepsilon (s) \right) ds
        + 2 \int_{0}^{t}
        \left( \sigma(s, u_v^\varepsilon (s), \mathcal{L}_{u^\varepsilon(s)} ) v(s),
        u_v^\varepsilon (s)
        \right) ds
        \nonumber \\
        & + \varepsilon \int_{0}^{t}
        \left\| \sigma(s, u_v^\varepsilon (s), \mathcal{L}_{u^\varepsilon(s)} )
        \right\|^2_{L_2(l^2, H)} ds
        + 2 \sqrt{\varepsilon} \int_{0}^{t}
        \left( u_v^\varepsilon (s), \sigma(s, u_v^\varepsilon (s), \mathcal{L}_{u^\varepsilon(s)} ) dW(s) \right).
    \end{align}
    For the third term on the left-hand side of \eqref{w-1},
    by \eqref{f-1} and Theorem \ref{mvsde-wellposed}, we have
    \begin{align}\label{w-2}
        & 2 \int_{0}^{t} \int_{\mathbb{R}^n}
        f(s, x, u_v^\varepsilon (s), \mathcal{L}_{u^\varepsilon(s)} )  u_v^\varepsilon (s)
        dx ds  \nonumber \\
        \geq &
        2 \lambda_1 \int_{0}^{t} \| u_v^\varepsilon(s) \|_{L^p(\mathbb{R}^n)}^{p} ds
        - 2 \int_{0}^{t} \| \psi_1(s) \|_{L^1(\mathbb{R}^n)} \left( 1 + \mathbb{E}\left[\|u^\varepsilon(s)\|^2\right] \right) ds
        ds\nonumber \\
        & - 2 \int_{0}^{t} \| \psi_1(s) \|_{L^\infty
            (\mathbb{R}^n)}    \| u_v^\varepsilon(s) \|^2  ds
        \nonumber \\
        \geq & 2 \lambda_1 \int_{0}^{t} \| u_v^\varepsilon(s) \|_{L^p(\mathbb{R}^n)}^{p} ds
        - 2 \left[ 1 + M_1( 1 + \|u_0\|^2 ) \right] \int_{0}^{t} \| \psi_1(s) \|_{L^1(\mathbb{R}^n)} ds
        ds\nonumber \\
        & - 2 \int_{0}^{t} \| \psi_1(s) \|_{L^\infty
            (\mathbb{R}^n)}    \| u_v^\varepsilon(s) \|^2  ds.
    \end{align}
    For the second term on the right-hand side of \eqref{w-1},
    by \eqref{g-3} and Theorem \ref{mvsde-wellposed}, we have
    \begin{align}\label{w-3}
        & 2 \int_{0}^{t}
        \left( g(s, \cdot, u_v^\varepsilon(s), \mathcal{L}_{u^\varepsilon(s)} ), u_v^\varepsilon(s) \right) ds
        \nonumber \\
        \leq & \int_{0}^{t} \| g(s, \cdot, u_v^\varepsilon(s), \mathcal{L}_{u^\varepsilon(s)} ) \|^2 ds
        + \int_{0}^{t} \| u_v^\varepsilon(s) \|^2 ds
        \nonumber \\
        \leq  & 4 \int_{0}^{t} \| \psi_g(s) \|_{L^\infty(\mathbb{R}^n) }^2
        \|  u_v^\varepsilon(s) \|^2
        ds
        + 2 (1+ 2
        \mathbb{E}\left[ \| u^\varepsilon(s) \|^2 \right]
        )
        \int_{0}^{t} \| \psi_g(s) \|^2 ds
        + \int_{0}^{t} \| u_v^\varepsilon(s) \|^2 ds
        \nonumber \\
        \leq & \int_{0}^{t} \left( 1 + 4 \| \psi_g (s) \|_{L^\infty(\mathbb{R}^n)}^2 \right) \| u_v^\varepsilon(s) \|^2 ds
        + 2 \left[ 1 + 2 M_1 ( 1 + \| u_0 \|^2) \right]
        \int_{0}^{t} \| \psi_g (s) \|^2 ds                        .
    \end{align}
    For the third term on the right-hand side of \eqref{w-1},
    by \eqref{sigma3} and Theorem \ref{mvsde-wellposed}, we have
    \begin{align}\label{w-4}
        & 2 \int_{0}^{t}
        \left( \sigma(s, u_v^\varepsilon (s), \mathcal{L}_{u^\varepsilon(s)} ) v(s),
        u_v^\varepsilon (s)
        \right) ds
        \nonumber \\
        \leq & \int_{0}^{t}
        \| \sigma(s, u_v^\varepsilon (s), \mathcal{L}_{u^\varepsilon(s)} ) \|_{L_2(l^2, H)}^2
        \| v(s) \|_{l^2}^2
        ds
        + \int_0^t \| u_v^\varepsilon (s) \|^2 ds
        \nonumber \\
        \leq
        & \int_{0}^{t}
        M_{\sigma,T} \left( 1  + \mathbb{E}\left[ \| u^\varepsilon(s) \|^2 \right] \right)
        \| v(s) \|_{l^2}^2
        ds
        + \int_0^t
        \left( 1 + M_{\sigma,T} \| v(s) \|_{l^2}^2 \right)
        \| u_v^\varepsilon (s) \|^2 ds  \nonumber \\
        \leq
        & M_{\sigma,T} \left[ 1 + M_1( 1 + \| u_0 \|^2 ) \right]
        \int_{0}^{t} \| v(s) \|_{l^2}^2 ds
        + \int_0^t
        \left( 1 + M_{\sigma,T} \| v(s) \|_{l^2}^2 \right)
        \| u_v^\varepsilon (s) \|^2 ds.
    \end{align}
    For the fourth term on the right-hand side of \eqref{w-1},
    by \eqref{sigma3} and Theorem \ref{mvsde-wellposed}, we have
    \begin{align}\label{w-6}
        & \varepsilon \int_{0}^{t}
        \left\| \sigma(s, u_v^\varepsilon (s), \mathcal{L}_{u^\varepsilon(s)} )
        \right\|^2_{L_2(l^2, H)} ds
        \nonumber \\
        \leq
        & \varepsilon \int_0^t
        M_{\sigma,T} \left( 1 + \| u_v^\varepsilon(s) \|^2 + \mathbb{E}\left[ \| u^\varepsilon(s) \|^2\right] \right) ds
        \nonumber \\
        \leq & \varepsilon T M_{\sigma,T} \left[ 1 + M_1( 1 + \|u_0\|^2 ) \right]
        + \varepsilon M_{\sigma,T} \int_0^t \| u_v^\varepsilon(s) \|^2 ds.
    \end{align}

    From \eqref{w-1}-\eqref{w-6},
    it follows that for all
    $t \in [0,T]$, $P$-almost surely,
    \begin{align}\label{w-7}
        & \|u_v^\varepsilon(t)\|^2
        + 2 \int_0^t \|(-\triangle)^{\frac{\alpha}{2}} u_v^\varepsilon (s) \|^2 ds
        + 2 \lambda_1 \int_{0}^{t} \| u_v^\varepsilon(s) \|_{L^p(\mathbb{R}^n)}^{p} ds
        \nonumber \\
        \leq
        & \| u_0 \|^2
        + \int_{0}^{t}
        \left( 2 + 4 \| \psi_g (s) \|_{L^\infty(\mathbb{R}^n)}^2
        + 2 \| \psi_1 (s)\|_{L^\infty(\mathbb{R}^n)}
        + \varepsilon M_{\sigma,T}
        + M_{\sigma,T} \| v(s) \|_{l^2}^2
        \right) \| u_v^\varepsilon(s) \|^2 ds
        \nonumber \\
        & + \left[ 1 + 2 M_1( 1 + \|u_0\|^2 ) \right]
        \left( 2 \int_{0}^{t} \| \psi_1(s) \|_{L^1(\mathbb{R}^n)} ds
        + 2 \int_{0}^{t} \| \psi_g (s) \|^2 ds
        + M_{\sigma,T} \int_{0}^{t} \| v(s) \|_{l^2}^2 ds
        \right)
        \nonumber \\
        & + \varepsilon T M_{\sigma,T} \left[ 1 + M_1( 1 + \|u_0\|^2 ) \right]
        + 2 \sqrt{\varepsilon} \int_{0}^{t}
        \left( u_v^\varepsilon (s), \sigma(s, u_v^\varepsilon (s), \mathcal{L}_{u^\varepsilon(s)} ) dW(s) \right).
    \end{align}

    If $u_0 \in H$ with $\|u_0\| \leq R$,
    and $v \in \mathcal{A}_N$,
    then we obtain by \eqref{w-7} that for all  $0 \leq r \leq t \leq T$,
    $P$-almost surely,
    \begin{align}\label{w-8}
        & \|u_v^\varepsilon(r)\|^2
        + 2 \int_0^r \|(-\triangle)^{\frac{\alpha}{2}} u_v^\varepsilon (s) \|^2 ds
        + 2 \lambda_1 \int_{0}^{r} \| u_v^\varepsilon(s) \|_{L^p(\mathbb{R}^n)}^{p} ds
        \nonumber \\
        \leq & R^2
        + \int_{0}^{r}
        \left( 2 + 4 \| \psi_g \|_{L^\infty(0,T; L^\infty(\mathbb{R}^n) ) }^2
        + 2 \| \psi_1 \|_{L^\infty(0,T; L^\infty(\mathbb{R}^n) ) }^2
        + \varepsilon M_{\sigma,T}
        + M_{\sigma,T} \| v(s) \|_{l^2}^2
        \right) \| u_v^\varepsilon(s) \|^2 ds
        \nonumber \\
        & + \left[ 1 + 2 M_1( 1 + R^2 ) \right]
        \left( 2 \int_{0}^{T} \| \psi_1(s) \|_{L^1(\mathbb{R}^n)} ds
        + 2 \int_{0}^{T} \| \psi_g (s) \|^2 ds
        + M_{\sigma,T} N
        \right)
        \nonumber \\
        & + \varepsilon T M_{\sigma,T} \left[ 1 + M_1( 1 + R^2 ) \right]
        + 2 \sqrt{\varepsilon} \int_{0}^{r}
        \left( u_v^\varepsilon (s), \sigma(s, u_v^\varepsilon (s), \mathcal{L}_{u^\varepsilon(s)} ) dW(s) \right)
        \nonumber \\
        \leq & c_1
        + c_1 \int_{0}^{r}
        \left( 1 + \| v(s) \|_{l^2}^2
        \right)
        \| u_v^\varepsilon (s) \|^2
        ds
        + K(t),
    \end{align}
    where $c_1 = c_1 (R, N, T)>0$,
    $K(t)= 2 \sqrt{\varepsilon}
    \sup_{r \in [0, t]}
    \left| \int_{0}^{r}
    \left( u_v^\varepsilon (s), \sigma(s, u_v^\varepsilon (s), \mathcal{L}_{u^\varepsilon(s)} ) dW(s)
    \right)
    \right|.
    $

    By Gronwall's inequality and \eqref{w-8}, we obtain for all $r \in [0,t]$,
    \begin{align}\label{w-9}
        & \|u_v^\varepsilon(r)\|^2
        + 2 \int_0^r \|(-\triangle)^{\frac{\alpha}{2}} u_v^\varepsilon (s) \|^2 ds
        + 2 \lambda_1 \int_{0}^{r} \| u_v^\varepsilon(s) \|_{L^p(\mathbb{R}^n)}^{p} ds
        \nonumber \\
        \leq & \left( c_1 + K(t) \right)
        e^{c_1 \int_0^r \left( 1 + \| v(s) \|_{l^2}^2 \right) ds}
        \nonumber \\
        \leq & \left( c_1 + K(t) \right) e^{c_1 (T + N)},  \ \ \    \mathbb{P} \text{-almost surely}.
    \end{align}
    Then by \eqref{w-9}, we obtain
    for all $t\in [0,T]$,
    \begin{align}\label{w-10}
        & \mathbb{E}
        \left[ \sup_{r \in [0,t]} \| u_v^\varepsilon(r) \|^2
        \right]
        + 2 \mathbb{E}
        \left[ \int_0^t \|(-\triangle)^{\frac{\alpha}{2}} u_v^\varepsilon (s) \|^2 ds
        \right]
        + 2 \lambda_1 \mathbb{E}
        \left[ \int_{0}^{t} \| u_v^\varepsilon(s) \|_{L^p(\mathbb{R}^n)}^{p} ds
        \right]
        \nonumber \\
        \leq & 2 \left( c_1
        + \mathbb{E} \left[ K(t) \right]
        \right) e^{c_1 (T + N)}  \nonumber \\
        \leq & c_2 + c_2 \mathbb{E} \left[ K(t) \right],
        \ \ \    \mathbb{P} \text{-almost surely},
    \end{align}
    where $c_2 = 2 \left( c_1 + 1 \right) e^{c_1 (T + N)}$.
    By the Burkholder inequality and \eqref{sigma3}, we have
    \begin{align}\label{w-11}
        c_2 \mathbb{E} \left[ K(t) \right]
        = & 2 c_2 \sqrt{\varepsilon}
        \mathbb{E}
        \left[ \sup_{r \in [0, t]}
        \left| \int_{0}^{r}
        \left( u_v^\varepsilon (s), \sigma(s, u_v^\varepsilon (s), \mathcal{L}_{u^\varepsilon(s)} ) dW(s)
        \right)
        \right|
        \right]  \nonumber \\
        \leq & 6 c_2 \sqrt{\varepsilon}
        \mathbb{E}\left[
        \left( \int_{0}^{t}
        \| u_v^\varepsilon(s) \|^2
        \| \sigma(s, u_v^\varepsilon (s), \mathcal{L}_{u^\varepsilon(s)} ) \|_{L_2(l^2, H)}^2
        ds
        \right)^{\frac{1}{2}}
        \right]
        \nonumber \\
        \leq & \frac{1}{2} \mathbb{E}
        \left[ \sup_{s \in [0,t]}
        \| u_v^\varepsilon(s) \|^2
        \right]
        + 18 c_2^2 \varepsilon \mathbb{E}
        \left[
        \int_{0}^{t} \| \sigma(s, u_v^\varepsilon (s), \mathcal{L}_{u^\varepsilon(s)} ) \|^2_{L_2(l^2, H)} ds
        \right]
        \nonumber \\
        \leq & \frac{1}{2} \mathbb{E}
        \left[ \sup_{s \in [0,t]}
        \| u_v^\varepsilon(s) \|^2
        \right]
        + 18 c_2^2 \varepsilon M_{\sigma,T}
        \mathbb{E}
        \left[ \int_{0}^{t}
        \left( 1 + \| u_v^\varepsilon(s) \|^2
        + \mathcal{L}_{u^\varepsilon(s)}(\|\cdot\|^2)
        \right)
        ds
        \right]  \nonumber \\
        = & \frac{1}{2} \mathbb{E}
        \left[ \sup_{s \in [0,t]}
        \| u_v^\varepsilon(s) \|^2
        \right]
        + 18 c_2^2 \varepsilon M_{\sigma,T}
        \left[ \int_{0}^{t}
        \left( 1 +  \mathbb{E} \left[ \| u_v^\varepsilon(s) \|^2 \right]
        + \mathbb{E} \left[ \|u^\varepsilon(s)\|^2 \right]
        \right)
        ds
        \right] \nonumber \\
        \leq & \frac{1}{2} \mathbb{E}
        \left[ \sup_{r \in [0,t]}
        \| u_v^\varepsilon(r) \|^2
        \right]
        + 18 c_2^2 \varepsilon T M_{\sigma,T} \left[ 1 + M_1(1+\|u_0\|^2) \right]
        \nonumber \\
        & + 18 c_2^2 \varepsilon M_{\sigma,T} \int_{0}^{t}
        \mathbb{E} \left[
        \sup_{r\in[0,s]}\| u_v^\varepsilon(r) \|^2 \right]
        ds.
    \end{align}

    By \eqref{w-10}-\eqref{w-11}, we obtain that for all $t\in[0,T]$ and $\varepsilon \in (0,1)$,
    \begin{align} \label{w-13}
        & \mathbb{E}
        \left[ \sup_{r \in [0,t]} \| u_v^\varepsilon(r) \|^2
        \right]
        + 4 \mathbb{E}
        \left[ \int_0^t \|(-\triangle)^{\frac{\alpha}{2}} u_v^\varepsilon (s) \|^2 ds
        \right]
        + 4 \lambda_1 \mathbb{E}
        \left[ \int_{0}^{t} \| u_v^\varepsilon(s) \|_{L^p(\mathbb{R}^n)}^{p} ds
        \right]
        \nonumber \\
        \leq & c_3
        + c_3 \int_{0}^{t} \mathbb{E}
        \left[  \sup_{r\in [0,s]}
        \| u_v^\varepsilon(r) \|^2
        \right]  ds,
    \end{align}
    where $c_3 = c_3(R, N, T) > 0$.

    By \eqref{w-13} and Gronwall's inequality,
    we have for all $t\in[0,T]$ and $\varepsilon\in (0,1)$,
    \begin{align*}
        & \mathbb{E}
        \left[ \sup_{r \in [0,t]} \| u_v^\varepsilon(r) \|^2
        \right]
        + 4 \mathbb{E}
        \left[ \int_0^t \|(-\triangle)^{\frac{\alpha}{2}} u_v^\varepsilon (s) \|^2 ds
        \right]
        + 4 \lambda_1 \mathbb{E}
        \left[ \int_{0}^{t} \| u_v^\varepsilon(s) \|_{L^p(\mathbb{R}^n)}^{p} ds
        \right]
        \leq c_3 e^{c_3 t},
    \end{align*}
    which implies \eqref{bound}.
 \end{proof}

 The next lemma is
 concerned with the convergence of
 solutions of \eqref{mvsde-2}-\eqref{mvsde-2 IV}
 as $\varepsilon \to 0$.

 \begin{lemma} \label{sdc}
    Suppose  $({\bf \Sigma 1})$-$({\bf \Sigma 3})$ hold
    and $u_0\in H$.
    Let $u^\varepsilon$
    and $u^0$ be the solutions
    of
    \eqref{mvsde-2}-\eqref{mvsde-2 IV} and \eqref{dde}-\eqref{dde IV},
    respectively.
    Then
    there exists   $C=C(T)>0$
    such that for all
    $\varepsilon \in (0,1)$,
    $$ \mathbb{E} \left[ \sup_{t \in [0,T]} \| u^\varepsilon(t) - u^0 (t) \|^2 \right]
    \leq C \varepsilon \left( 1 + \| u_0 \|^2
    \right).
    $$
 \end{lemma}
 \begin{proof}

    By \eqref{mvsde-2}-\eqref{mvsde-2 IV} and \eqref{dde}-\eqref{dde IV},
    we find that
    $u^\varepsilon  - u^0
    $ satisfies the   equation:
    \begin{align} \label{difference-1}
        & d \left( u^\varepsilon(t) - u^0 (t) \right)
        + (-\triangle)^\alpha \left( u^\varepsilon(t) - u^0 (t) \right) dt
        \nonumber \\
        &  + \left( f(t, \cdot,  u^\varepsilon(t), \mathcal{L}_{u^\varepsilon (t)} )
        - f(t, \cdot,  u^0(t), \mathcal{L}_{u^0(t)} )
        \right) dt
        \nonumber \\
        = & \left( g(t, \cdot, u^\varepsilon(t), \mathcal{L}_{u^\varepsilon (t)} )
        - g(t, \cdot, u^0(t), \mathcal{L}_{u^0(t)} )
        \right) dt
        + \sqrt{\varepsilon}
        \sigma(t, u^\varepsilon (t), \mathcal{L}_{u^\varepsilon (t)} ) dW(t),
    \end{align}
    with initial condition $u^\varepsilon(0) - u^0(0) = 0$.

    Applying It\^{o}'s formula to \eqref{difference-1}, we have
    \begin{align} \label{difference-2}
        & \| u^\varepsilon(t) - u^0 (t) \|^2
        + 2 \int_0^t \| (-\triangle)^{\frac{\alpha}{2}} \left( u^\varepsilon(s) - u^0 (s) \right) \|^2 ds
        \nonumber \\
        &  + 2 \int_0^t \int_{\mathbb{R}^n}
        \left( f(s, x, u^\varepsilon(s), \mathcal{L}_{u^\varepsilon(s)} )
        - f(s, x, u^0(s), \mathcal{L}_{u^0(s)} )
        \right)
        \left( u^\varepsilon(s) - u^0 (s)
        \right) dx ds
        \nonumber \\
        = & 2 \int_0^t
        \left( g(s, \cdot, u^\varepsilon(s), \mathcal{L}_{u^\varepsilon (s)} )
        - g(s, \cdot, u^0(s) , \mathcal{L}_{u^0(s)} ),
        u^\varepsilon(s) - u^0 (s)
        \right) ds
        \nonumber \\
        & + \varepsilon
        \int_0^t
        \| \sigma(s, u^\varepsilon (s), \mathcal{L}_{u^\varepsilon (s)} )
        \|_{L_2(l^2,H)}^2 ds
        \nonumber \\
        & + 2 \sqrt{\varepsilon}
        \int_0^t
        \left( u^\varepsilon(s) - u^0 (s),
        \sigma(s, u^\varepsilon (s), \mathcal{L}_{u^\varepsilon (s)} ) dW(s)
        \right).
    \end{align}
    For the third term on the left-hand side of \eqref{difference-2},
    by \eqref{f-2}, \eqref{f-4} and H\"{o}lder's inequality, we have
    \begin{align} \label{difference-3}
        & 2 \int_0^t \int_{\mathbb{R}^n}
        \left( f(s, x, u^\varepsilon(s), \mathcal{L}_{u^\varepsilon(s)} )
        - f(s, x, u^0(s), \mathcal{L}_{u^0(s)} )
        \right)
        \left( u^\varepsilon(s) - u^0 (s)
        \right) dx ds
        \nonumber \\
        = & 2 \int_0^t \int_{\mathbb{R}^n}
        \left( f(s, x, u^\varepsilon(s), \mathcal{L}_{u^\varepsilon(s)} )
        - f(s, x, u^0(s), \mathcal{L}_{u^\varepsilon(s)} )
        \right)
        \left( u^\varepsilon(s) - u^0 (s)
        \right) dx ds
        \nonumber \\
        & + 2 \int_0^t \int_{\mathbb{R}^n}
        \left( f(s, x, u^0(s), \mathcal{L}_{u^\varepsilon(s)} )
        - f(s, x, u^0(s), \mathcal{L}_{u^0(s)} )
        \right)
        \left( u^\varepsilon(s) - u^0 (s)
        \right) dx ds
        \nonumber \\
        \geq & -2 \int_0^t \int_{\mathbb{R}^n}
        \psi_4(s, x) | u^\varepsilon(s) - u^0(s) |^2
        dx ds
        \nonumber \\
        & -2 \int_0^t \int_{\mathbb{R}^n}
        \psi_3(s, x)
        \mathbb{W}_2
        \left(
        \mathcal{L}_{u^\varepsilon(s)},
        \mathcal{L}_{u^0(s)}
        \right)
        | u^\varepsilon(s) - u^0 (s) |
        dx ds
        \nonumber \\
        \geq & -2 \int_0^t
        \| \psi_4(s) \|_{L^\infty(\mathbb{R}^n)}
        \| u^\varepsilon(s) - u^0(s) \|^2
        ds  \nonumber \\
        & -2 \int_0^t
        \| \psi_3(s) \|
        \| u^\varepsilon(s) - u^0 (s) \|
        \mathbb{W}_2
        \left(
        \mathcal{L}_{u^\varepsilon(s)},
        \mathcal{L}_{u^0(s)}
        \right)
        ds  \nonumber \\
        \geq & - \int_0^t
        \left( 2 \| \psi_4(s) \|_{L^\infty(\mathbb{R}^n)}
        + \| \psi_3(s) \|^2
        \right)
        \| u^\varepsilon(s) - u^0(s) \|^2
        ds
        - \int_0^t
        \mathbb{W}_2
        \left(
        \mathcal{L}_{u^\varepsilon(s)},
        \mathcal{L}_{u^0(s)}
        \right)^2
        ds.
    \end{align}
    For the first term on the right-hand side of \eqref{difference-2},
    by \eqref{g-2} and H\"{o}lder's inequality, we obtain
    \begin{align} \label{difference-4}
        & 2 \int_0^t
        \left( g(s, \cdot, u^\varepsilon(s), \mathcal{L}_{u^\varepsilon (s)} )
        - g(s, \cdot, u^0(s), \mathcal{L}_{u^0(s)} ),
        u^\varepsilon(s) - u^0 (s)
        \right) ds
        \nonumber \\
        \leq & 2 \int_0^t \int_{\mathbb{R}^n}
        \left| g(s, x, u^\varepsilon(s), \mathcal{L}_{u^\varepsilon (s)} )
        - g(s, x, u^0(s), \mathcal{L}_{u^0(s)} )
        \right|
        | u^\varepsilon(s) - u^0 (s) |
        dx ds \nonumber \\
        \leq & 2 \int_0^t \int_{\mathbb{R}^n}
        \psi_g(s, x)
        \left(  | u^\varepsilon(s) - u^0(s) |
        + \mathbb{W}_2 \left( \mathcal{L}_{u^\varepsilon(s)}, \mathcal{L}_{u^0(s)} \right)
        \right)
        | u^\varepsilon(s) - u^0 (s) |
        dx ds \nonumber \\
        \leq & 2 \int_0^t
        \|\psi_g(s)\|_{L^\infty(\mathbb{R}^n)}
        \| u^\varepsilon(s) - u^0(s) \|^2
        ds
        + 2 \int_0^t
        \| \psi_g(s) \|
        \| u^\varepsilon(s) - u^0 (s) \|
        \mathbb{W}_2 \left( \mathcal{L}_{u^\varepsilon(s)}, \mathcal{L}_{u^0(s)} \right)
        ds \nonumber \\
        \leq & \int_0^t
        \left( 2 \|\psi_g(s)\|_{L^\infty(\mathbb{R}^n)}
        +  \| \psi_g(s) \|^2
        \right)
        \| u^\varepsilon(s) - u^0(s) \|^2
        ds
        + \int_0^t
        \mathbb{W}_2 \left( \mathcal{L}_{u^\varepsilon(s)}, \mathcal{L}_{u^0(s)}
        \right)^2
        ds.
    \end{align}
    Then from \eqref{difference-2}-\eqref{difference-4},
    it follows that
    \begin{align} \label{difference-5}
        & \| u^\varepsilon(t) - u^0 (t) \|^2
        + 2 \int_0^t \| (-\triangle)^{\frac{\alpha}{2}} \left( u^\varepsilon(s) - u^0 (s) \right) \|^2 ds
        \nonumber \\
        \leq & \int_0^t
        \left( 2 \| \psi_4(s) \|_{L^\infty(\mathbb{R}^n)}
        + \| \psi_3(s) \|^2
        + 2 \|\psi_g(s)\|_{L^\infty(\mathbb{R}^n)}
        + \| \psi_g(s) \|^2
        \right)
        \| u^\varepsilon(s) - u^0(s) \|^2
        ds  \nonumber \\
        & + 2 \int_0^t
        \mathbb{W}_2
        \left(
        \mathcal{L}_{u^\varepsilon(s)},
        \mathcal{L}_{u^0(s)}
        \right)^2
        ds
        + \varepsilon
        \int_0^t
        \| \sigma(s, u^\varepsilon (s), \mathcal{L}_{u^\varepsilon (s)} )
        \|_{L_2(l^2,H)}^2 ds
        \nonumber \\
        & + 2 \sqrt{\varepsilon}
        \int_0^t
        \left( u^\varepsilon(s) - u^0 (s),
        \sigma(s, u^\varepsilon (s), \mathcal{L}_{u^\varepsilon (s)} ) dW(s)
        \right).
    \end{align}
    For the third term on the right-hand side of \eqref{difference-5},
    by \eqref{sigma3} and \eqref{sigmalip}, we obtain
    \begin{align} \label{difference-6}
        & \varepsilon
        \int_0^t
        \| \sigma(s, u^\varepsilon (s), \mathcal{L}_{u^\varepsilon (s)} )
        \|_{L_2(l^2,H)}^2 ds
        \nonumber \\
        \leq  & 2 \varepsilon
        \int_0^t
        \| \sigma(s, u^\varepsilon (s), \mathcal{L}_{u^\varepsilon (s)} )
        - \sigma(s, u^0 (s), \mathcal{L}_{u^0(s)} )
        \|_{L_2(l^2,H)}^2 ds
        \nonumber \\
        & + 2 \varepsilon
        \int_0^t
        \| \sigma(s, u^0 (s), \mathcal{L}_{u^0(s)} )
        \|_{L_2(l^2,H)}^2 ds
        \nonumber \\
        \leq & 2 \varepsilon L_{\sigma}
        \int_0^t
        \left( \| u^\varepsilon (s) - u^0 (s) \|^2
        + \mathbb{W}_2 \left( \mathcal{L}_{u^\varepsilon (s)}, \mathcal{L}_{u^0(s)}
        \right)^2
        \right)
        ds  \nonumber \\
        & + 2 \varepsilon T M_{\sigma,T}
        \left( 1 + 2 \sup_{s \in [0,T]} \| u^0(s) \|^2
        \right).
    \end{align}
    For the fourth term on the right-hand side of \eqref{difference-5},
    by the Burkholder inequality and \eqref{difference-6}, we obtain
    \begin{align} \label{difference-7}
        & 2 \mathbb{E}
        \left[ \sqrt{\varepsilon}
        \sup_{r \in [0,t]}
        \left| \int_0^r
        \left( u^\varepsilon(s) - u^0 (s),
        \sigma(s, u^\varepsilon (s), \mathcal{L}_{u^\varepsilon (s)} ) dW(s)
        \right)
        \right|
        \right]
        \nonumber \\
        \leq & 6 \sqrt{\varepsilon} \mathbb{E}
        \left[
        \left(   \int_0^t
        \| u^\varepsilon(s) - u^0(s) \|^2
        \| \sigma(s, u^\varepsilon (s), \mathcal{L}_{u^\varepsilon (s)} ) \|_{L_2(l^2,H)}^2
        ds
        \right)^{\frac{1}{2}}
        \right]
        \nonumber \\
        \leq & \frac{1}{2} \mathbb{E}
        \left[ \sup_{s \in [0,t]}
        \| u^\varepsilon(s) - u^0(s) \|^2
        \right]
        + 18 \varepsilon \mathbb{E}
        \left[
        \int_{0}^{t} \| \sigma(s, u_v^\varepsilon (s), \mathcal{L}_{u^\varepsilon(s)} ) \|^2_{L_2(l^2, H)} ds
        \right]
        \nonumber \\
        \leq & \frac{1}{2} \mathbb{E}
        \left[ \sup_{s \in [0,t]}
        \| u^\varepsilon(s) - u^0(s) \|^2
        \right]
        + 36 \varepsilon L_{\sigma}
        \int_0^t \mathbb{E}
        \left[ \| u^\varepsilon (s) - u^0 (s) \|^2
        + \mathbb{W}_2 \left( \mathcal{L}_{u^\varepsilon (s)}, \mathcal{L}_{u^0(s)}
        \right)^2
        \right]
        ds  \nonumber \\
        & + 36 \varepsilon T M_{\sigma,T}
        \left( 1 + 2 \sup_{s \in [0,T]} \| u^0(s) \|^2
        \right).
    \end{align}
    Then by \eqref{difference-5}-\eqref{difference-7}, we have
    for all $t\in [0,T]$,
    \begin{align} \label{difference-8}
        & \mathbb{E}
        \left[ \sup_{r \in [0,t]}
        \| u^\varepsilon(r) - u^0 (r) \|^2
        \right]
        \nonumber \\
        \leq & 2 \int_0^t
        \left( 2 \| \psi_4(s) \|_{L^\infty(\mathbb{R}^n)}
        + \| \psi_3(s) \|^2
        + 2 \|\psi_g(s)\|_{L^\infty(\mathbb{R}^n)}
        + \| \psi_g(s) \|^2
        \right)
        \mathbb{E}
        \left[ \| u^\varepsilon(s) - u^0(s) \|^2
        \right]
        ds  \nonumber \\
        & + 4 \int_0^t
        \mathbb{W}_2
        \left(
        \mathcal{L}_{u^\varepsilon(s)},
        \mathcal{L}_{u^0(s)}
        \right)^2
        ds
        + 76 \varepsilon T M_{\sigma,T}
        \left( 1 + 2 \sup_{s \in [0,T]} \| u^0(s) \|^2
        \right)
        \nonumber \\
        & + 76 \varepsilon L_{\sigma}
        \int_0^t
        \left(
        \mathbb{E} \left[ \| u^\varepsilon (s) - u^0 (s) \|^2 \right]
        + \mathbb{W}_2 \left( \mathcal{L}_{u^\varepsilon (s)}, \mathcal{L}_{u^0(s)}
        \right)^2
        \right)
        ds  \nonumber \\
        \leq & 2 \int_0^t
        \left( 2 \| \psi_4(s) \|_{L^\infty(\mathbb{R}^n)}
        + \| \psi_3(s) \|^2
        + 2 \|\psi_g(s)\|_{L^\infty(\mathbb{R}^n)}
        + \| \psi_g(s) \|^2
        \right)
        \mathbb{E}
        \left[ \| u^\varepsilon(s) - u^0(s) \|^2
        \right]
        ds  \nonumber \\
        & + \left( 4 + 152 \varepsilon L_{\sigma} \right)
        \int_0^t
        \mathbb{E}
        \left[\| u^\varepsilon(s) - u^0(s) \|^2
        \right]
        ds
        + 76 \varepsilon T M_{\sigma,T}
        \left( 1 + 2 \sup_{s \in [0,T]} \| u^0(s) \|^2
        \right).
    \end{align}
    Applying Gronwall's inequality to \eqref{difference-8}, we find,
    for all $t\in [0,T]$,
    \begin{align} \label{difference-9}
        \mathbb{E}
        \left[ \sup_{r \in [0,t]}
        \| u^\varepsilon(r) - u^0 (r) \|^2
        \right]
        \leq  C_T \varepsilon
        \left( 1 + 2 \sup_{s \in [0,T]} \| u^0(s) \|^2
        \right),
    \end{align}
    where
    $C_T = 76 T M_{\sigma,T}
    e^{ 2 \int_0^T
        \left( 2 \| \psi_4(s) \|_{L^\infty(\mathbb{R}^n)}
        + \| \psi_3(s) \|^2
        + 2 \|\psi_g(s)\|_{L^\infty(\mathbb{R}^n)}
        + \| \psi_g(s) \|^2
        + 2
        + 76 \varepsilon L_{\sigma}
        \right)
        ds}.
    $
    Then by \eqref{difference-8} and Theorem \ref{dde-wellposed}, we
    get
    \begin{align*}
        & \mathbb{E}
        \left[ \sup_{r \in [0,T]}
        \| u^\varepsilon(r) - u^0 (r) \|^2
        \right]
        \leq  C_T \varepsilon
        \left( 1 + 2 M_0( 1 + \| u_0 \|^2
        \right),
    \end{align*}
    as desired.
 \end{proof}

 Next, we verify
 $({\bf H2})$ for
 $\mathcal{G} ^\varepsilon$ and $\mathcal{G}^0$
 as
 in Proposition \ref{ldpth}.

 \begin{theorem}\label{h2}
    Suppose that $({\bf \Sigma 1})$-$({\bf \Sigma 3})$ hold,
    and $\{ v^\varepsilon \} \subseteq \mathcal{A}_N$ with $ N>0 $.
    Then
    $$
    \lim_{\varepsilon \to 0}
    \left ( \mathcal{G}^\varepsilon
    \left( W + \varepsilon^{-\frac{1}{2}}
    \int_{0}^{\cdot} v^\varepsilon(t) dt
    \right)
    - \mathcal{G}^0
    \left( \int_{0}^{\cdot} v^\varepsilon(t) dt
    \right)
    \right ) = 0
    \ \ \  \text{in probability}
    $$
    in $C([0,T], H) \bigcap L^2(0,T; V)$.
 \end{theorem}

 \begin{proof}
    Let $u_{v^\varepsilon}^\varepsilon
    = \mathcal{G}^\varepsilon
    \left( W + \varepsilon^{-\frac{1}{2}} \int_{0}^{\cdot} v^\varepsilon(t) dt
    \right)$.
    By \eqref{sce} we have
    \begin{align} \label{h2-1}
        & d u_{v^\varepsilon}^\varepsilon (t)
        + (-\triangle)^\alpha u_{v^\varepsilon}^\varepsilon (t) dt
        + f(t, \cdot, u_{v^\varepsilon}^\varepsilon(t), \mathcal{L}_{u^\varepsilon(t)} ) dt
        \\
        & = g(t, \cdot, u_{v^\varepsilon}^\varepsilon(t), \mathcal{L}_{u^\varepsilon(t)} ) dt
        + \sigma(t, u_{v^\varepsilon}^\varepsilon (t), \mathcal{L}_{u^\varepsilon(t)} ) v^\varepsilon(t) dt
        + \sqrt{\varepsilon} \sigma(t, u_{v^\varepsilon}^\varepsilon(t), \mathcal{L}_{u^\varepsilon(t)} ) dW(t)
    \end{align}
    with initial data
    $u_{v^\varepsilon}^\varepsilon(0) = u_0 \in H$.

    Let $u_{v^\varepsilon}
    = \mathcal{G}^0 \left( \int_{0}^{\cdot} v^\varepsilon(t) dt \right)$.
    Then $u_{v^\varepsilon}$ is the unique solution of
    \begin{align} \label{h2-2}
        & \frac{\partial u_{v^\varepsilon} (t)}{\partial t}
        + (-\triangle)^\alpha u_{v^\varepsilon} (t)
        + f(t, \cdot,  u_{v^\varepsilon}(t), \mathcal{L}_{u^0(t)} )
        \nonumber \\
        = & g(t, \cdot, u_{v^\varepsilon}(t), \mathcal{L}_{u^0(t)} )
        + \sigma(t, u_{v^\varepsilon}(t), \mathcal{L}_{u^0(t)}) v^\varepsilon(t)
    \end{align}
    with initial data
    $u_{v^\varepsilon}(0) = u_0 \in H$.
    It  is sufficient to show
    $ u_{v^\varepsilon}^\varepsilon - u_{v^\varepsilon}$
    converges to $0$
    in probability  in
    $C([0,T], H) \bigcap L^2(0,T; V)$
    as $\varepsilon \rightarrow 0$.

    By \eqref{h2-1} and \eqref{h2-2}, we have
    \begin{align} \label{h2-3}
        & d \left( u_{v^\varepsilon}^\varepsilon (t)
        - u_{v^\varepsilon}(t)
        \right)
        + (-\triangle)^\alpha
        \left( u_{v^\varepsilon}^\varepsilon (t)
        - u_{v^\varepsilon}(t)
        \right) dt
        \nonumber \\
        & + \left( f(t, \cdot, u_{v^\varepsilon}^\varepsilon(t), \mathcal{L}_{u^\varepsilon(t)} )
        - f(t, \cdot, u_{v^\varepsilon}(t), \mathcal{L}_{u^0(t)} )
        \right) dt  \nonumber \\
        = & \left( g(t, \cdot, u_{v^\varepsilon}^\varepsilon(t), \mathcal{L}_{u^\varepsilon(t)} )
        - g(t, \cdot, u_{v^\varepsilon}(t), \mathcal{L}_{u^0(t)} )
        \right) dt  \nonumber \\
        & + \left( \sigma(t, u_{v^\varepsilon}^\varepsilon(t), \mathcal{L}_{u^\varepsilon(t)} )
        - \sigma(t, u_{v^\varepsilon}(t), \mathcal{L}_{u^0(t)} )
        \right) v^\varepsilon(t) dt
        + \sqrt{\varepsilon} \sigma(t, u_{v^\varepsilon}^\varepsilon(t), \mathcal{L}_{u^\varepsilon(t)} ) dW(t),
    \end{align}
    with initial data
    $u_{v^\varepsilon}^\varepsilon(0) - u_{v^\varepsilon}(0) = 0$.
    Then applying It\^{o}'s formula to \eqref{h2-3}, we obtain
    \begin{align}\label{h2-4}
        & \| u_{v^\varepsilon}^\varepsilon(t) - u_{v^\varepsilon}(t) \|^2
        + 2 \int_0^t
        \| (-\triangle)^{\frac{\alpha}{2}}
        \left( u_{v^\varepsilon}^\varepsilon (s)
        - u_{v^\varepsilon}(s)
        \right)
        \|^2 ds  \nonumber \\
        & + 2 \int_{0}^{t} \int_{\mathbb{R}^n}
        \left( f(s, x, u_{v^\varepsilon}^\varepsilon(s), \mathcal{L}_{u^\varepsilon(s)} )
        - f(s, x, u_{v^\varepsilon}(s), \mathcal{L}_{u^0(s)} )
        \right)
        \left( u_{v^\varepsilon}^\varepsilon(s) - u_{v^\varepsilon}(s)
        \right) dx ds  \nonumber \\
        = & 2 \int_{0}^{t}
        \left( g(s, \cdot, u_{v^\varepsilon}^\varepsilon(s), \mathcal{L}_{u^\varepsilon(s)} )
        - g(s, \cdot, u_{v^\varepsilon}(s), \mathcal{L}_{u^0(s)} ),
        u_{v^\varepsilon}^\varepsilon(s) - u_{v^\varepsilon}(s)
        \right) ds  \nonumber \\
        & + 2 \int_{0}^{t}
        \left(
        \left( \sigma(s, u_{v^\varepsilon}^\varepsilon(s), \mathcal{L}_{u^\varepsilon(s)} )
        - \sigma(s, u_{v^\varepsilon}(s), \mathcal{L}_{u^0(s)})
        \right) v^\varepsilon(s),
        u_{v^\varepsilon}^\varepsilon(s) - u_{v^\varepsilon}(s)
        \right) ds
        \nonumber \\
        & + \varepsilon \int_{0}^{t}
        \| \sigma(s, u_{v^\varepsilon}^\varepsilon(s), \mathcal{L}_{u^\varepsilon(s)}) \|^2_{L_2(l^2, H)}ds
        \nonumber \\
        & + 2 \sqrt{\varepsilon} \int_{0}^{t}
        \left( u_{v^\varepsilon}^\varepsilon(s) - u_{v^\varepsilon}(s),
        \sigma(s, u_{v^\varepsilon}^\varepsilon(s), \mathcal{L}_{u^\varepsilon(s)} ) dW(s)
        \right).
    \end{align}
    For the third term on the left-hand side of \eqref{h2-4},
    by \eqref{f-2} and \eqref{f-4}, we obtain
    \begin{align} \label{h2-5}
        & 2 \int_{0}^{t} \int_{\mathbb{R}^n}
        \left( f(s, x, u_{v^\varepsilon}^\varepsilon(s), \mathcal{L}_{u^\varepsilon(s)} )
        - f(s, x, u_{v^\varepsilon}(s), \mathcal{L}_{u^0(s)} )
        \right)
        \left( u_{v^\varepsilon}^\varepsilon(s) - u_{v^\varepsilon}(s)
        \right) dx ds  \nonumber \\
        = & 2 \int_{0}^{t} \int_{\mathbb{R}^n}
        \left( f(s, x, u_{v^\varepsilon}^\varepsilon(s), \mathcal{L}_{u^\varepsilon(s)} )
        - f(s, x, u_{v^\varepsilon}(s), \mathcal{L}_{u^\varepsilon(s)} )
        \right)
        \left( u_{v^\varepsilon}^\varepsilon(s) - u_{v^\varepsilon}(s)
        \right) dx ds  \nonumber \\
        & + 2 \int_{0}^{t} \int_{\mathbb{R}^n}
        \left( f(s, x, u_{v^\varepsilon}(s), \mathcal{L}_{u^\varepsilon(s)} )
        - f(s, x, u_{v^\varepsilon}(s), \mathcal{L}_{u^0(s)} )
        \right)
        \left( u_{v^\varepsilon}^\varepsilon(s) - u_{v^\varepsilon}(s)
        \right) dx ds  \nonumber \\
        \geq & - 2
        \int_0^t
        \| \psi_4 (s)\|
        _{  L^\infty(\mathbb{R}^n )}\| u_{v^\varepsilon}^\varepsilon(s)
        - u_{v^\varepsilon}(s) \|^2 ds
        \nonumber \\
        & - 2
        \int_0^t \mathbb{W}_2(\mathcal{L}_{u^\varepsilon(s)}, \mathcal{L}_{u^0(s)})
        \| \psi_3 (s)\|
        \| u_{v^\varepsilon}^\varepsilon(s)
        - u_{v^\varepsilon}(s) \| ds  \nonumber \\
        \geq & -
        \int_0^t
        \left( 2\| \psi_4
        (s)  \|_{  L^\infty(\mathbb{R}^n))}
        + \| \psi_3(s) \|^2
        \right)
        \| u_{v^\varepsilon}^\varepsilon(s)
        - u_{v^\varepsilon}(s) \|^2 ds
        \nonumber \\
        & -
        \int_0^t \left( \mathbb{W}_2(\mathcal{L}_{u^\varepsilon(s)}, \mathcal{L}_{u^0(s)}) \right)^2 ds.
    \end{align}

    For the first term on the right-hand side of \eqref{h2-4},
    by \eqref{g-2}, we obtain
    \begin{align} \label{h2-6}
        & 2 \int_{0}^{t}
        \left( g(s, \cdot, u_{v^\varepsilon}^\varepsilon(s), \mathcal{L}_{u^\varepsilon(s)} )
        - g(s, \cdot, u_{v^\varepsilon}(s), \mathcal{L}_{u^0(s)} ),
        u_{v^\varepsilon}^\varepsilon(s) - u_{v^\varepsilon}(s)
        \right) ds  \nonumber \\
        \leq & 2
        \int_{0}^{t}
        \| \psi_g (s) \|_{  L^\infty(\mathbb{R}^n )}
        \| u_{v^\varepsilon}^\varepsilon(s) - u_{v^\varepsilon}(s)
        \|^2
        ds  \nonumber \\
        & + 2 \int_{0}^{t}
        \| \psi_g(s) \|
        \mathbb{W}_2 (\mathcal{L}_{u^\varepsilon(s)}, \mathcal{L}_{u^0(s)} )
        \|u_{v^\varepsilon}^\varepsilon(s) - u_{v^\varepsilon}(s)\|
        ds  \nonumber \\
        \leq & \int_{0}^{t}
        \left( 2 \| \psi_g
        (s) \|_{  L^\infty(\mathbb{R}^n)}
        + \| \psi_g(s) \|^2
        \right)
        \| u_{v^\varepsilon}^\varepsilon(s) - u_{v^\varepsilon}(s)
        \|^2
        ds  \nonumber \\
        & + \int_{0}^{t}
        \left( \mathbb{W}_2 (\mathcal{L}_{u^\varepsilon(s)}, \mathcal{L}_{u^0(s)} )
        \right)^2
        ds.
    \end{align}
    For the second term on the right-hand side of \eqref{h2-4},
    by \eqref{sigmalip}, we obtain
    \begin{align} \label{h2-7}
        & 2 \int_{0}^{t}
        \left(
        \left( \sigma(s, u_{v^\varepsilon}^\varepsilon(s), \mathcal{L}_{u^\varepsilon(s)} )
        - \sigma(s, u_{v^\varepsilon}(s), \mathcal{L}_{u^0(s)} )
        \right) v^\varepsilon(s),
        u_{v^\varepsilon}^\varepsilon(s) - u_{v^\varepsilon}(s)
        \right) ds
        \nonumber \\
        \leq & 2 \int_{0}^{t}
        \| \sigma(s, u_{v^\varepsilon}^\varepsilon(s), \mathcal{L}_{u^\varepsilon(s)} )
        - \sigma(s, u_{v^\varepsilon}(s), \mathcal{L}_{u^0(s)} )
        \|_{L_2(l^2, H)}
        \| v^\varepsilon(s) \|_{l^2}
        \| u_{v^\varepsilon}^\varepsilon(s) - u_{v^\varepsilon}(s) \|
        ds
        \nonumber \\
        \leq & 2 L_{\sigma}^{\frac{1}{2}}
        \int_{0}^{t}
        \| v^\varepsilon(s) \|_{l^2}
        \| u_{v^\varepsilon}^\varepsilon(s) - u_{v^\varepsilon}(s) \|^2
        ds
        \nonumber \\
        &
        +
        2
        L_{\sigma}^{\frac{1}{2}}
        \int_{0}^{t}
        \mathbb{W}_2 (\mathcal{L}_{u^\varepsilon(s)}, \mathcal{L}_{u^0(s)} )
        \| v^\varepsilon(s) \|_{l^2}
        \| u_{v^\varepsilon}^\varepsilon(s) - u_{v^\varepsilon}(s) \|
        ds
        \nonumber \\
        \leq & 2 L_{\sigma}^{\frac{1}{2}}
        \int_{0}^{t}
        \| v^\varepsilon(s) \|_{l^2}
        (1+ \| v^\varepsilon(s) \|_{l^2})
        \| u_{v^\varepsilon}^\varepsilon(s) - u_{v^\varepsilon}(s) \|^2
        ds
        +
        {\frac 12} L_{\sigma}^{\frac{1}{2}}
        \int_{0}^{t}
        \left ( \mathbb{W}_2 (\mathcal{L}_{u^\varepsilon(s)}, \mathcal{L}_{u^0(s)} )
        \right )^2  ds.
    \end{align}


    From \eqref{h2-4}-\eqref{h2-7}, it follows that
    \begin{align}\label{h2-8}
        & \| u_{v^\varepsilon}^\varepsilon(t) - u_{v^\varepsilon}(t) \|^2
        + 2 \int_0^t
        \| (-\triangle)^{\frac{\alpha}{2}}
        \left( u_{v^\varepsilon}^\varepsilon (s)
        - u_{v^\varepsilon}(s)
        \right)
        \|^2 ds  \nonumber \\
        \leq &  c_1
        \int_0^t \left( 1    + \| v^\varepsilon(s) \|^2_{l^2} \right)
        \| u_{v^\varepsilon}^\varepsilon(s)
        - u_{v^\varepsilon}(s) \|^2 ds
        + c_1
        \int_0^t \mathbb{E} \left[ \| u^\varepsilon(s) - u^0(s) \|^2 \right] ds
        \nonumber \\
        &
        + \varepsilon \int_{0}^{t}
        \| \sigma(s, u_{v^\varepsilon}^\varepsilon(s), \mathcal{L}_{u^\varepsilon(s)}) \|^2_{L_2(l^2, H)}ds
        + 2 \sqrt{\varepsilon} \int_{0}^{t}
        \left( u_{v^\varepsilon}^\varepsilon(s) - u_{v^\varepsilon}(s),
        \sigma(s, u_{v^\varepsilon}^\varepsilon(s), \mathcal{L}_{u^\varepsilon(s)} ) dW(s)
        \right),
    \end{align}
    where
    $c_1=c_1(T)>0$.

    Given $R>0$ and $\varepsilon \in (0,1)$,  define a stopping time
    \begin{align*}
        \tau_R^\varepsilon
        = \inf
        \left\{ t \geq 0: \| u_{v^\varepsilon}^\varepsilon(t) \| \geq R
        \right\} \wedge T.
    \end{align*}
    Then by \eqref{h2-8}, we obtain that for any $t \in [0,T]$,
    \begin{align}\label{h2-9}
        & \sup_{r\in [0,t]}
        \Big(
        \| u_{v^\varepsilon}^\varepsilon(r \wedge \tau_R^\varepsilon) - u_{v^\varepsilon}(r \wedge \tau_R^\varepsilon) \|^2
        + 2 \int_0^{r \wedge \tau_R^\varepsilon}
        \| (-\triangle)^{\frac{\alpha}{2}}
        \left( u_{v^\varepsilon}^\varepsilon (s)
        - u_{v^\varepsilon}(s)
        \right)
        \|^2 ds
        \Big) \nonumber \\
        \leq & c_1
        \int_0^{t \wedge \tau_R^\varepsilon}
        \left( 1
        + \| v^\varepsilon(s) \|^2
        _{l^2} \right)
        \| u_{v^\varepsilon}^\varepsilon(s)
        - u_{v^\varepsilon}(s) \|^2
        ds
        + c_1
        \int_0^{T} \mathbb{E} \left[ \| u^\varepsilon(s) - u^0(s) \|^2 \right] ds
        \nonumber \\
        & + \varepsilon \int_{0}^{T \wedge \tau_R^\varepsilon}
        \| \sigma(s, u_{v^\varepsilon}^\varepsilon(s), \mathcal{L}_{u^\varepsilon(s)}) \|^2_{L_2(l^2, H)} ds
        \nonumber \\
        & + 2 \sqrt{\varepsilon}
        \sup_{r \in [0,T]}
        \left| \int_{0}^{r \wedge \tau_R^\varepsilon}
        \left( u_{v^\varepsilon}^\varepsilon(s) - u_{v^\varepsilon}(s),
        \sigma(s, u_{v^\varepsilon}^\varepsilon(s), \mathcal{L}_{u^\varepsilon(s)} ) dW(s)
        \right)
        \right|.
    \end{align}
    Since $v^\varepsilon\in \mathcal{A}_N$,
    by Gronwall's inequality and \eqref{h2-9}, we obtain
    that for all  $t \in [0,T]$,
    \begin{align}\label{h2-10}
        & \sup_{r \in [0,t]}
        \Big(
        \| u_{v^\varepsilon}^\varepsilon(r \wedge \tau_R^\varepsilon) - u_{v^\varepsilon}(r \wedge \tau_R^\varepsilon) \|^2
        + 2 \int_0^{r \wedge \tau_R^\varepsilon}
        \| (-\triangle)^{\frac{\alpha}{2}}
        \left( u_{v^\varepsilon}^\varepsilon (s)
        - u_{v^\varepsilon}(s)
        \right)
        \|^2 ds
        \Big) \nonumber \\
        \leq &
        c_1 e^{
            c_1 (T +N )  }
        \int_0^T \mathbb{E} \left[ \| u^\varepsilon(s) - u^0(s) \|^2 \right]
        ds
        + \varepsilon
        e^{
            c_1 (T +N )  }
        \int_{0}^{T \wedge \tau_R^\varepsilon}
        \| \sigma(s, u_{v^\varepsilon}^\varepsilon(s), \mathcal{L}_{u^\varepsilon(s)}) \|^2_{L_2(l^2, H)} ds
        \nonumber \\
        & + 2 \sqrt{\varepsilon}
        e^{
            c_1 (T +N )  }
        \sup_{r \in [0,T]}
        \left| \int_{0}^{r \wedge \tau_R^\varepsilon}
        \left( u_{v^\varepsilon}^\varepsilon(s) - u_{v^\varepsilon}(s),
        \sigma(s, u_{v^\varepsilon}^\varepsilon(s), \mathcal{L}_{u^\varepsilon(s)} ) dW(s)
        \right)
        \right|.
    \end{align}

    For the first term on the right-hand side of \eqref{h2-10},
    by Lemma \ref{sdc}, we get
    \begin{align}\label{h2-11-1}
        c_1
        e^{
            c_1 (T +N )  }
        \int_0^T \mathbb{E} \left[ \| u^\varepsilon(s) - u^0(s) \|^2 \right]
        ds
        \leq
        c_1
        e^{
            c_1 (T +N )  }
        T
        C \varepsilon
        \left( 1 + \|u_0\|^2 \right)
        \rightarrow    \ 0,  \    \text{as}\  \varepsilon \to 0.
    \end{align}
    For the second term on the right-hand side of \eqref{h2-10},
    by \eqref{sigma3}
    and Theorem \ref{mvsde-wellposed}, we have
    \begin{align}\label{h2-11}
        & \varepsilon
        e^{
            c_1 (T +N )  }
        \int_{0}^{T \wedge \tau_R^\varepsilon}
        \| \sigma(s, u_{v^\varepsilon}^\varepsilon(s), \mathcal{L}_{u^\varepsilon(s)}) \|^2_{L_2(l^2, H)} ds
        \nonumber \\
        \leq & \varepsilon
        e^{
            c_1 (T +N )  }
        \int_{0}^{T \wedge \tau_R^\varepsilon}
        M_{\sigma, T}
        \left( 1 + \| u_{v^\varepsilon}^\varepsilon(s)\|^2
        +
        \mathbb{E}\left [
        \| u^\varepsilon(s) \|^2
        \right ]
        \right)
        ds \nonumber \\
        \leq & \varepsilon
        e^{
            c_1 (T +N )  }
        T M_{\sigma, T} \left( 1 +  R^2
        + M_1
        +M_1 \|u_0\|^2 \right)
        \nonumber \\
        \rightarrow & \ 0,  \ \ \   \text{as}\  \varepsilon \to 0, \ \ \   \mathbb{P} \text{-almost surely}.
    \end{align}
    For the third term on the right-hand side of \eqref{h2-10},
    by Doob's maximal inequality, \eqref{sigma3}, \eqref{dceue}
    and   Theorem \ref{mvsde-wellposed}, we have
    \begin{align*}
        &
        \mathbb{E}
        \left[
        \varepsilon
        \sup_{r \in [0,T]}
        \left|
        \int_{0}^{r \wedge \tau_R^\varepsilon}
        \left( u_{v^\varepsilon}^\varepsilon(s) - u_{v^\varepsilon}(s),
        \sigma(s, u_{v^\varepsilon}^\varepsilon(s), \mathcal{L}_{u^\varepsilon(s)} ) dW(s)
        \right)
        \right|^2
        \right] \nonumber \\
        \leq & 4 \varepsilon
        \mathbb{E}
        \left[ \int_{0}^{T \wedge \tau_R^\varepsilon}
        \| u_{v^\varepsilon}^\varepsilon(s) - u_{v^\varepsilon}(s)
        \|^2
        \| \sigma(s, u_{v^\varepsilon}^\varepsilon(s), \mathcal{L}_{u^\varepsilon(s)} ) \|_{L_2(l^2, H)}^2
        ds
        \right]  \nonumber \\
        \leq & 4 \varepsilon \left(R + M_2 \right)^2
        \mathbb{E}
        \left[ \int_{0}^{T \wedge \tau_R^\varepsilon}
        \| \sigma(s, u_{v^\varepsilon}^\varepsilon(s), \mathcal{L}_{u^\varepsilon(s)} ) \|_{L_2(l^2, H)}^2
        ds
        \right]  \nonumber \\
        \leq & 4 \varepsilon
        \left(R +M_2 \right)^2
        T M_{\sigma, T}
        \left( 1 +   R^2
        +M_1+M_1
        \|u_0\|^2
        \right) \nonumber \\
        \rightarrow & \  0, \ \ \  \text{as} \ \varepsilon \to 0,
    \end{align*}
    which implies that
    \begin{align} \label{h2-12}
        \sqrt{\varepsilon}
        \sup_{r \in [0,T]}
        \left| \int_{0}^{r \wedge \tau_R^\varepsilon}
        \left( u_{v^\varepsilon}^\varepsilon(s) - u_{v^\varepsilon}(s),
        \sigma(s, u_{v^\varepsilon}^\varepsilon(s), \mathcal{L}_{u^\varepsilon(s)} ) dW(s)
        \right)
        \right|
        \rightarrow  \  0   \ \   \text{in probability}.
    \end{align}

    It then follows from \eqref{h2-10}-\eqref{h2-12}  that
    \begin{align}\label{h2-13}
        & \sup_{t\in [0,T]}
        \left(
        \| u_{v^\varepsilon}^\varepsilon(t \wedge \tau_R^\varepsilon)
        - u_{v^\varepsilon}(t \wedge \tau_R^\varepsilon) \|^2
        \right)  \nonumber \\
        & + \int_0^{T \wedge \tau_R^\varepsilon}
        \| (-\triangle)^{\frac{\alpha}{2}}
        \left( u_{v^\varepsilon}^\varepsilon (t)
        - u_{v^\varepsilon}(t)
        \right)
        \|^2 dt
        \rightarrow 0 \ \ \ \ \   \text{in\ probability}.
    \end{align}

    By Theorem \ref{mvsde-wellposed} and Lemma \ref{pre-h2}, we have for any $\varepsilon \in [0,1]$,
    \begin{align}\label{h2-14}
        P \left( \tau_R^\varepsilon < T \right)
        & \leq  P
        \left( \sup_{t \in [0,T]}
        \| u_{v^\varepsilon}^\varepsilon(t) \| \geq R
        \right)
        \leq \frac{1}{R^2}
        \mathbb{E}
        \left[ \sup_{t \in [0,T]}
        \| u_{v^\varepsilon}^\varepsilon(t) \|^2
        \right]
        \leq \frac{M_3}{R^2} .
    \end{align}
    From \eqref{h2-14}, it follows that for any $\delta > 0$,
    \begin{align} \label{h2-15}
        & P
        \left( \sup_{t \in [0,T]}
        \left\| u_{v^\varepsilon}^\varepsilon(t)
        - u_{v^\varepsilon}(t)
        \right\| > \delta
        \right)  \nonumber \\
        \leq & P
        \left( \sup_{t \in [0,T]}
        \left\| u_{v^\varepsilon}^\varepsilon(t)
        - u_{v^\varepsilon}(t)
        \right\|  >  \delta,
        \tau_R^{\varepsilon} = T
        \right)
        + P
        \left( \sup_{t \in [0,T]}
        \left\| u_{v^\varepsilon}^\varepsilon(t)
        - u_{v^\varepsilon}(t)
        \right\|  >  \delta,
        \tau_R^{\varepsilon} < T
        \right)  \nonumber \\
        \leq & P
        \left( \sup_{t \in [0,T]}
        \left\| u_{v^\varepsilon}^\varepsilon(t \wedge \tau_R^{\varepsilon})
        - u_{v^\varepsilon}(t \wedge \tau_R^{\varepsilon})
        \right\|  >  \delta
        \right)
        + \frac{M_3}{R^2} .
    \end{align}
    For \eqref{h2-15}, first letting $\varepsilon \rightarrow 0$, and then $R \rightarrow +\infty$,
    we obtain by \eqref{h2-13} that
    \begin{align}\label{h2-16}
        \lim_{\varepsilon \rightarrow 0}
        \left( u_{v^\varepsilon}^\varepsilon
        - u_{v^\varepsilon}
        \right) = 0
        \ \   \text{in probability
            in } \  C([0,T],H).
    \end{align}
    From \eqref{h2-14}, we obtain that for any $\delta > 0$,
    \begin{align} \label{h2-17}
        & P
        \left( \int_0^T
        \| u_{v^\varepsilon}^\varepsilon(t)
        - u_{v^\varepsilon}(t)
        \|_V^2  > \delta
        \right)  \nonumber \\
        \leq & P
        \left( \int_0^T
        \| u_{v^\varepsilon}^\varepsilon(t)
        - u_{v^\varepsilon}(t)
        \|_V^2  > \delta,
        \tau_R^{\varepsilon} = T
        \right) \nonumber \\
        & + P
        \left( \int_0^T
        \| u_{v^\varepsilon}^\varepsilon(t)
        - u_{v^\varepsilon}(t)
        \|_V^2  > \delta,
        \tau_R^{\varepsilon} < T
        \right)  \nonumber \\
        \leq & P
        \left( \int_0^{T \wedge \tau_R^\varepsilon}
        \| u_{v^\varepsilon}^\varepsilon(t)
        - u_{v^\varepsilon}(t)
        \|_V^2  > \delta
        \right)
        + \frac{M_3}{R^2} .
    \end{align}
    First letting $\varepsilon \rightarrow 0$, and then $R \rightarrow +\infty$,
    we obtain by \eqref{h2-13} that
    \begin{align} \label{h2-18}
        \lim_{\varepsilon \rightarrow 0}
        \left( u_{v^\varepsilon}^\varepsilon - u_{v^\varepsilon}
        \right)
        =0 \ \    \text{in
            probability
            in } \  L^2(0, T; V).
    \end{align}
    By \eqref{h2-16} and \eqref{h2-18},
    we see that
    $
    u_{v^\varepsilon}^\varepsilon - u_{v^\varepsilon}
    \rightarrow 0
    $ in probability
    in   $C([0,T], H) \bigcap L^2(0,T; V)$
    as $\varepsilon \rightarrow 0$,
    which concludes the proof.
 \end{proof}

 Now, we
 are ready to prove
 the LDP of \eqref{mvsde-2}-\eqref{mvsde-2 IV}
 in
 $C([0,T], H) \bigcap L^2(0,T; V)$.

 {\bf Proof of Theorem \ref{main} :}

 It follows from Theorem \ref{h1} and Theorem \ref{h2} that
 the solution family $\{ u^\varepsilon \}$ of \eqref{mvsde-2}-\eqref{mvsde-2 IV}
 satisfies
 the conditions $({\bf H1})$ and $({\bf H2})$ in Proposition \ref{ldpth}
 in the space
 $ C([0,T], H)
 \bigcap  L^2(0,T; V) $, and hence
 it satisfies the LDP
 in that space.

 \subsection{Large deviations
    with  strongly
    dissipative drift  }

 In this subsection, we   investigate the LDP of
 \eqref{mvsde-2}-\eqref{mvsde-2 IV} in
 the  product space
 $$
 C([0,T], H)
 \bigcap
 L^2(0,T; V)
 \bigcap
 L^p(0,T; L^p(\mathbb{R}^n ) ).
 $$
 To that end, we further assume
 the  nonlinear drift   $f$ in
 \eqref{mvsde-1} satisfies the following dissipative condition:
 for any $t,u_1, u_2 \in \mathbb{R}$,
 $x \in \mathbb{R}^n$
 and $\mu \in \mathcal{P}_2(H)$,
 \begin{align} \label{f-5}
    \left(
    f(t, x, u_1, \mu) - f(t, x, u_2, \mu)
    \right)
    \left( u_1 - u_2
    \right)
    \geq \lambda_4 |u_1 -u_2|^p
    - \psi_5(t,x) |u_1 - u_2|^2,
 \end{align}
 where the constant $\lambda_4>0$
 and $\psi_5 \in L_{loc}^\infty(\mathbb{R}, L^\infty(\mathbb{R}^n) )$.

 \begin{lemma} \label{h1-sd}
    Suppose that $({\bf \Sigma 1})$-$({\bf \Sigma 3})$ and \eqref{f-5} hold,
    $u_0 \in H$, and
    $v, v_i \in L^2(0, T; l^2)$ for $i \in \mathbb{N}$.
    If $v_i \rightarrow v$ weakly in $L^2(0, T; l^2)$,
    then
    $\mathcal{G}^0 \big( \int_0^\cdot v_i(s) ds \big)$
    converges to
    $\mathcal{G}^0 \big( \int_0^\cdot v(s) ds \big)$
    in the space
    $$C([0,T], H)
    \bigcap
    L^2(0,T; V)
    \bigcap
    L^p(0,T; L^p(\mathbb{R}^n ) ).
    $$
 \end{lemma}

 \begin{proof}
    By \eqref{sc-23} and \eqref{f-5}, we obtain
    \begin{align}\label{sd-1}
        & \frac{d}{dt} \| u_{v_i}(t) - u_v(t) \|^2
        + 2 \| (-\triangle)^{\frac{\alpha}{2}} \left( u_{v_i}(t) - u_v(t) \right) \|^2
        + 2 \lambda_4 \| u_{v_i}(t) - u_v(t)\|_{L^p(\mathbb{R}^n)}^p
        \nonumber \\
        \leq & 2 \| \psi_5(t) \|_{L^\infty(\mathbb{R}^n)}
        \| u_{v_i}(t) - u_v(t) \|^2
        \nonumber \\
        & + 2
        \left( g(t, \cdot, u_{v_i}(t), \mathcal{L}_{u^0(t)} )
        - g(t, \cdot, u_v(t), \mathcal{L}_{u^0(t)} ), \
        u_{v_i}(t) - u_v(t)
        \right)  \nonumber \\
        & + 2
        \left(
        \sigma(t, u_{v_i}(t), \mathcal{L}_{u^0(t)} ) v_i(t)
        - \sigma(t, u_v(t), \mathcal{L}_{u^0(t)} ) v(t), \
        u_{v_i}(t) - u_v(t)
        \right).
    \end{align}
    Then together with \eqref{sd-1} and the argument of \eqref{sc-28},
    the desired result follows immediately.
 \end{proof}

 \begin{lemma}\label{h2-sd}
    Suppose that $({\bf \Sigma 1})$-$({\bf \Sigma 3})$ and \eqref{f-5} hold,
    and $\{ v^\varepsilon \} \subseteq \mathcal{A}_N$ with $ N < \infty$.
    Then
    $$
    \lim_{\varepsilon \to 0}
    \left ( \mathcal{G}^\varepsilon
    \left( W + \varepsilon^{-\frac{1}{2}}
    \int_{0}^{\cdot} v^\varepsilon(t) dt
    \right)
    - \mathcal{G}^0
    \left( \int_{0}^{\cdot} v^\varepsilon(t) dt
    \right)
    \right ) = 0
    \ \ \  \text{in probability}
    $$
    in the space
    $C([0,T], H) \bigcap L^2(0,T; V) \bigcap L^p(0,T; L^p(\mathbb{R}^n ) )$.
 \end{lemma}

 \begin{proof}
    It follows from \eqref{h2-4} and \eqref{f-5} that
    \begin{align*}\label{sd-2}
        & \| u_{v^\varepsilon}^\varepsilon(t) - u_{v^\varepsilon}(t) \|^2
        + 2 \int_0^t \| (-\triangle)^{\frac{\alpha}{2}}
        \left( u_{v^\varepsilon}^\varepsilon (s)
        - u_{v^\varepsilon}(s)
        \right)
        \|^2 ds  \nonumber \\
        & + 2 \lambda_4 \int_0^t
        \| u_{v^\varepsilon}^\varepsilon (s)
        - u_{v^\varepsilon}(s) \|_{L^p(\mathbb{R}^n)}^p ds
        \nonumber \\
        \leq & 2 \int_{0}^{t} \| \psi_5(s) \|_{L^\infty(\mathbb{R}^n)}
        \| u_{v^\varepsilon}^\varepsilon(s) - u_v(s) \|^2 ds
        \nonumber \\
        & + 2 \int_{0}^{t}
        \left( g(s, \cdot, u_{v^\varepsilon}^\varepsilon(s), \mathcal{L}_{u^\varepsilon(s)} )
        - g(s, \cdot, u_{v^\varepsilon}(s), \mathcal{L}_{u^0(s)} ),\
        u_{v^\varepsilon}^\varepsilon(s) - u_{v^\varepsilon}(s)
        \right) ds  \nonumber \\
        & + 2 \int_{0}^{t}
        \left(
        \left( \sigma(s, u_{v^\varepsilon}^\varepsilon(s), \mathcal{L}_{u^\varepsilon(s)} )
        - \sigma(s, u_{v^\varepsilon}(s), \mathcal{L}_{u^0(s)} )
        \right) v^\varepsilon(s),\
        u_{v^\varepsilon}^\varepsilon(s) - u_{v^\varepsilon}(s)
        \right) ds  \nonumber \\
        & + \varepsilon \int_{0}^{t}
        \| \sigma(s, u_{v^\varepsilon}^\varepsilon(s), \mathcal{L}_{u^\varepsilon(s)} ) \|^2_{L_2(l^2, H)}ds
        \nonumber \\
        & + 2 \sqrt{\varepsilon} \int_{0}^{t}
        \left( u_{v^\varepsilon}^\varepsilon(s) - u_{v^\varepsilon}(s),\
        \sigma(s, u_{v^\varepsilon}^\varepsilon(s), \mathcal{L}_{u^\varepsilon(s)} ) dW(s)
        \right).
    \end{align*}
    The rest of the proof is
    is similar to   Theorem \ref{h2}. The details are omitted.
 \end{proof}

 Now, we present the LDP
 of the stochastic equation  in
 $ C([0,T], H) \bigcap L^2(0,T; V) \bigcap L^p(0,T; L^p(\mathbb{R}^n ) )
 $.

 \begin{theorem}\label{main-sd}
    Suppose that $({\bf \Sigma 1})$-$({\bf \Sigma 3})$ and \eqref{f-5} hold,
    and $u^\varepsilon$ is the solution of \eqref{mvsde-2}-\eqref{mvsde-2 IV}.
    Then as $\varepsilon \rightarrow 0$, the family $\{ u^\varepsilon \}$
    satisfies the
    LDP in the space
    $$ C([0,T], H)  \bigcap  L^2(0,T; V) \bigcap L^p(0,T; L^p(\mathbb{R}^n ) ) $$
    with   good rate function   given by \eqref{rate}.
 \end{theorem}

 \begin{proof}
    By Lemma \ref{h1-sd}
    and  Lemma \ref{h2-sd},
    we see  that
    the conditions of Proposition \ref{ldpth} are fulfilled
    in
    $  C([0,T], H)  \bigcap  L^2(0,T; V) \bigcap L^p(0,T; L^p(\mathbb{R}^n ) ) $,
    and hence the family
    $\{ u^\varepsilon \}$
    satisfies the
    LDP in this case.
 \end{proof}

\section*{Acknowledgements}

 The work is partially supported by the NNSF of China (11471190, 11971260).

\section*{Declarations}

{\bf Conflict of interest}
 The authors have no relevant financial or non-financial interests to disclose.

\end{document}